\documentclass[11pt]{amsart}

\usepackage[colorlinks=true, pdfstartview=FitV, linkcolor=blue, citecolor=blue, urlcolor=blue, breaklinks=true]{hyperref}
\usepackage{amsmath,amsfonts,amssymb,amsthm,amscd,comment,paralist,etoolbox,mathtools,enumitem,mathdots}
\usepackage[usenames,dvipsnames]{xcolor}
\usepackage{pxfonts}
\usepackage{mathrsfs} 
\usepackage{tabu}

\usepackage[all]{xy}
\usepackage{tikz}
\usetikzlibrary{arrows}

\makeatletter
\def\namedlabel#1#2{\begingroup
    #2%
    \def\@currentlabel{#2}%
    \phantomsection\label{#1}\endgroup
}

\newcommand{\pushright}[1]{\ifmeasuring@#1\else\omit\hfill$\displaystyle#1$\fi\ignorespaces}
\newcommand{\pushleft}[1]{\ifmeasuring@#1\else\omit$\displaystyle#1$\hfill\fi\ignorespaces}
\makeatother

\newcommand\C{\mathbb{C}}
\newcommand\Z{\mathbb{Z}}
\newcommand\Q{\mathbb{Q}}
\newcommand\R{\mathbb{R}}
\newcommand\N{\mathbb{N}}
\newcommand\F{\mathbb{F}}

\newcommand\bS{\mathbb{S}}
\newcommand\kk{\Bbbk}
\newcommand\blambda{{\boldsymbol{\lambda}}}

\newcommand\fh{\mathfrak{h}}

\newcommand\fP{\mathfrak{P}}
\newcommand\fS{\mathfrak{S}}

\newcommand\partition{\mathcal{P}}

\newcommand\cH{\mathcal{H}}
\newcommand\cB{\mathcal{B}}

\newcommand\ba{\mathbf{a}}
\newcommand\bF{\mathbf{F}}
\newcommand\bb{\mathbf{b}}
\newcommand\bc{\mathbf{c}}

\newcommand\bv{\mathbf{v}}
\newcommand\one{\mathbf{1}}

\newcommand\op{\mathrm{op}}

\newcommand\tr{\mathrm{tr}}

\newcommand\rR{\mathrm{R}}
\newcommand\rL{\mathrm{L}}

\newcommand\sP{\mathsf{P}}
\newcommand\sQ{\mathsf{Q}}

\newcommand{\md}{\textup{-mod}}
\newcommand{\pmd}{\textup{-pmod}}
\newcommand{\bimod}{\textup{-bimod}}
\newcommand{\proj}{\textup{proj}}
\newcommand\Cl{\textup{Cl}}

\newcommand\pr[1]{\prescript{r}{}{#1}}

\newcommand\pl[1]{\prescript{\ell}{}{#1}}

\newcommand\psiB{\psi}

\newcommand\ts{\textstyle}

\newcommand\redcircle[1]{\filldraw[fill=white, draw=red] #1 circle (3pt)}
\newcommand\bluedot[1]{\filldraw[blue] #1 circle (2pt)}

\DeclareMathOperator{\Hom}{Hom}
\DeclareMathOperator{\HOM}{HOM}
\DeclareMathOperator{\dist}{dist}
\DeclareMathOperator{\End}{End}
\DeclareMathOperator{\END}{END}
\DeclareMathOperator{\Fun}{Fun}    
\DeclareMathOperator{\Span}{Span}

\DeclareMathOperator{\Res}{Res}
\DeclareMathOperator{\Ind}{Ind}

\DeclareMathOperator{\grdim}{grdim}
\DeclareMathOperator{\id}{id}

\DeclareMathOperator{\shift}{shift}
\DeclareMathOperator{\Ann}{Ann}

\newtheorem{theo}{Theorem}[section]
\newtheorem{prop}[theo]{Proposition}
\newtheorem{lem}[theo]{Lemma}
\newtheorem{cor}[theo]{Corollary}
\newtheorem{conj}[theo]{Conjecture}

\theoremstyle{definition}
\newtheorem{defin}[theo]{Definition}
\newtheorem{rem}[theo]{Remark}
\newtheorem{eg}[theo]{Example}

\numberwithin{equation}{section}
\allowdisplaybreaks

\newtoggle{comments}
\newtoggle{details}
\newtoggle{prelimnote}
\newtoggle{detailsnote}

\toggletrue{detailsnote}   

\iftoggle{comments}{%
  \newcommand{\comments}[1]{
    \ \\
    {\color{red}
      \textbf{Comments:} #1
    }
    \\
    }
}{%
  \newcommand{\comments}[1]{}
}

\iftoggle{details}{%
  \newcommand{\details}[1]{
      \ \\
      {\color{OliveGreen}
        \textbf{Details:} #1
      }
      \\
  }
}{%
  \newcommand{\details}[1]{}
}

\iftoggle{prelimnote}{%
  \newcommand{\prelim}{\textsc{Preliminary version} \bigskip}
}{%
  \newcommand{\prelim}{}
}

\begin{document}
%

\title{A general approach to Heisenberg categorification via wreath product algebras}

\author{Daniele Rosso}
\address{D.~Rosso: Department of Mathematics, University of California Riverside}
\urladdr{\url{http://math.ucr.edu/~rosso}}
\email{daniele.rosso@ucr.edu}

\author{Alistair Savage}
\address{A.~Savage: Department of Mathematics and Statistics, University of Ottawa}
\urladdr{\url{http://alistairsavage.ca}}
\email{alistair.savage@uottawa.ca}
\thanks{The second author was supported by a Discovery Grant from the Natural Sciences and Engineering Research Council of Canada.}

\begin{abstract}
  We associate a monoidal category $\cH_B$, defined in terms of planar diagrams, to any graded Frobenius superalgebra $B$.  This category acts naturally on modules over the wreath product algebras associated to $B$.  To $B$ we also associate a (quantum) lattice Heisenberg algebra $\fh_B$.  We show that, provided $B$ is not concentrated in degree zero, the Grothendieck group of $\cH_B$ is isomorphic, as an algebra, to $\fh_B$.  For specific choices of Frobenius algebra $B$, we recover existing results, including those of Khovanov and Cautis--Licata.  We also prove that certain morphism spaces in the category $\cH_B$ contain generalizations of the degenerate affine Hecke algebra.  Specializing $B$, this proves an open conjecture of Cautis--Licata.
\end{abstract}


\prelim

\maketitle
\thispagestyle{empty}

\tableofcontents

%
\section{Introduction}
%

The relationship between the representation theory of the symmetric group, the ring of symmetric functions, and the Heisenberg algebra goes back to work of Geissinger in the 1970s (see \cite{Gei77}).  Using the more modern concept of \emph{categorification}, Geissinger's work was substantially strengthened by Khovanov in \cite{Kho14}.  More precisely, Khovanov introduced a category, defined in terms of planar diagrams, whose Grothendieck group is conjecturally isomorphic to the Heisenberg algebra.  This work has inspired a great deal of research into Heisenberg categorification, including $q$-deformations (\cite{LS13}), relations to the geometry of Hilbert schemes (\cite{CL12}), categorification of vertex operators (\cite{CL11}), and categorified braid group actions (\cite{CLS14}).

Many of the constructions in the field of Heisenberg categorification use variations of Khovanov's original Heisenberg category.  For example, the work of \cite{CL12} involves two Heisenberg categories associated to a finite subgroup of $\mathrm{SL}(2,\C)$.  In this case, a natural grading on the morphism spaces allows the authors to prove that their category does in fact categorify the Heisenberg algebra (as opposed to the conjectural status of Khovanov's categorification).

The goal of the current paper is to unify, simplify, and generalize the existing graphical Heisenberg categorifications.  We associate a monoidal category $\cH_B'$ to any $\N$-graded Frobenius superalgebra $B$ over an algebraically closed field $\F$ of characteristic zero.  The objects of $\cH_B'$ are generated by two objects $\sP$ and $\sQ$, which we depict as upward and downward oriented arrows.  Morphisms between sequences (tensor products) of $\sP$'s and $\sQ$'s are given by planar diagrams involving oriented strands and dots labeled by elements of $B$.  These diagrams are considered up to super isotopy and modulo local relations (see Section~\ref{sec:cat-def}).  We then take the idempotent completion $\cH_B$ of $\cH_B'$.  The category $\cH_B$ acts naturally on categories of modules over the wreath product algebras $B^{\otimes n} \rtimes \F[S_n]$.  Note that we do \emph{not} require $B$ to be a symmetric algebra.  In fact, the Nakayama automorphism of $B$ plays an important role in the graphical calculus (see \eqref{eq:dotsup-a}--\eqref{eq:dotsdown-b}).   We also do not require the trace map of $B$ to be of even parity.  Trace maps of odd parity lead to interesting ``super isotopy invariance'' in the graphical category (see Section~\ref{sec:cat-def}).

Let $K_0(B)$ denote the Grothendieck group of the category of finitely-generated projective $B$-modules.  There is a natural bilinear form on $K_0(B)$ induced by the graded dimension of the space of homomorphisms between modules (see \eqref{eq:bilinear-form}).  Thus one can form the \emph{(quantum) lattice Heisenberg algebra} $\fh_B$ associated to $B$ (see Definition~\ref{def:hB} and Remark~\ref{rem:lattice-Heisenberg}).  The main result of the current paper (Theorem~\ref{theo:main-iso}) is that, if $B$ is not concentrated in $\N$-degree zero, then the Grothendieck group of $\cH_B$ is isomorphic to the Heisenberg algebra $\fh_B$.  This can be viewed as a strengthening of the results of \cite[\S7]{RS15a} (see \cite[Rem.~7.7]{RS15a}).

The morphism spaces appearing in the category $\cH_B'$ are interesting mathematical objects deserving of further study.  For instance, we show (Proposition~\ref{prop:chim-properties}) that the endomorphism algebra of the object $\sP^n$ contains a generalization of the degenerate affine Hecke algebra depending on the graded Frobenius superalgebra $B$.  For a certain choice of $B$, this proves a conjecture of Cautis--Licata stated in \cite{CL12} (see Corollary~\ref{cor:CL-conjecture}).  We also obtain a natural generalization to wreath product algebras of the Jucys--Murphy elements in the group algebra of the symmetric group (see Remark~\ref{rem:Jucys-Murphy}).

The results of the current paper recover as special cases the Heisenberg categorification results mentioned earlier in the introduction.  Firstly, the category $\cH_B$ reduces to Khovanov's original category when $B=\F$.  Secondly, certain choices of $B$ recover both the main category of study in \cite{CL12} as well as the ``alternate category'' used there to prove some of their results.  Other choices of $B$ recover the categories considered in \cite[\S 6]{CS14} and \cite{HS15}, related to the categorification of twisted presentations of Heisenberg algebras.  However, even in these cases, the approach of the current paper offers several benefits:
\begin{asparaitem}
  \item The presentations of the Heisenberg algebra best suited to categorification are often deduced on a case-by-case basis from the standard presentation via generating function techniques.  See, for example, \cite[\S1]{Kho14} and \cite[Lem.~1]{CL12}.  However, in Proposition~\ref{prop:hB-presentation}, we are able to determine such presentations in the general setting by adopting the Heisenberg double point of view.  The general presentations specialize to those of \cite{Kho14,CL12} without the need to perform a generating function computation in each case.  Other specializations yield corrected versions of the relations in \cite[Prop.~5.1]{CS14} and \cite[Prop.~1]{HS15} (see Remark~\ref{rem:missing-relations}).

  \item Even for the choices of $B$ made in \cite{CL12,HS15}, we are able to prove directly that $\cH_B$ categorifies $\fh_B$, without resorting to the ``alternate category'' used in the proofs in those papers. We also fill in a gap in the arguments in \cite{CS14} and \cite{HS15} (see Remark~\ref{rem:missing-isoms}).

  \item The choice of the finite subgroup $\Z_2$ of $\mathrm{SL}(2,\C)$ requires special treatment in \cite{CL12}.  This is not necessary in the approach of the current paper.
\end{asparaitem}

We expect that many of the existing constructions in the literature related to Heisenberg categorification can be extended to the setting of the current paper.  Some such possible directions include the following:
\begin{asparaenum}
  \item While the current paper recovers the algebraic results of \cite{CL12}, that paper also discusses actions on derived categories of coherent sheaves on certain Hilbert schemes.  We expect that, at least for some choices of $B$, there should exist analogous geometric actions of the categories $\cH_B$.  This would result in a strengthening of the results of \cite{Kru15} (see \cite[\S3.4]{Kru15}).

  \item We expect that the categorical braid group actions of \cite{CLS14} can be generalized to the setting of the current paper.

  \item In \cite{CLLS15}, the authors identify the trace (or zeroth Hochschild homology) of Khovanov's Heisenberg category with a quotient of the $W$-algebra $W_{1+\infty}$.  In \cite[\S.1.1]{CLLS15}, the authors predict that their results should have analogues for more general Heisenberg categories.

  \item In \cite{CL11}, the authors construct 2-representations of quantum affine algebras from 2\nobreakdash-\hspace{0pt}representations of Heisenberg algebras using categorical vertex operators.  It is reasonable to expect that some of these results can be generalized.

  \item It is natural to ask if the construction of the current paper can be $q$-deformed in a manner analogous to the $q$-deformation given in \cite{LS13} of Khovanov's original construction \cite{Kho14}.

  \item In \cite[\S4]{Web13}, Webster develops a diagrammatic categorification of tensor products of representations of quantum groups.  Starting from the categorification in the current paper, methods similar to those of Webster should enable one to categorify tensor products of representations of Heisenberg algebras, resulting in a categorification of higher level representations (with the Fock space involved in the current paper being of level one).
\end{asparaenum}

Note also that wreath products of finite subgroups of $\mathrm{SL}_2(\C)$ (or a twisted versions thereof) have appeared in algebraic constructions of the Fock space of Heisenberg algebras.  See, for example, \cite{FJW00,FJW02,JW02} and references therein.

The organization of the current paper is as follows.  In Section~\ref{sec:superalebras}, we review basic facts about graded superalgebras and their modules and establish notation.  Then, in Section~\ref{sec:wreath}, we introduce Frobenius algebras and wreath product algebras.  In Section~\ref{sec:wreath-modules}, we classify the simple and indecomposable projective modules for wreath product algebras.  We introduce the Heisenberg algebra $\fh_B$ associated to a Frobenius algebra $B$ in Section~\ref{sec:Heis-alg}.  In particular, we deduce there, in Proposition~\ref{prop:hB-presentation}, the presentations best suited to categorification.  In Section~\ref{sec:cat-def}, we define the graphical category $\cH_B$ and, in Section~\ref{sec:actions}, we define a natural action of $\cH_B$ on categories of modules over wreath product algebras.  We examine the morphism spaces of $\cH_B$ in Section~\ref{sec:morphism-spaces}.  In particular, we prove there a generalization of a conjecture of Cautis--Licata (see Proposition~\ref{prop:chim-properties} and Corollary~\ref{cor:CL-conjecture}).   In Section~\ref{sec:key-isoms} we prove the key isomorphisms in $\cH_B$ that descend to the relations of the Heisenberg algebra $\fh_B$ (Theorem~\ref{theo:functor-isos}).  Finally, we prove our main result, that $\cH_B$ categorifies $\fh_B$ (Theorem~\ref{theo:main-iso}) in Section~\ref{sec:main-result}.

\medskip

\paragraph{\textbf{Notation}} In computations, we will often write equation numbers above equals signs to indicate the relations that are being used at each step.

\iftoggle{detailsnote}{
\medskip

\paragraph{\textbf{Note on the arXiv version}} For the interested reader, the tex file of the arXiv version of this paper includes hidden details of some straightforward computations and arguments that are omitted in the pdf file.  These details can be displayed by switching the \texttt{details} toggle to true in the tex file and recompiling.
}{}

\subsection*{Acknowledgements}

The authors would like to thank A.~Licata and S.~Cautis for sharing their preliminary notes, based on conversations with M.~Khovanov, regarding Heisenberg categorification depending on symmetric Frobenius algebras.  They would also like to thank J.~Brundan, A.~Licata, A.~Ram, J.~Sussan, and O.~Yacobi for useful conversations.

%
\section{Modules for superalgebras} \label{sec:superalebras}
%

We let $\N$ and $\N_+$ denote the set of nonnegative and positive integers respectively.  We let $\Z_2=\Z/2\Z$ be the ring of integers mod $2$, and let $\F$ be a field (often assumed to be algebraically closed of characteristic zero).  By a slight abuse of terminology, we will use the terms \emph{module} and \emph{representation} interchangeably.  We call an $\F$-algebra $R$ a graded (more precisely, $\N$-graded) superalgebra over $\F$ if
\[ \ts
  R \cong \bigoplus_{i\in\N,\, \epsilon\in\Z_2} R_{i,\epsilon} \quad \text{ with } \quad R_{i,\epsilon} R_{j,\tau} \subseteq R_{i+j,\epsilon+\tau}, \quad i,j\in\N,~\epsilon,\tau\in\Z_2.
\]
If $R$ is a graded superalgebra, a graded $R$-supermodule $M$ is analogously a $\Z\times\Z_2$-graded vector space over $\F$ such that
\[
  R_{i,\epsilon} M_{j,\tau} \subseteq M_{i+j,\epsilon+\tau} \quad \text{for all } i \in \N,\ j \in \Z,\ \epsilon,\tau \in \Z_2.
\]
For the remainder of this paper, we will use the term \emph{algebra} to mean graded superalgebra and \emph{module} to mean graded supermodule.  Furthermore, all modules are left modules unless otherwise specified.  If $v$ is a homogeneous element in a $\Z$-graded (resp.\ $\Z_2$-graded) vector space, we will denote by $|v|$ (resp.\ $\bar v$) its degree.  Whenever we write an expression involving degrees of elements, we will implicitly assume that such elements are homogeneous.

For $M$, $N$ two $\Z\times\Z_2$-graded vector spaces over $\F$, we define a $\Z\times\Z_2$-grading on the space $\HOM_\F(M,N)$ of all $\F$-linear maps by setting $\HOM_\F(M,N)_{i,\epsilon}$ to be the subspace of all homogeneous maps of degree $(i,\epsilon)$. That is,
\[
  \HOM_\F(M,N)_{i,\epsilon} := \{\alpha \in \HOM_\F(M,N)\ |\  \alpha(M_{j,\tau}) \subseteq N_{i+j,\epsilon+\tau} \ \forall\ j \in \Z,\ \tau \in \Z_2\}.
\]
If $M$, $N$ are two $R$-modules, then we define the $\Z \times \Z_2$-graded $\F$-vector space
\begin{gather*}
  \HOM_R(M,N) := \bigoplus_{i \in \Z,\, \epsilon \in \Z_2} \HOM_R(M,N)_{i,\epsilon},\\
  \HOM_R(M,N)_{i,\epsilon} := \{f \in \HOM_\F(M,N)_{i,\epsilon}\ |\ f(am) = (-1)^{\bar a \epsilon} a f(m) \ \forall\ a \in R,\ m\in M\}.
\end{gather*}
and the $\F$-vector space
\[
  \Hom_R(M,N) := \HOM_R(M,N)_{0,0}.
\]
We define the dual of an $R$-module $M$ to be $M^* := \HOM_R(M,\F)$.  We also set $\End_R(M) = \Hom_R(M,M)$ and $\END_R(M) = \HOM_R(M,M)$.

Let $R\md$ denote the category of finitely-generated left $R$-modules and let $R\pmd$ denote the category of finitely-generated projective left $R$-modules.  In both categories, we take the morphisms from $M$ to $N$ to be $\Hom_R(M,N)$. For each $n\in\Z$, we have a degree shift functor
\[
  \{n\} \colon R\md \to R\md,\quad M \mapsto \{n\}M.
\]
Here $\{n\}M$ is the same underlying vector space as $M$, with the same $R$-action, but a new grading given by $(\{n\}M)_{i,\epsilon}=M_{i-n,\epsilon}$.  We have an analogous functor on the category of right $R$-modules, and hence on the category of $R$-bimodules.  We also have a parity shift functor
\[
  \Pi \colon R\md \to R\md,\quad M \mapsto \Pi M,
\]
that switches the $\Z_2$-grading of the spaces, i.e.\ $(\Pi M)_{i,\epsilon}=M_{i,\epsilon+1}$. The action of $R$ on $\Pi M$ is given by $r \cdot m=(-1)^{\bar r}rm$, where $rm$ is the action on $M$.  The parity shift functor $\Pi$ for a right module switches the $\Z_2$-grading as above but, unlike the case of left modules, it does not change the signs in the action.  For $n \in \Z$ and $\epsilon \in \Z_2$, we define the functor
\[
  \{n,\epsilon\} \colon R\md \to R\md,\quad M \mapsto \{n,\epsilon\}M := \Pi^\epsilon \{n\} M.
\]
Both the degree shift and parity shift functors leave morphisms unchanged.  Note that we choose to write the shift $\{n,\epsilon\}$ on the left since it commutes with the right action but not necessarily the left action.  We will sometimes write an element of $\HOM_R(M,N)_{n,\epsilon}$ as an element of $\Hom_R(\{n,\epsilon\}M,N)$ or an element of $\Hom_R(M,\{-n,\epsilon\}N)$ when we wish to emphasize its degree.

If $M$ is a left $R$-module and $a \in R$, we define the map
\[
  \pl{a} \colon M \to M,\quad m \mapsto am.
\]
Similarly, if $M$ is a right $R$-module and $a \in R$, we define the map
\[
  \pr{a} \colon M \to M,\quad m \mapsto (-1)^{\bar a \bar m} ma.
\]

\section{Wreath product algebras} \label{sec:wreath}

\subsection{Frobenius algebras}

For the remainder of the paper, we fix an algebraically closed field $\F$ of characteristic zero and a Frobenius algebra $B$ over $\F$ with trace map $\tr_B$ of degree $(-\delta,\sigma)$, $\delta \in \N$, $\sigma \in \Z_2$.  In other words, $\tr_B \colon B \to \F$ is a linear map of degree $(-\delta,\sigma)$ whose kernel contains no nonzero left ideals of $B$.  Here we view $\F$ as concentrated in degree zero.  We denote by $\psiB \colon B \to B$ its Nakayama automorphism, which means, in particular, that, for all homogeneous $a, b\in B$, we have $\tr_B(ab)=(-1)^{\bar{a}\bar{b}}\tr_B(b\psiB(a))$.  We refer the reader to \cite[\S6]{RS15a} for a more detailed discussion of Frobenius graded superalgebras.

Whenever we refer to a basis $\cB $ of $B$, we will assume that it consists of homogeneous elements with respect to the grading.  Given a basis $\cB$, since $B$ is a Frobenius algebra, we can find a right dual basis $\cB^\vee = \{b_k^\vee \mid b_k\in\cB\}$ defined by the property that
\[
  \tr_B(b_k b_\ell^\vee)=\delta_{k,\ell}\quad \text{for all } k,\ell.
\]
Then we have
\begin{equation} \label{eq:dual-bases-decomp}
  \sum_{c \in \cB} \tr_B(bc^\vee)c = b = \sum_{c \in \cB} \tr_B(cb)c^\vee \quad \text{for all } b \in B
\end{equation}
and
\begin{equation}
  |b^\vee| = |b| + \delta,\quad \overline{b^\vee} = \bar b + \sigma,\quad \text{for all } b \in \cB.
\end{equation}
In particular, since $\tr$ has parity $\sigma$, we have
\begin{equation} \label{eq:trace-parity-condition}
  \tr (b) \ne 0 \implies \bar b = \sigma \quad \text{and} \quad \tr(b^\vee) \ne 0 \implies \bar b = 0.
\end{equation}
Note also that if $F$ is any Frobenius algebra with trace map $\tr_F$ of degree $(-\delta_F,\sigma_F)$, Nakayama automorphism $\psi_F$, and basis $\mathcal{F}$, then
\[
  \delta_{f_1, f_2} = \tr_F(f_1 f_2^\vee) = (-1)^{\bar f_1 \overline{f_2^\vee}} \tr_F (f_2^\vee \psi_F(f_1)) = (-1)^{\bar f_1 (\sigma_F + \bar f_2)} \tr_F (f_2^\vee \psi_F(f_1)) \quad \text{for all } f_1,f_2 \in \mathcal{F},
\]
and so
\begin{equation} \label{eq:double-dual-basis}
  \left( f^\vee \right)^\vee = (-1)^{\sigma_F \bar f + \bar f} \psi_F(f),\quad f \in \mathcal{F}.
\end{equation}

The following lemma will be useful in later computations.

\begin{lem} \label{lem:casimir-properties}
  The elements $\sum_{b \in \cB} b \otimes b^\vee$ and $\sum_{b \in \cB} b^\vee \otimes b$ of $B \otimes B$ are independent of the basis $\cB$ and
  \begin{equation} \label{eq:b-bvee-swap}
    \sum_{b \in \cB} b \otimes b^\vee = \sum_{b \in \cB} (-1)^{\sigma \bar b + \bar b} b^\vee \otimes \psiB(b),\quad \sum_{b \in \cB} b^\vee \otimes b = \sum_{b \in \cB} (-1)^{\sigma \bar b + \bar b} \psiB(b) \otimes b^\vee.
  \end{equation}
\end{lem}

\begin{proof}
  The fact that $\sum_{b \in \cB} b \otimes b^\vee$ and $\sum_{b \in \cB} b^\vee \otimes b$ are independent of $\cB$ is standard (and the proof straightforward).
  \details{
    Enumerate the elements of $\cB$ so that $\cB = \{b_1,\dotsc,b_\ell\}$.  Let $\cB' = \{b_1',\dotsc,b_\ell'\}$ be another basis of $B$.  Then there exist invertible $\ell \times \ell$ matrices $M = (m_{ij})$ and $M' = (m_{ij}')$ with entries in $\F$ such that, for $i = 1, \dotsc, \ell$,
    \[
      b_i = \sum_{j=1}^\ell m_{ij} b_j',\quad b_i^\vee = \sum_{j=1}^\ell m_{ij}' (b_j')^\vee.
    \]
    Then, for $i,j \in \{1,\dotsc,\ell\}$, we have
    \[
       \delta_{i,j} = \tr_B(b_i b_j^\vee) = \sum_{r,s=1}^\ell m_{ir} m_{js}' \tr_B(b_r' (b_s')^\vee) = \sum_{r,s=1}^\ell m_{ir} m_{js}' \delta_{r,s} = \sum_{r=1}^\ell m_{ir} m_{jr}'.
    \]
    It follows that $(M')^T = M^{-1}$.  Thus
    \begin{gather*}
      \sum_{i=1}^\ell b_i \otimes b_i^\vee = \sum_{i,j,k=1}^\ell m_{ij} m_{ik}' b_j' \otimes (b_k')^\vee = \sum_{j=1}^\ell b_j' \otimes (b_j')^\vee \quad \text{and} \\
      \sum_{i=1}^\ell b_i^\vee \otimes b_i = \sum_{i,j,k=1}^\ell m_{ij}' m_{ik} (b_j')^\vee \otimes b_k' = \sum_{j=1}^\ell (b_j')^\vee \otimes b_j'.
    \end{gather*}
  }
    Using the independence of the basis, we have
    \[
      \sum_{b \in \cB} b \otimes b^\vee = \sum_{b \in \cB^\vee} b \otimes b^\vee = \sum_{b \in \cB} b^\vee \otimes (b^\vee)^\vee \stackrel{\eqref{eq:double-dual-basis}}{=} \sum_{b \in \cB} (-1)^{\sigma \bar b + \bar b} b^\vee \otimes \psiB(b).
    \]
    The proof of the second relation in \eqref{eq:b-bvee-swap} is analogous.
\end{proof}

\subsection{Wreath product algebras}

For $n \in \N_+$, the symmetric group $S_n$ acts on $V^{\otimes n}$, for any $\Z_2$-graded vector space $V$ (in particular, for $V = B$), by superpermutations.  More precisely, if $s_k \in S_n$ is the simple transposition $(k,k+1)$ for $1 \le k \le n-1$, then
\begin{equation} \label{eq:superpermutation}
  s_k \cdot (v_1 \otimes \dotsb \otimes v_n)
  = (-1)^{\bar v_k \bar v_{k+1}} v_1 \otimes \dotsb \otimes v_{k-1} \otimes v_{k+1} \otimes v_k \otimes v_{k+2} \otimes \dotsb \otimes v_n.
\end{equation}

For $n\in\N$, let $A_n := B^{\otimes n} \rtimes S_n$ be the wreath product algebra, with grading inherited from $B$ (in other words, we take $S_n$ to lie in degree zero).  By convention, we set $A_0:=\F$.  As shown in \cite[Lem.~7.2]{RS15a} (although the trace map on $S_n$ is different in that reference, the proof for the choice made here is analogous), $A_n$ is a Frobenius algebra with degree $(-n \delta, n \sigma)$ trace map $\tr_n = \tr_B^{\otimes n} \otimes \tr_{S_n}$, where $\tr_{S_n}$ is the trace map on $\F S_n$ given by $\tr_{S_n}(\tau) = \delta_{\tau,1}$, $\tau \in S_n$.  The corresponding Nakayama automorphism $\psi_n \colon A_n \to A_n$ is given by
\begin{gather*}
  \psi_n(b_1 \otimes \dotsb \otimes b_n) = \psiB(b_1) \otimes \dotsb \otimes \psiB(b_n),\quad b_1,\dotsc, b_n \in B,\\
  \psi_n(s_i) = (-1)^{\sigma} s_i,\quad i=1,\dotsc,n-1.
\end{gather*}
\details{
  Suppose $\bb \in B^{\otimes n}$ and $i \in \{1,\dotsc,n-1\}$.  Write $\bb = \bb' + \bb''$ where
  \[
    \bb' \in (B_\sigma)^{\otimes n},\quad \bb'' \in \bigoplus_{\substack{\epsilon_1,\dotsc,\epsilon_n \in \Z_2 \\ \text{not all equal to } \sigma}} \left( B_{\epsilon_1} \otimes \dotsb \otimes B_{\epsilon_n} \right).
  \]
  Here, for $\epsilon \in \Z_2$, we set $B_\epsilon = \bigoplus_{m \in \N} B_{m,\epsilon}$.  Then
  \[
    \tr_n((\bb \otimes s_i)s_i) = \tr_n(\bb) = \tr_B^{\otimes n}(\bb') = (-1)^\sigma \tr_B^{\otimes n}(s_i \cdot \bb') = (-1)^\sigma \tr_B^{\otimes n} (s_i \cdot \bb) = (-1)^{\sigma} \tr_n(s_i(\bb \otimes s_i)).
  \]
  On the other hand, for $\tau \in S_n$ with $\tau \ne s_i$, we have
  \[
    \tr_n((\bb \otimes \tau)s_i) = 0 = \tr_n(s_i(\bb \otimes \tau)).
  \]
  Thus $\psi_n(s_i) = (-1)^\sigma s_i$.

  For $b_1,\dotsc,b_n \in B$ and $\bb \in B^{\otimes n}$, we have
  \[
    \tr_n((b_1 \otimes \dotsb \otimes b_n)\bb) = \tr_B^{\otimes n}((b_1 \otimes \dotsb \otimes b_n)\bb) = \tr_B^{\otimes n}(\bb(\psiB(b_1) \otimes \dotsb \otimes \psiB(b_n))).
  \]
  On the other hand, for $\tau \in S_n$ with $\tau \ne 1$, we have
  \[
    \tr_n((b_1 \otimes \dotsb \otimes b_n)(\bb \otimes \tau)) = 0 = \tr_n((\bb \otimes \tau)(b_1 \otimes \dotsb \otimes b_n)).
  \]
  Thus $\psi_n(b_1 \otimes \dotsb \otimes b_n) = \psiB(b_1) \otimes \dotsb \otimes \psiB(b_n)$.
}

We naturally view $A_n$ as a subalgebra of $A_{n+1}$ via the maps
\[
  \bb \mapsto \bb \otimes 1_B,\quad s_i \mapsto s_i,\quad \bb \in B^{\otimes n},\ i \in \{1,\dotsc,n-1\}.
\]
Via the second map, we also view $S_n$ as the subgroup of $S_{n+1}$ fixing $n+1$.  In this way, $A_{n+1}$ is naturally an $(A_{n+1}, A_n)$-bimodule as well as an $(A_n,A_{n+1})$-bimodule.  We denote these two bimodules by $(n+1)_n$ and $\prescript{}{n}(n+1)$ respectively.  The notation $(n)$ will denote $A_n$ viewed in the natural way as an $(A_n,A_n)$-bimodule.  We will use juxtaposition to denote the tensor product.  For example, $(n+1)_n (n+1)_n$ is $A_{n+1} \otimes_{A_n} A_{n+1}$, viewed as an $(A_{n+1},A_n)$-bimodule.  We will use the notation $1_n$ to denote the unit element of $A_n$.

\begin{eg}[Sergeev algebra] \label{eg:Sergeev-algebra}
  When $B$ is the rank one Clifford algebra $\Cl$ with one odd generator $c$ and relation $c^2=1$, then the wreath product algebra $\bS_n := \Cl^{\otimes n} \rtimes S_n$ is a \emph{Sergeev algebra} (sometimes also called a \emph{Hecke--Clifford algebra}).
\end{eg}

\begin{lem} \label{lem:decom}
  The algebra $A_{n+1}$ is a finitely-generated free left and right $A_n$-module.  More precisely, we have an isomorphism
  \[
    A_{n+1} \cong \bigoplus_{b \in \cB} \bigoplus_{i = 1}^{n+1} A_n (1^{\otimes n} \otimes b) s_n \dotsm s_i
  \]
  of left $A_n$-modules and an isomorphism
  \[
    A_{n+1} \cong \bigoplus_{b \in \cB} \bigoplus_{i = 1}^{n+1} (1^{\otimes n} \otimes b) s_i \dotsm s_n A_n
  \]
  of right $A_n$-modules. Here $\cB$ is a basis of $B$ and $s_i$ denotes the simple transposition $(i,i+1) \in S_n$.  By convention, $s_n \dotsm s_i = s_i \dotsm s_n = 1$ when $i=n+1$.
\end{lem}

\begin{proof}
  This is a direct consequence of the fact that $\F [S_{n+1}] \cong \bigoplus_{i=1}^{n+1} \F [S_n] s_n \dotsm s_i$ as a left $S_n$-module.
\end{proof}

In the following proposition, we refer to a trace map in the sense of Frobenius extensions.  See \cite[\S4]{PS15} for a discussion of this concept in the graded super setting.

\begin{prop} \label{prop:Frob-ext}
  The algebra $A_{n+1}$ is a Frobenius extension of $A_n$ with degree $(-\delta,\sigma)$ trace map
  \begin{gather*}
    \tr \colon \prescript{}{n}(n+1)_n \to \prescript{}{n}(n)_{n},\\
    \tr((b_1 \otimes \dotsb \otimes b_{n+1})\tau) =
    \begin{cases}
      (-1)^{\sigma(\bar b_1 + \dotsb + \bar b_n)} \tr_B(b_{n+1}) (b_1 \otimes \dotsb \otimes b_n) \tau & \text{if } \tau \in S_n, \\
      0 & \text{otherwise}.
    \end{cases}
  \end{gather*}
  Furthermore,
  \begin{gather*}
    \left\{ \left. x_{b,i} = \left( 1^{\otimes n} \otimes b \right) s_n \dotsm s_i\ \right| \ b \in \cB,\ i=1,\dotsc,n+1 \right\} \quad \text{and} \\
    \left\{ \left. y_{b,i} = s_i \dotsm s_n \left( 1^{\otimes n} \otimes b^\vee \right)\ \right| \ b \in \cB,\ i=1,\dotsc,n+1 \right\}
  \end{gather*}
  are dual sets of generators in the sense of \cite[Prop.~4.9]{PS15}.
\end{prop}

\begin{proof}
  Since $\tr$ is clearly an $(A_n,A_n)$-bimodule homomorphism, it suffices, by \cite[Prop.~4.9]{PS15}, to show that the given sets satisfy \cite[(4.2)]{PS15}.
  \details{
    To see that $\tr$ is an $(A_n,A_n)$-bimodule homomorphism, note that it commutes with the left and right action by elements of $S_n$ since such elements only permute the first $n$ factors of $B^{\otimes (n+1)}$ and, for $\tau_1 \in S_n$ and $\tau_2 \in S_{n+1}$, we have $\tau_1 \tau_2 \in S_n$ if and only if $\tau_2 \in S_n$.  Now, for $b_1,\dotsc,b_{n+1}, b_1',\dotsc, b_n' \in B$ and $\tau \in S_n$ (the case $\tau \not \in S_n$ is straightforward), we have
    \begin{align*}
      \tr((b_1' \otimes \dotsb \otimes b_n' \otimes 1_B)&(b_1 \otimes \dotsb \otimes b_{n+1}) \tau)
      = (-1)^{\sum_{i < j \le n} \bar b_i \bar b_j'} \tr((b_1'b_1 \otimes \dotsb \otimes b_n' b_n \otimes b_{n+1}) \tau) \\
      &= (-1)^{\sum_{i < j \le n} \bar b_i \bar b_j'} (-1)^{\sigma (\bar b_1 + \dotsb + \bar b_n + \bar b_1' + \dotsb + \bar b_n')} \tr_B(b_{n+1}) (b_1' b_1 \otimes \dotsb \otimes b_n' b_n) \tau \\
      &= (-1)^{\sigma (\bar b_1 + \dotsb + \bar b_n + \bar b_1' + \dotsb + \bar b_n')} (b_1' \otimes \dotsb \otimes b_n') \tr_B(b_{n+1}) (b_1 \otimes \dotsb \otimes b_n) \tau \\
      &= (-1)^{\sigma(\bar b_1' + \dotsb + \bar b_n')} (b_1' \otimes \dotsb \otimes b_n') \tr((b_1 \otimes \dotsb \otimes b_{n+1})\tau)
    \end{align*}
    and
    \begin{align*}
      \tr((b_1 \otimes \dotsb \otimes b_{n+1}) &\tau (b_1' \otimes \dotsb \otimes b_n' \otimes 1_B))
      = (-1)^{\sum_{i < j \le n+1} \bar b_{\tau^{-1}(i)}' \bar b_j} \tr((b_1 b_{\tau^{-1}(1)}' \otimes \dotsb \otimes b_n b_{\tau^{-1}(n)}' \otimes b_{n+1}) \tau) \\
      &= (-1)^{\sum_{i < j \le n+1} \bar b_{\tau^{-1}(i)}' \bar b_j} (-1)^{\sigma (\bar b_1 + \dotsb + \bar b_n + \bar b_1' + \dotsb + \bar b_n')} \tr_B(b_{n+1}) (b_1 b_{\tau^{-1}(1)}' \otimes \dotsb \otimes b_n b_{\tau^{-1}(n)}') \tau \\
      &\! \! \stackrel{\eqref{eq:trace-parity-condition}}{=} (-1)^{\sigma (\bar b_1 + \dotsb + \bar b_n)} \tr_B(b_{n+1}) (b_1 \otimes \dotsb \otimes b_n) \tau (b_1' \otimes \dotsb \otimes b_n')\\
      &= \tr((b_1 \otimes \dotsb \otimes b_{n+1})\tau) (b_1' \otimes \dotsb \otimes b_n').
    \end{align*}
    Thus $\tau$ is an $(A_n,A_n)$-bimodule homomorphism of degree $(-\delta,\sigma)$.xxx
  }
  Since the $x_{b,i}$ form a basis of $A_{n+1}$ as a left $A_n$-module by Lemma~\ref{lem:decom}, to prove that the first equation in \cite[(4.2)]{PS15} is satisfied, it suffices to show it holds for $a=x_{b,i}$, $b \in \cB$, $i \in \{1,\dotsc,n+1\}$.  Similarly, it suffices to show that the first and last terms in \cite[(4.2)]{PS15} are equal for $a=y_{b,i}$, $b \in \cB$, $i \in \{1,\dotsc,n+1\}$.  But this follows immediately from the fact that, for $b,c \in \cB$, $i,j \in \{1,\dotsc,n+1\}$, we have
  \[
    \tr(x_{b,i} y_{c,j}) = \tr((1^{\otimes n} \otimes b)s_n \dotsb s_i s_j \dotsb s_n (1^{\otimes n} \otimes c^\vee)) = \delta_{i,j} \tr_B(bc^\vee) = \delta_{i,j} \delta_{b,c},
  \]
  where the second equality follows from the fact that $s_n \dotsb s_i s_j \dotsb s_n \in S_n$ if and only if $i=j$.
  \details{
    We then have
    \[
      (-1)^{\sigma \bar b} \sum_{\substack{c \in \cB \\ j \in \{1,\dotsc,n+1\}}} (-1)^{\sigma \bar c} \tr(x_{b,i} y_{c,j}) x_{c,j} = x_{b,i}
      \qquad \text{and} \qquad
      \sum_{\substack{c \in \cB \\ j \in \{1,\dotsc,n+1\}}}y_{c,j} \tr(x_{c,j} y_{b,i}) = y_{b,i}
    \]
    as desired.  (Note that the twistings $\alpha$ and $\beta$ in \cite[Prop.~4.9]{PS15} are trivial in the setting of the current paper.)
  }
\end{proof}

If $R$ and $S$ are rings, we will naturally identify an $(R,S)$-bimodule $M$ with the bimodules $R \otimes_R M$ and $M \otimes_S S$.

\begin{prop} \label{prop:adjunction-maps}
  The maps
  \begin{gather*}
    \varepsilon_\rR \colon (n+1)_n(n+1) \to (n+1),\quad \varepsilon_\rR(a \otimes a') = aa',\quad a \in (n+1)_n,\ a' \in \prescript{}{n}(n+1), \\
    \eta_\rR \colon (n) \to {_n}(n+1)_n,\quad \eta_{\rR}(\bb \tau) = (\bb \otimes 1_B)\tau,\quad \bb \in B^{\otimes n},\ \tau \in S_n, \\
    \varepsilon_\rL \colon \{-\delta,\sigma\}\prescript{}{n}(n+1)_n \to (n),\\
    \varepsilon_\rL((b_1 \otimes \dotsb \otimes b_{n+1}) \tau) =
    \begin{cases}
      0 & \text{if } \tau \not \in S_n,\\
      (-1)^{\sigma (\bar b_1 + \dotsb + \bar b_n)} \tr_B(b_{n+1}) (b_1 \otimes \dotsb \otimes b_n) \tau & \text{if } \tau \in S_n,
    \end{cases}\\
    \eta_\rL \colon (n+1) \to \{-\delta,\sigma\} (n+1)_n \otimes_{A_n} \prescript{}{n}(n+1),\\
    \eta_\rL(a) = (-1)^{\sigma \bar a} a \sum_{\substack{b \in \cB \\ i \in \{1,\dotsc,n+1\}}} (-1)^{\sigma \bar b^\vee} s_i \dotsm s_n (1_B^{\otimes n} \otimes b^\vee) \otimes (1_B^{\otimes n} \otimes b) s_n \dotsm s_i,\quad a \in (n+1),
  \end{gather*}
  are bimodule homomorphisms and satisfy the relations
  \begin{gather}
    (\varepsilon_\rR \otimes \id) \circ (\id \otimes \eta_\rR) = \id, \quad
    (\id \otimes \varepsilon_\rR) \circ (\eta_\rR \otimes \id) = \id, \label{eq:right-unit-counit} \\
    (\varepsilon_\rL \otimes \id) \circ (\id \otimes \eta_\rL) = (-1)^\sigma \id, \quad
    (\id \otimes \varepsilon_\rL) \circ (\eta_\rL \otimes \id) = \id. \label{eq:left-unit-counit}
  \end{gather}
  In particular, $(n+1)_n$ is left adjoint to $\prescript{}{n}(n+1)$ and right adjoint to $\{-\delta,\sigma\} \prescript{}{n}(n+1)$ in the 2-category of bimodules over rings.
\end{prop}

\begin{proof}
  The map $\varepsilon_\rR$ is multiplication and $\eta_\rR$ is the inclusion $A_n \hookrightarrow A_{n+1}$.  Thus, both are clearly bimodule homomorphisms.  The relations \eqref{eq:right-unit-counit} are standard.  They are the usual unit-counit equations making induction left adjoint to restriction.
  \details{
    For $a \in (n+1)_n$, $a' \in \prescript{}{n}(n+1)$, we have
    \begin{gather*}
      (\varepsilon_\rR \otimes \id) \circ (\id \otimes \eta_\rR) (a) = (\varepsilon_\rR \otimes \id) \circ (\id \otimes \eta_\rR) (a \otimes 1_n) =  (\varepsilon_\rR \otimes \id)(a \otimes 1_{n+1}) = a, \\
      (\id \otimes \varepsilon_\rR) \circ (\eta_\rR \otimes \id) (a') = (\id \otimes \varepsilon_\rR) \circ (\eta_\rR \otimes \id) (1_n \otimes a') = (\id \otimes \varepsilon_\rR)(1_{n+1} \otimes a') = a'.
    \end{gather*}
  }

  By Proposition~\ref{prop:Frob-ext} and \cite[Th.~6.2]{PS15}, we know that $(n+1)_n$ is left adjoint to $\{-\delta,\sigma\} \prescript{}{n}(n+1)$, with the unit $\eta$ and counit $\varepsilon$ defined as in \cite[(6.1), (6.2)]{PS15}.  These maps are precisely the $\eta_\rL$ and $\varepsilon_\rL$ in the statement of the proposition.  (Note that the twistings $\alpha$ and $\beta$ of \cite{PS15} are trivial in the setting of the current paper.)  Alternatively, the relations \eqref{eq:left-unit-counit} can easily be checked directly.
  \details{
    For the left equation in \eqref{eq:left-unit-counit} we have
    \[
      1_{n+1}
      \mapsto \sum_{\substack{b \in \cB \\ i \in \{1,\dotsc,n+1\}}} (-1)^{\sigma \bar b^\vee} s_i \dotsm s_n (1_B^{\otimes n} \otimes b^\vee) \otimes (1_B^{\otimes n} \otimes b) s_n \dotsm s_i
      \mapsto (-1)^\sigma \sum_{b \in \cB} \tr(b^\vee) b
      = (-1)^\sigma.
    \]
    For the right equation in \eqref{eq:left-unit-counit} we have
    \[
      1_{n+1}
      \mapsto \sum_{\substack{b \in \cB \\ i \in \{1,\dotsc,n+1\}}} (-1)^{\sigma \bar b^\vee} s_i \dotsm s_n (1_B^{\otimes n} \otimes b^\vee) \otimes (1_B^{\otimes n} \otimes b) s_n \dotsm s_i
      \mapsto \sum_{b \in \cB} b^\vee \tr(b)
      = 1_{n+1}.
    \]
  }
\end{proof}

\section{Modules for wreath product algebras} \label{sec:wreath-modules}

In this section we give a complete description of the simple and indecomposable projective modules for wreath product algebras.  In the ungraded and non-super setting, such a description (for the simple modules) is a special case of \cite[Th.~A.5]{Mac80}.  (We refer the reader also to \cite[Th.~A.6]{RR03} for a correction to an error in the proof of \cite[Th.~A.5]{Mac80}, although this correction is not needed for the case of wreath product algebras, since one can take the map $\alpha$ of \cite{RR03} to be trivial.)  However, the generalization to the super setting is nontrivial due to the fact that the outer tensor product of simple modules may not be simple in general.  We refer the reader not familiar with the behaviour of modules in the super setting to the overview in \cite[Ch.~12]{Kle05}.

Let $V_1,\dotsc,V_\ell$ be a complete list of simple finite-dimensional modules for $B$, up to to shift and isomorphism.  Shifting with respect to the $\Z$-grading if necessary, we may assume that the $V_i$ are non-negatively graded, with nonzero degree zero piece.   Recall that a simple module is said to be of type $Q$ if it is isomorphic to its parity shift, and type $M$ otherwise.  After possibly reordering, we assume that
\begin{center}
  $V_1,\dotsc,V_r$ are of type $M$ and $V_{r+1},\dotsc,V_\ell$ are of type $Q$.
\end{center}
For $i_1,\dotsc,i_n \in \{1,\dotsc,\ell\}$, the module $B^{\otimes n}$-module $V_{i_1} \boxtimes \dotsb \boxtimes V_{i_n}$ contains a simple submodule $\widehat{V}_{i_1,\dotsc,i_n}$ such that $V_{i_1} \boxtimes \dotsb \boxtimes V_{i_n}$ is isomorphic, as a $B^{\otimes n}$-module, to a direct sum of copies of $\widehat{V}_{i_1,\dotsc,i_n}$ and possibly its parity shift.  The modules
\[
  \widehat{V}_{i_1,\dotsc,i_n},\quad (i_1,\dotsc,i_n) \in \{1,\dotsc,\ell\}^n,
\]
form a complete set of representatives of the isomorphism classes of simple $B^{\otimes n}$-modules, up to shift.  (See, for example, \cite[Lem.~12.2.13]{Kle05}.)

For a $B^{\otimes n}$-module $V$ and $\tau \in S_n$, we define the twist $\prescript{\tau}{}{V}$ to be the $B^{\otimes n}$-module with underlying additive group $V$ and with action
\[
  (\bb,v) \mapsto (\tau^{-1} \cdot \bb)v,\quad \bb \in B^{\otimes n},\ v \in V.
\]
Then, for $\tau \in S_n$ and $(i_1,\dotsc,i_n) \in \{1,\dotsc,\ell\}^n$, we have
\begin{equation} \label{eq:simple-tau-action}
  \prescript{\tau}{}{\widehat{V}_{i_1,\dotsc,i_n}} \cong \widehat{V}_{i_{\tau^{-1}(1)},\dotsc, i_{\tau^{-1}(n)}} \qquad \text{(up to parity shift)}.
\end{equation}
\details{
  We have the vector space isomorphism
  \[
    \tau \colon \prescript{\tau}{}{(V_{i_1} \boxtimes \dotsb \boxtimes V_{i_n})} \to V_{i_{\tau^{-1}(1)}} \otimes \dotsb \otimes V_{i_{\tau^{-1}(n)}}.
  \]
  Since
  \[
    \tau( (\tau^{-1} \cdot \bb)v) = \bb(\tau \cdot v),
  \]
  this is also an isomorphism of $B^{\otimes n}$-modules.  Restricting to simple summands implies the stated result.
}

For $m \in \N_+$, let $\partition(m)$ denote the set of partitions of $m$ and let $\partition_\textup{strict}(m)$ denote the set of strict partitions of $m$ (i.e.\ partitions of $m$ with pairwise distinct parts).  Set $\partition = \bigcup_{m \in \N_+} \partition(m)$ and $\partition_\textup{strict} = \bigcup_{m \in \N_+} \partition_\textup{strict}(m)$, and let
\[
  \fP = \partition^r \times \partition_\textup{strict}^{\ell-r},\quad \fP(n) = \{(\lambda^1,\dotsc,\lambda^\ell) \in \fP \mid \textstyle \sum_{i=1}^\ell |\lambda^i| = n \},
\]
where $|\lambda|$ denotes the size of a partition (i.e.\ the sum of its parts).  We let $\ell(\lambda)$ denote the length of the partition $\lambda$ (i.e.\ the number of nonzero parts).  It is well known that irreducible representations of the symmetric group $S_m$ are enumerated by $\partition(m)$ and the simple modules of the Sergeev algebra $\bS_m$ are enumerated, up to parity shift, by $\partition_\textup{strict}(m)$.  (For the Sergeev algebra case see, for example, \cite[\S 2]{WW12}.)  For $\lambda \in \partition(m)$, let $L_\lambda$ denote the corresponding irreducible representation of $S_m$. Similarly, for each $\lambda \in \partition_\textup{strict}(m)$, let $L'_\lambda$ denote the corresponding simple module of the Sergeev algebra $\bS_m$.  Note that, if $\ell(\lambda)$ is even, then $\bS_m$ has two simple modules corresponding to $\lambda$, related by a parity shift.  We have made a choice of one of these.

Recall the definition of $\Cl$ and $\bS_m$ given in Example~\ref{eg:Sergeev-algebra}.  For $\blambda = (\lambda^1,\dotsc,\lambda^\ell) \in \fP$, define
\begin{gather*}
  V_\blambda := V_1^{\otimes |\lambda^1|} \otimes \dotsb \otimes V_\ell^{\otimes |\lambda^\ell|},\quad
  R_\blambda := \HOM_{B^{\otimes n}}(V_\blambda,V_\blambda) \cong \F^{\otimes (|\lambda^1| + \dotsb + |\lambda^r|)} \otimes \Cl^{\otimes (|\lambda^{r+1}| + \dotsb + |\lambda^n|)},\\
  S_\blambda := S_{|\lambda^1|} \times \dotsb \times S_{|\lambda^\ell|},\quad
  \fS_\blambda := \F S_{|\lambda^1|} \otimes \dotsb \otimes \F S_{|\lambda^r|} \otimes \bS_{|\lambda^{r+1}|} \otimes \dotsb \otimes \bS_{|\lambda^\ell|}.
\end{gather*}
We fix the isomorphism above and view it as equality.  Then we can naturally view $R_\blambda$ as a subalgebra of $\fS_\blambda$, and $S_\blambda$ as a subgroup of the group of units of $\fS_\blambda$.  Then we have
\begin{equation} \label{eq:RS=fS}
  R_\blambda S_\blambda = \fS_\blambda,\quad \blambda \in \fP.
\end{equation}

For $\lambda^{r+1},\dotsc,\lambda^\ell \in \partition_\textup{strict}$, we let $L'_{\lambda^{r+1},\dotsc,\lambda^\ell}$ denote a simple $\bS_{|\lambda^{r+1}|} \otimes \dotsb \otimes \bS_{|\lambda^\ell|}$-submodule of $L_{\lambda^{r+1}}' \boxtimes \dotsb \boxtimes L_{\lambda^\ell}'$.  This choice is unique up to isomorphism and parity shift.  Then, for $\blambda = (\lambda^1,\dotsc,\lambda^\ell) \in \fP$,
\[
  L_\blambda := L_{\lambda^1} \boxtimes \dotsb \boxtimes L_{\lambda^r} \boxtimes L'_{\lambda^{r+1},\dotsc,\lambda^\ell}
\]
is a simple $\fS_\blambda$-module.  In fact $L_\blambda$, $\blambda \in \fP$, is a complete list of simple $\fS_\blambda$-modules, up to isomorphism and shift.

Note that $V_\blambda$ is naturally a $(B^{\otimes n}, R_\blambda)$-bimodule.  We also have a natural action of $S_\blambda$ on $V_\blambda$ permuting the factors.  The embedding $R_\blambda \hookrightarrow \fS_\blambda$ gives rise to an action of $R_\blambda$ on $L_\blambda$.  We can thus define a $B^{\otimes n} \rtimes S_\blambda$-action on $V_\blambda \otimes_{R_\blambda} L_\blambda$ by
\begin{equation} \label{eq:Bn-blambda-action}
  (\bb \tau, \bv \otimes w) \mapsto \bb (\tau \cdot \bv) \otimes \tau w,\quad \bb \in B^{\otimes n},\ \tau \in S_\blambda,\ \bv \in V_\blambda,\ w \in L_\blambda.
\end{equation}
\details{
  Consider the map
  \begin{gather*}
    \xi \colon (B^{\otimes n} \rtimes S_\blambda) \times (V_\blambda \times L_\blambda) \to V_\blambda \otimes_{R_\blambda} L_\blambda, \\
    (\bb \tau, (\bv, w)) \mapsto \bb (\tau \cdot \bv) \otimes_{R_\blambda} \tau w,\quad \bb \in B^{\otimes n},\ \tau \in S_\blambda,\ \bv \in V_\blambda,\ w \in L_\blambda.
  \end{gather*}
  This map is linear in the $V_\blambda$ and $L_\blambda$ arguments.  Now, for $a \in R_\blambda$, $\bb \in B^{\otimes n}$, $\tau \in S_\blambda$, $\bv \in V_\blambda$, $w \in L_\blambda$, we have
  \begin{align*}
    \xi(\bb \tau, (\bv a, w))
    &= \bb(\tau \cdot (\bv a)) \otimes_{R_\blambda} \tau w \\
    &= (\bb (\tau \cdot \bv)) \tau(a) \otimes_{R_\blambda} \tau w \\
    &= (\bb (\tau \cdot \bv)) \otimes_{R_\blambda} \tau(a) \tau w \\
    &= (\bb (\tau \cdot \bv)) \otimes_{R_\blambda} \tau a w \\
    &= \xi(\bb \tau,(\bv,aw)).
  \end{align*}
  Thus, the map $\xi$ induces the map \eqref{eq:Bn-blambda-action}.  Denoting this map now by $\cdot$, we have
  \[
    (\bb_1 \tau_1) \cdot ((\bb_2 \tau_2) \cdot (\bv \otimes w)) = \bb_1 \tau_1 \cdot (\bb_2 (\tau_2 \cdot \bv)) \otimes \tau_1 \tau_2 = \bb_1 \tau_1(\bb_2) ((\tau_1 \tau_2) \cdot \bv) \otimes \tau_1 \tau_2 = (\bb_1 \tau_1(\bb_2) \tau_1 \tau_2)(\bv \otimes w).
  \]
}

\begin{lem} \label{lem:sergeev-action-big-hom-space}
  If $M$ is a $B^{\otimes n} \rtimes S_n$-module and $\blambda \in \fP(n)$, then the space $\HOM_{B^{\otimes n}} (V_\blambda, M)$ is a $\fS_\blambda$-module with action given by
  \[
    (a \tau) \cdot f = (-1)^{\bar a \bar f} \tau f \tau^{-1} \pr{a},\quad a \in R_\blambda,\ \tau \in S_\blambda,\ f \in \HOM_{B^{\otimes n}} (V_\blambda, M).
  \]
\end{lem}

\begin{proof}
  The proof consists of straightforward verification and will be omitted.
  \details{
    Let $a \in R_\blambda$, $\tau \in S_\blambda$, $f \in \HOM_{B^{\otimes n}} (V_\blambda, M)$, $\bb \in B^{\otimes n}$, and $v \in V_\blambda$.  Then we have
    \begin{multline*}
      \tau f \tau^{-1} \pr{a} (\bb v)
      = (-1)^{\bar a (\bar \bb + \bar v)} \tau f \tau^{-1} ((\bb v) \cdot a) \\
      = (-1)^{\bar a (\bar \bb + \bar v)} \tau f (\tau^{-1}(\bb) \tau^{-1}(v \cdot a))
      = (-1)^{\bar a (\bar \bb + \bar v) + \bar \bb \bar f} \tau \left( \tau^{-1}(\bb) f (\tau^{-1}(v \cdot a)) \right) \\
      = (-1)^{\bar a (\bar \bb + \bar v) + \bar \bb \bar f} \bb \tau f (\tau^{-1}(v \cdot a))
      = (-1)^{\bar \bb (\bar f + \bar a)} \bb \left( \tau f  \tau^{-1} \pr{a}(v) \right).
    \end{multline*}
    Thus $(-1)^{\bar a f} \tau f \tau^{-1} \pr{a} \in \HOM_{B^{\otimes n}}(V_\blambda,M)$.

    The action is clearly linear.  Now, if $\tau' \in S_\blambda$ and $a' = R_\blambda$, we have
    \begin{align*}
      (a \tau) \cdot ((a' \tau') \cdot f)
      &= (-1)^{\bar a' \bar f} (a \tau) \cdot \left( \tau' f (\tau')^{-1} \pr{a'} \right) \\
      &= (-1)^{\bar a' \bar f + \bar a (\bar f + \bar a')} \tau \tau' f (\tau')^{-1} \pr{a'} \tau^{-1} \pr{a} \\
      &= (-1)^{\bar a' \bar f + \bar a (\bar f + \bar a')} \tau \tau' f (\tau \tau')^{-1} \pr{\tau(a')} \pr{a} \\
      &= (-1)^{(\bar a + \bar a') \bar f} (\tau \tau') f (\tau \tau')^{-1} \pr{(a \tau(a'))} \\
      &= (a \tau(a') \tau \tau') \cdot f.
    \end{align*}
    Thus the given expression does indeed define a $\fS_\blambda$-action.
  }
\end{proof}

For $\blambda \in \fP$, define
\[
  E_\blambda =
  \begin{cases}
    \Cl & \text{if $|\lambda^{r+1}| + \dotsb + |\lambda^\ell|$ is odd}, \\
    \F & \text{otherwise}.
  \end{cases}
\]
Then $E_\blambda \cong \HOM_{B^{\otimes n}}(\widehat{V}_\blambda, \widehat{V}_\blambda)$, where $\widehat{V}_\blambda$ is the simple submodule of $V_\blambda$.  Also, $E_\blambda \cong \HOM_{R_\blambda}(U_\blambda,U_\blambda) \cong \HOM_{R_\blambda}(U_\blambda^*,U_\blambda^*)$, where $U_\blambda$ is the simple $R_\blambda$-module (unique up to isomorphism and shift).  We fix such isomorphisms and view them as equalities from now on.

\begin{lem}
  For $\blambda \in \fP(n)$, we have $V_\blambda \cong \widehat{V}_\blambda \otimes_{E_\blambda} U_\blambda$ as $(B^{\otimes n}, R_\blambda)$-bimodules.
\end{lem}

\begin{proof}
  Let $m = |\lambda^{r+1}| + \dotsc + |\lambda^\ell|$.  By repeated application of \cite[Lem.~12.2.13]{Kle05}, we see that $\widehat{V}_\blambda$ and $U_\blambda$ are of type $Q$ if $m$ is odd and of type $M$ if $m$ is even.  Furthermore, $V_\blambda$ is a direct sum of $2^{\lfloor m/2 \rfloor}$ copies of $\widehat{V}_\blambda$ (up to shift) and the dimension of $U_\blambda$ is $2^{\lfloor (m+1)/2 \rfloor}$  (see, for example, \cite[p.~158]{Kle05}).  Finally, $E_\blambda$ is two dimensional when $m$ is odd and one dimensional when $m$ is even.  Thus the result follows by dimension considerations.
\end{proof}

\begin{prop} \label{prop:wreath-simples}
  The induced modules
  \[
    \Ind_{B^{\otimes n} \rtimes S_\blambda}^{B^{\otimes n} \rtimes S_n} (V_\blambda \otimes_{R_\blambda} L_\blambda),\quad \blambda \in \fP(n),
  \]
  are a complete list of simple $B^{\otimes n} \rtimes S_n$-modules up to isomorphism and shift.
\end{prop}

\begin{proof}
  Let $M$ be a simple $B^{\otimes n} \rtimes S_n$-module and let $U$ be a simple $B^{\otimes n}$-submodule of $M$.  It follows that $\sum_{\tau \in S_n} \tau U$ is a nonzero $B^{\otimes n} \rtimes S_n$-submodule of $M$ and hence is equal to $M$.  Note that $\tau U \cong \prescript{\tau}{}{U}$ as $B^{\otimes n}$-modules.
  \details{
    For $\bb \in B^{\otimes n}$ and $u \in U$, we have $\bb(\tau u) = \tau (\tau^{-1} \cdot \bb)u$, and hence $\tau U \cong \prescript{\tau}{}{E}$.
  }
  Thus, by \eqref{eq:simple-tau-action}, we may assume, without loss of generality, that $U = \widehat{V}_\blambda$ for some $\blambda \in \fP(n)$.

  Set $N = \sum_{\tau \in S_\blambda} \tau U$, which is a $B^{\otimes n} \rtimes S_\blambda$-module. Let $\tau_1,\dotsc,\tau_s$ be left coset representatives of $S_\blambda$ in $S_n$.  Then the $\tau_i N$ are the isotypic components of $M$ as a $B^{\otimes n}$-module, and we have
  \[
    M = \bigoplus_{i=1}^s \tau_i N = \Ind_{B^{\otimes n} \rtimes S_\blambda}^{B^{\otimes n} \rtimes S_n}(N) = (B^{\otimes n} \rtimes S_n) \otimes_{B^{\otimes n} \rtimes S_\blambda} N.
  \]
  Here we consider a simple module and its parity shift to lie in the same isotypic component.

  Let
  \[
    W = \HOM_{B^{\otimes n}}(V_\blambda,M) = \HOM_{B^{\otimes n}}(V_\blambda,N).
  \]
  Then $W$ is a $\fS_\blambda$-module by Lemma~\ref{lem:sergeev-action-big-hom-space}.  In fact, $W$ must simple, since if $W'$ is a proper $\fS_\blambda$-submodule of $W$, then $\Ind_{B^{\otimes n} \rtimes S_\blambda}^{B^{\otimes n} \rtimes S_n} (V_\blambda \otimes_{R_\blambda} W')$ is a proper $B^{\otimes n} \rtimes S_n$-submodule of $M$.  Since $V_\blambda$ depends only on the size of the partitions in $\blambda$, we may, without loss of generality, assume that $W \cong L_\blambda$.  Then $V_\blambda \otimes_{R_\blambda} W$ is a $B^{\otimes n} \rtimes S_\blambda$-module with action given by \eqref{eq:Bn-blambda-action}, and, as $B^{\otimes n} \rtimes S_n$-modules,
  \begin{align*}
    V_\blambda \otimes_{R_\blambda} W
    &= V_\blambda \otimes_{R_\blambda} \HOM_{B^{\otimes n}}(V_\blambda, N) \\
    &\cong V_\blambda \otimes_{R_\blambda} V_\blambda^* \otimes_{B^{\otimes n}} N \\
    &\cong \widehat{V}_\blambda \otimes_{E_\blambda} U_\blambda \otimes_{R_\blambda} U_\blambda^* \otimes_{E_\blambda} \left( \widehat{V}_\blambda \right)^* \otimes_{B^{\otimes n}} N \\
    &\cong \widehat{V}_\blambda \otimes_{E_\blambda} \HOM_{R_\blambda}(U_\blambda^*,U_\blambda^*) \otimes_{E_\blambda} \left( \widehat{V}_\blambda \right)^* \otimes_{B^{\otimes n}} N \\
    &\cong \widehat{V}_\blambda \otimes_{E_\blambda} E_\blambda \otimes_{E_\blambda} \left( \widehat{V}_\blambda \right)^* \otimes_{B^{\otimes n}} N \\
    &\cong \widehat{V}_\blambda \otimes_{E_\blambda} \left( \widehat{V}_\blambda \right)^* \otimes_{B^{\otimes n}} M \\
    &\cong \HOM_{E_\blambda}(\widehat{V}_\blambda , \widehat{V}_\blambda) \otimes_{B^{\otimes n}} N \\
    &\cong \left( B^{\otimes n}/\Ann_{B^{\otimes n}}(\widehat{V}_\blambda) \right) \otimes_{B^{\otimes n}} N \\
    &\cong N,
  \end{align*}
  where $\Ann_{B^{\otimes n}}(\widehat{V}_\blambda)$ is the annihilator of $\widehat{V}_\blambda$ in $B^{\otimes n}$ and, in the last isomorphism, we use the fact that $N$ is the $\widehat{V}_\blambda$-isotypic component of $M$.

  It follows that every simple $B^{\otimes n} \rtimes S_n$-module is of the form given in the lemma.  Uniqueness follows immediately from the above, since $\blambda$ was uniquely determined by $M$.  No two choices of $\blambda$ yield isomorphic $B^{\otimes n} \rtimes S_n$-modules, as can be seen from the decomposition as $B^{\otimes n}$-modules.
\end{proof}

For $i=1,\dotsc,\ell$, let $P_i$ be the projective cover of the $B^{\otimes n}$-module $V_i$ and, for $\blambda = (\lambda^1,\dotsc,\lambda^\ell) \in \fP$, define
\[
  P_\blambda := P_1^{\otimes |\lambda^1|} \otimes \dotsb \otimes P_\ell^{\otimes |\lambda^\ell|}.
\]
Since any $B^{\otimes n}$-endomorphism of $V_\blambda$ induces a $B^{\otimes n}$-endomorphism of $P_\blambda$, we have that $P_\blambda$ is naturally a $(B^{\otimes n},R_\blambda)$-bimodule.  Then $P_\blambda \otimes_{R_\blambda} L_\blambda$ carries a $B^{\otimes n} \rtimes S_\blambda$-action as in \eqref{eq:Bn-blambda-action}.

\begin{prop} \label{prop:induced-projectives}
  The
  \[
    \Ind_{B^{\otimes n} \rtimes S_\blambda}^{B^{\otimes n} \rtimes S_n} (P_\blambda \otimes_{R_\blambda} L_\blambda),\quad \blambda \in \fP(n),
  \]
  are a complete list of indecomposable projective $B^{\otimes n} \rtimes S_n$-modules up to isomorphism and shift.
\end{prop}

\begin{proof}
  First note that if $P'_j$ is the projective cover of a simple module $V'_j$ for an algebra $B_j$ for $j=1,2$, then $P'_1 \boxtimes P'_2$ is the projective cover of $V'_1 \boxtimes V'_2$ in the category of $B_1\otimes B_2$-modules.
  \details{
    We have short exact sequences
    \begin{gather*}
      0\to K_1\to P'_1\xrightarrow{p_1} V'_1\to 0, \\
      0\to K_2\to P'_2\xrightarrow{p_2} V'_2\to 0,
    \end{gather*}
    where $K_1$ and $K_2$ are superfluous submodules of $P'_1$ and $P'_2$ respectively.  In particular, for any submodule $M_j\subseteq P'_j$, either $M_j\subseteq K_j$ or $M_j=P'_j$, for $j=1,2$.  We then obtain another short exact sequence
    \[
      0\to K_1\boxtimes P'_2+P'_1\boxtimes K_2 \to P'_1\boxtimes P'_2 \xrightarrow{p_1\boxtimes p_2} V'_1\boxtimes V'_2\to 0.
    \]
    It is clear that $K_1\boxtimes P'_2 + P'_1 \boxtimes K_2 \subseteq \ker(p_1\boxtimes p_2)$, but since these are all vector spaces over $\mathbb{F}$, we have equality by a dimension count.  Now suppose that $M \subseteq P'_1\boxtimes P'_2$ is a submodule, then $M=\sum_r M_1^r\boxtimes M_2^r$ for some finite sets of submodules $M_j^r\subseteq P'_j$ for $j=1,2$.  If $M + K_1\boxtimes P'_2+P'_1\boxtimes K_2=P'_1\boxtimes P'_2$, then there is an $r_0$ such that $M_j^{r_0} \not \subseteq K_j$ for $j=1,2$. But this then implies that $M_j^{r_0}=P'_j$ for $j=1,2$, therefore $M=P'_1\boxtimes P'_2$. This proves that $K_1\boxtimes P'_2+P'_1\boxtimes K_2$ is a superfluous submodule.
  }
  This implies that, for each $\blambda \in \fP(n)$, $P_\blambda$ is the projective cover of $V_\blambda$ in the category of $B^{\otimes n}$-modules.  Let $p_\blambda \colon P_\blambda \twoheadrightarrow V_\blambda$ be the corresponding projection.

  It follows from the argument of \cite[(A.3)]{Mac80} that $P_\blambda \otimes_{R_\blambda} L_\blambda$ is projective as a $B^{\otimes n} \rtimes S_\blambda$-module.
  \details{
    More generally, if $P$ is a $B^{\otimes n} \rtimes S_n$-module that is projective as a $B^{\otimes n}$-module, then it is projective as a $B^{\otimes n} \rtimes S_n$-module.  Indeed, choosing a set of generators of $P$, we obtain an exact sequence
    \[
      0 \to Q \xrightarrow{\alpha} F \xrightarrow{\beta} P \to 0
    \]
    of $B^{\otimes n} \rtimes S_n$-modules, where $F$ is a free $B^{\otimes n} \rtimes S_n$-module.  Since $P$ is projective as an $B^{\otimes n}$-module, there exists a $B^{\otimes n}$-module homomorphism $\gamma \colon P \to F$ such that $\beta \gamma = \id_P$.  Define
    \[
      \gamma^* = \frac{1}{n!} \sum_{\tau \in S_n} \tau \gamma \tau^{-1}.
    \]
    Then, for $\sum_{\nu \in S_n} a_\nu \nu \in B^{\otimes n} \rtimes S_n$ and $x \in P$, we have
    \begin{gather*}
      \gamma^* \left( \sum_{\nu \in S_n} a_\nu \nu \cdot x \right)
      = \frac{1}{n!} \sum_{\tau,\nu \in S_n} \tau \gamma \tau^{-1} (a_\nu \nu \cdot x)
      = \frac{1}{n!} \sum_{\tau,\nu \in S_n} \tau \gamma ((\tau^{-1} \cdot a_\nu) (\tau^{-1} \nu) \cdot x) \\
      = \frac{1}{n!} \sum_{\tau,\nu \in S_n} \tau (\tau^{-1} \cdot a_\nu) \gamma ((\tau^{-1} \nu) \cdot x)
      = \frac{1}{n!} \sum_{\tau,\nu \in S_n} a_\nu \tau \gamma ((\tau^{-1} \nu) \cdot x) \\
      = \frac{1}{n!} \sum_{\tau,\nu \in S_n} a_\nu \nu \tau \gamma (\tau^{-1} \cdot x)
      = \left( \sum_{\nu \in S_n} a_\nu \nu \right) \cdot \gamma^*(x),
    \end{gather*}
    where in the second-to-last equality we replaced $\tau$ with $\nu \tau$ (which leaves the sum unchanged).
  }
  Since the kernel of $p_\blambda$ is a superflouous submodule of $P_\blambda$, we know that the kernel of $p_\blambda \otimes \id_{L_\blambda} \colon P_\blambda \otimes_{R_\blambda} L_\blambda \twoheadrightarrow V_\blambda \otimes_{R_\blambda} L_\blambda$ is a superfluous submodule of $P_\blambda \otimes_{R_\blambda} L_\blambda$. Hence $P_\blambda \otimes_{R_\blambda} L_\blambda$ is the projective cover of $V_\blambda \otimes_{R_\blambda} L_\blambda$ as $B^{\otimes n} \rtimes S_\blambda$-modules.

  Since $\F S_n$ is free as a right $\F S_\blambda$-module, we have that $\Ind_{B^{\otimes n} \rtimes S_\blambda}^{B^{\otimes n} \rtimes S_n} (P_\blambda \otimes_{R_\blambda} L_\blambda)$ is projective as a $B^{\otimes n} \rtimes S_n$-module.  In addition, the kernel of the map
  \[
    \id_{B^{\otimes n} \rtimes S_n} \otimes (p_\blambda \otimes \id_{L_\blambda}) \colon (B^{\otimes n} \rtimes S_n) \otimes_{B^{\otimes n} \rtimes S_\blambda} (P_\blambda \otimes_{R_\blambda} L_\blambda) \twoheadrightarrow (B^{\otimes n} \rtimes S_n) \otimes_{B^{\otimes n} \rtimes S_\blambda} (V_\blambda \otimes_{R_\blambda} L_\blambda)
  \]
  is superfluous, and so $\Ind_{B^{\otimes n} \rtimes S_\blambda}^{B^{\otimes n} \rtimes S_n}(P_\blambda \otimes_{R_\blambda} L_\blambda)$ is the projective cover of $\Ind_{B^{\otimes n} \rtimes S_\blambda}^{B^{\otimes n} \rtimes S_n}(V_\blambda \otimes_{R_\blambda} L_\blambda)$.  Thus, the result follows from Proposition~\ref{prop:wreath-simples}.
\end{proof}

We can also describe the indecomposable projectives of Proposition \ref{prop:induced-projectives} in terms of minimal idempotents.  For $i \in \{1,\dotsc,\ell\}$ and $\blambda \in \fP(n)$, choose minimal idempotents $e_i \in B$ and $e_\blambda' \in \fS_\blambda$ such that $P_i \cong Be_i$ and $L_\blambda \cong \fS_\blambda e_\blambda'$.  Then set
\begin{equation} \label{eq:e-blambda-def}
  e_\blambda = \left( e_1^{\otimes |\lambda^1|} \otimes \dotsb \otimes e_\ell^{\otimes |\lambda^\ell|} \right) e_\blambda'.
\end{equation}
Now, as $B^{\otimes n} \rtimes S_\blambda$-modules, we have
\begin{multline*}
  (B^{\otimes n} \otimes_\F \F S_\blambda)e_\blambda
  = \left( (B e_1)^{\otimes |\lambda^1|} \otimes \dotsb \otimes (B e_\ell)^{\otimes |\lambda^\ell|} \otimes_\F \F S_\blambda \right) e_\blambda'
  \cong \left( P_\blambda \otimes_{R_\blambda} R_\blambda \otimes_\F \F S_\blambda \right) e_\blambda' \\
  \cong \left( P_\blambda \otimes_{R_\blambda} \fS_\blambda \right) e_\blambda'
  \cong P_\blambda \otimes_{R_\blambda} L_\blambda.
\end{multline*}
Therefore
\begin{equation} \label{eq:wreath-module-idemponent}
  (B^{\otimes n}\rtimes S_n) e_\blambda
  \cong (B^{\otimes n}\rtimes S_n)\otimes_{B^{\otimes n}\rtimes S_{\blambda}}(B^{\otimes n} \otimes_\F \F S_{\blambda}) e_\blambda
  \cong  \Ind_{B^{\otimes n} \rtimes S_\blambda}^{B^{\otimes n} \rtimes S_n} (P_\blambda \otimes_{R_\blambda} L_\blambda).
\end{equation}

\section{The Heisenberg algebra $\fh_B$} \label{sec:Heis-alg}

In this section, we introduce the Heisenberg algebra $\fh_B$ that we will categorify in the remainder of the paper.  Our definition is in the spirit of the concept of a Heisenberg double.   We deduce in Proposition~\ref{prop:hB-presentation} an explicit presentation of $\fh_B$.  The reader not familiar with the general Heisenberg double construction can simply take this concrete presentation (which implicitly involves \eqref{eq:Q-even-reduction} and \eqref{eq:P-even-reduction}) as a definition.  Our reason for the more abstract definition of $\fh_B$ is that the Heisenberg double point of view allows one to deduce a uniform presentation that holds for all choices of $B$.  It also makes it clear how one could choose other generating sets, yielding other presentations of $\fh_B$.  As mentioned in the introduction, this in contrast to the approach in other places in the literature, where presentations are deduced on a case-by-case basis using generating function techniques.

\subsection{The Heisenberg algebra}

For $n \in \N$, let $K_0'(A_n)$ be the split Grothendieck group of the category $A_n\pmd$ of finitely generated projective $A_n$-modules.  More precisely, $K_0'(A_n)$ is defined to be the quotient of the free $\Z$-module with generators corresponding to finitely-generated projective $A_n$-modules, by the $\Z$-submodule generated by $M_1-M_2+M_3$ for all $M_1,M_2,M_3 \in A_n\pmd$ such that $M_2 \cong M_1 \oplus M_3$.  We denote the class of an object $M \in A_n\pmd$ in $K_0'(A_n)$ by $[M]$.

The split Grothendieck group $K_0'(A_n)$ is naturally a module over
\[
  \Z_{q,\pi} = \Z[q,q^{-1},\pi]/(\pi^2-1),
\]
where
\[
  \pi [M] := [\Pi M],\quad q^n[M] := [\{n\}M],\quad \text{for all } M \in B\pmd,\ n \in \Z.
\]
Let
\begin{equation} \label{eq:k-def}
  \kk =
  \begin{cases}
    \Z_{q,\pi} & \text{if all simple left $B$-modules are of type $M$}, \\
    \Z[\frac{1}{2},q,q^{-1}] & \text{if $B$ has a simple left module of type $Q$}.
  \end{cases}
\end{equation}
We will identify $\Z[\frac{1}{2},q,q^{-1}]$ with $\Z_{q,\pi}[\frac{1}{2}]/(\pi-1)$.  Thus, any $\kk$-module is naturally a $\Z_{q,\pi}$-module and $\pi$ acts as the identity in the case that $\kk = \Z[\frac{1}{2},q,q^{-1}]$.  Then, for $n \in \N$, define
\[
  K_0(A_n) = K_0'(A_n) \otimes_{\Z_{q,\pi}} \kk.
\]
We set
\[
  H^\pm = \bigoplus_{n \in \N} K_0(A_n),
\]
except that we twist the action of $\kk$ on $H^-$ by the involution of $\kk$ mapping $q$ to $q^{-1}$ and fixing $\pi$. In other words, we have $q^n[M] = [\{-n\}M]$ for $[M] \in H^-$.

It is shown in \cite[\S7]{RS15a} that $H^+$ and $H^-$ are Hopf algebras with product and coproduct
\begin{gather*}
  [M] \cdot [N] = \left[ \Ind_{A_m \otimes A_n}^{A_{m+n}}(M \boxtimes N) \right],\quad M \in A_m\pmd,\ N \in A_n\pmd, \\
  \nabla([M]) = \sum_{n=0}^m \left[ \Res^{A_m}_{A_n \otimes A_{m-n}}(M) \right],\quad M \in A_m\pmd.
\end{gather*}

For a $\Z \times \Z_2$-graded vector space $V$, define its graded dimension by
\begin{equation}
  \grdim V := \sum_{\substack{n \in \Z \\ \epsilon \in \Z_2}} q^n \pi^\epsilon \dim V_{n,\epsilon} \in \Z_{q,\pi}.
\end{equation}
We have a $\Z_{q,\pi}$-bilinear pairing $\langle -, - \rangle \colon H^- \otimes H^+ \to \Z_{q,\pi}$ given by
\begin{equation} \label{eq:bilinear-form}
  \langle [M], [N] \rangle =
  \begin{cases}
    \grdim \HOM_{A_n}(M,N) & \text{if } M,N \in A_n\pmd \text{ for some } n \in \N, \\
    0 & \text{otherwise}.
  \end{cases}
\end{equation}

\begin{defin}[Heisenberg algebra $\fh_B$] \label{def:hB}
  We define $\fh_B = H^+ \# H^-$ to be the $\kk$-algebra isomorphic to $H^+ \otimes_\kk H^-$ as a $\kk$-module, such that $H^\pm$ are subalgebras and where the product of elements of $H^-$ with elements of $H^+$ is given as follows: For $[M] \in H^-$ and $[N] \in H^+$, where $M \in A_m\pmd$, $N \in A_n\pmd$, $m,n \in \N$, we have
  \begin{gather*}
    [M] [N] = \sum_{k,\ell} \left\langle [M^1_k], [N^2_\ell] \right\rangle [N^1_\ell] [M^2_k],\quad \text{where} \\
    \Delta([M]) = \sum_k [M^1_k \boxtimes M^2_k] \quad \text{and} \quad \Delta([N]) = \sum_\ell [N^1_\ell \boxtimes N^2_\ell].
  \end{gather*}
  For $a \in H^+$ and $x \in H^-$, we denote the element $a \otimes x$ of $\fh_B$ by $a \# x$, or sometimes simply by $ax$.
\end{defin}

\begin{rem}[Relation to the Heisenberg double] \label{rem:Heis-double}
  The algebra $\fh_B$ is the projective Heisenberg double $\fh_\proj(A)$ associated to the tower of algebras $A = \bigoplus_{n \in \N} A_n$ in \cite[Def.~5.4]{RS15a}.  The fact that $A$ satisfies the hypotheses of \cite[Def.~5.4]{RS15a}, namely that the tower $A$ fulfills the axioms of \cite[Def.~4.1]{RS15a}, is proved in \cite[\S 7]{RS15a}.  We also use the fact that $H^+$, as defined here, is isomorphic, as a Hopf algebra, to the $H^+_\proj$ of \cite{RS15a}.  The adjective \emph{twisted} of \cite[Def.~5.4]{RS15a} is not needed in the current setting since the twisting is trivial for the tower of wreath product algebras (see \cite[Prop.~7.3]{RS15a}).  If we extend the ground ring to a field, then $\fh_B$ is the Heisenberg double of $H^+$.  This will allow us to use certain properties of the Heisenberg double, such as the faithfulness of the Fock space representation, in our proofs.
\end{rem}

\subsection{Presentations}

For $(i,\lambda) \in \{1,\dotsc,r\} \times \partition$ or $(i,\lambda) \in \{r+1,\dotsc,\ell\} \times \partition_\textup{strict}$, define
\[
  \lambda \langle i \rangle = (0,\dotsc,0,\lambda,0,\dotsc,0) \in \fP,
\]
where the $\lambda$ appears in the $i$-th position.  Then let
\[
  P_i^\lambda := [P_i^{\otimes |\lambda|} \otimes_{R_{\lambda \langle i \rangle}} L_{\lambda \langle i \rangle}] \in K_0(A_{|\lambda|}) \subseteq H^+,\quad
  Q_i^\lambda := [P_i^{\otimes |\lambda|} \otimes_{R_{\lambda \langle i \rangle}} L_{\lambda \langle i \rangle}] \in K_0(A_{|\lambda|}) \subseteq H^-.
\]
By \eqref{eq:wreath-module-idemponent}, we have an isomorphism of $A_n$-modules
\[
  P_i ^{\otimes |\lambda|} \otimes_{R_{\lambda \langle i \rangle}} L_{\lambda \langle i \rangle} \cong A_n e_{\lambda \langle i \rangle}.
\]

For the Sergeev algebra $\mathbb{S}_n$, the minimal idempotents (up to a constant) are described in \cite[Cor.~3.3.4]{Ser99}. In particular, for the strict partition $(n)$, the idempotent is just the complete symmetrizer $e_{(n)} = \frac{1}{n!}\sum_{\tau \in S_n} \tau$, as explained in \cite[\S3.3.5]{Ser99}.  Now, if $e_{(1^n)} = \frac{1}{n!}\sum_{\tau \in S_n} (-1)^{\ell(\tau)} \tau$ is the complete antisymmetrizer, then the map $\pr{(c^{\otimes n})}$ gives an isomorphism $\bS_n e_{(n)} \cong \{0,n\} \bS_n e_{(1^n)}$.  Thus, even though the partition $(1^n)$ is not strict for $n > 1$, we define
\begin{equation} \label{eq:1n-type-Q-def}
  P_i^{(1^n)} := P_i^{(n)},\quad Q_i^{(1^n)} := Q_i^{(n)},\quad i \in \{r+1,\dotsc,\ell\},\ n \in \N_+.
\end{equation}

For $i \in \{1,\dotsc,\ell\}$, define
\begin{equation}
  N_i :=
  \begin{cases}
    \N_+, & 1 \le i \le r, \\
    2\N + 1, & r < i \le n.
  \end{cases}
\end{equation}

\begin{lem} \label{lem:Heis-double-generators}
  The subalgebra $H^+$ of $\fh_B$ is isomorphic, as an algebra, to the polynomial algebras
  \[
    \kk [P_i^{(n)} \mid i \in \{1,\dotsc,\ell\},\ n \in N_i] \quad \text{and} \quad
    \kk [P_i^{(1^n)} \mid i \in \{1,\dotsc,\ell\},\ n \in N_i].
  \]
  Similarly, the subalgebra $H^-$ of $\fh_B$ is isomorphic, as an algebra, to the polynomial algebras
  \[
    \kk [Q_i^{(n)} \mid i \in \{1,\dotsc,\ell\},\ n \in N_i] \quad \text{and} \quad
    \kk [Q_i^{(1^n)} \mid i \in \{1,\dotsc,\ell\},\ n \in N_i].
  \]
  In particular, $\fh_B$ is generated by the following sets:
  \begin{itemize}
    \item $P_i^{(n)}$, $Q_i^{(n)}$, $i \in \{1,\dotsc,\ell\}$, $n \in N_i$,
    \item $P_i^{(1^n)}$, $Q_i^{(1^n)}$, $i \in \{1,\dotsc,\ell\}$, $n \in N_i$,
    \item $P_i^{(n)}$, $Q_i^{(1^n)}$, $i \in \{1,\dotsc,\ell\}$, $n \in N_i$,
    \item $P_i^{(1^n)}$, $Q_i^{(n)}$, $i \in \{1,\dotsc,\ell\}$, $n \in N_i$.
  \end{itemize}
\end{lem}

\begin{proof}
  We prove the first assertion, since the proof of the second is identical.  By \cite[Lem.~7.5]{RS15a}, the subalgebra $H^+$ is commutative.  It follows from Proposition~\ref{prop:induced-projectives} that $H^+$ has a $\kk$-basis given by
  \[
    [P_1^{\otimes |\lambda_1|} \otimes_{R_{\lambda^1 \langle 1 \rangle}} L_{\lambda^1 \langle 1 \rangle}] \dotsm [P_\ell^{\otimes |\lambda_\ell|} \otimes_{R_{\lambda^\ell \langle \ell \rangle}} L_{\lambda^\ell \langle \ell \rangle}],\quad (\lambda^1,\dotsc,\lambda^\ell) \in \fP.
  \]
  Note that it is crucial here that we have inverted $2$ in $\kk$ when $B$ has a module of type $Q$, since the $\fS_{(\lambda^1,\dotsc,\lambda^\ell)}$-module $L_{\blambda^1 \langle 1 \rangle} \boxtimes \dotsb \boxtimes L_{\blambda^\ell \langle \ell \rangle}$ is, in general, a copy of $2^m$ copies of $L_{(\blambda^1,\dotsc,\blambda^\ell)}$ and its shift, for some $m \in \N$.

  Now suppose $i \in \{1,\dotsc,r\}$.  Then the reasoning used in the proof of \cite[Lem.~5]{CL12} shows that, in $H^+$, we have, for $\lambda,\mu \in \partition$,
  \[
    \left[ P_i^{|\lambda|} \otimes_{R_{\lambda \langle i \rangle}} L_{\lambda \langle i \rangle} \right] \cdot \left[ P_i^{|\mu|} \otimes_{R_{\mu \langle i \rangle}} L_{\mu \langle i \rangle} \right]
    = \left[ A_{|\lambda|} e_{\lambda \langle i \rangle} \right] \cdot \left[ A_{|\mu|} e_{\mu \langle i \rangle} \right]
    = \sum_{\nu \in \partition} \left[A_{|\nu|} e_{\nu \langle i \rangle} \right]^{\oplus c^\nu_{\lambda,\mu}}
    = \sum_{\nu \in \partition} \left[ P_i^{|\nu|} \otimes_{R_{\nu \langle i \rangle}} L_{\nu \langle i \rangle} \right]^{\oplus c^\nu_{\lambda,\mu}},
  \]
  where $c^\nu_{\lambda,\mu}$ are the Littlewood-Richardson coefficients.  It follows that the subalgebra of $H^+$ spanned by the $[P_i^{|\lambda|} \otimes_{R_{\lambda \langle i \rangle}} L_{\lambda \langle i \rangle}]$, $\lambda \in \partition$, is isomorphic to the ring of symmetric functions, with $[P_i^{|\lambda|} \otimes_{R_{\lambda \langle i \rangle}} L_{\lambda \langle i \rangle}]$ corresponding to the Schur function $s_\lambda$.  Since the ring of symmetric functions is a polynomial algebra in the complete homogeneous symmetric functions $s_{(n)}$, $n \in \N_+$, and also a polynomial algebra in the elementary symmetric functions $s_{(1^n)}$, $n \in \N_+$, this subalgebra of $H^+$ is a polynomial algebra in the $P_i^{(n)}$ or the $P_i^{(1^n)}$.

  Now suppose $i \in \{r+1,\dotsc,\ell\}$.  Here the argument is similar, but with the representation theory of the symmetric group replaced by the representation theory of the Sergeev algebra.   In this case, $[P_i^{|\lambda|} \otimes_{R_{\lambda \langle i \rangle}} L_{\lambda \langle i \rangle}]$, $\lambda \in \partition_\textup{strict}$, corresponds to $2^{- \lfloor \ell(\lambda)/2 \rfloor} Q_\lambda$, where $Q_\lambda$ is the Schur $Q$-function.  (See \cite[\S 4C]{Joz89} or the more recent \cite[(3.14)]{WW12}.)  Now, the space spanned over $\kk$ by the Schur $Q$-functions is a polynomial algebra in the $Q_{(n)}$, $n \in 2 \N + 1$.   See, for example, \cite[III, (8.5)]{Mac95}, and note that while that reference states the result over $\Q$, the key relation \cite[III, (8.2')]{Mac95} only requires that 2 be invertible.  It follows that the subalgebra of $H^+$ spanned by the $[P_i^{|\lambda|} \otimes_{R_{\lambda \langle i \rangle}} L_{\lambda \langle i \rangle}]$, $\lambda \in \partition_\textup{strict}$, is a polynomial algebra in the $P_i^{(n)} = P_i^{(1^n)}$, $n \in 2 \N + 1$.
\end{proof}

\begin{rem} \label{rem:even-type-Q}
  It follows from the proof of Lemma~\ref{lem:Heis-double-generators} that, for $i \in \{r+1,\dotsc, \ell\}$ and $m \in \N_+$, we have
  \begin{gather}
    Q_i^{(2m)} = \frac{1}{2} (-1)^{m-1} \left( Q_i^{(m)} \right)^2 + \sum_{r=1}^{m-1} (-1)^{r-1} Q_i^{(r)} Q_i^{(2m-r)}, \label{eq:Q-even-reduction} \\
    P_i^{(2m)} = \frac{1}{2} (-1)^{m-1} \left( P_i^{(m)} \right)^2 + \sum_{r=1}^{m-1} (-1)^{r-1} P_i^{(r)} P_i^{(2m-r)}, \label{eq:P-even-reduction}
  \end{gather}
  since the corresponding relations hold for the Schur Q-functions.  See, for example, \cite[III, (8.2')]{Mac95}.  This allows one to recursively express $P_i^{(2m)}$, $m \in \N_+$, in terms of $P_i^{(n)}$ for odd $n$ and similarly for $Q_i^{(2m)}$.  From now on, we will in this way view $P_i^{(2m)}$ and $Q_i^{(2m)}$, $m \in \N_+$, as elements of the polynomial algebras $\kk[P_i^{(n)} \mid n \in N_i]$ and $\kk[Q_i^{(n)} \mid n \in N_i]$, respectively.
\end{rem}

For a $\Z \times \Z_2$-graded vector space $V$ and $k \in \N_+$, define the symmetric and exterior algebras
\begin{gather}
  S^k(V) = V^{\otimes k}/\Span_\F \{\mathbf{v} - \tau \cdot \mathbf{v} \mid \tau \in S_n,\ \mathbf{v} \in V^{\otimes k}\}, \label{eq:SkV-def} \\
  \Lambda^k(V) = V^{\otimes k}/\Span_\F \{\mathbf{v} -(-1)^{\ell(\tau)} \tau \cdot \mathbf{v} \mid \tau \in S_n,\ \mathbf{v} \in V^{\otimes k}\}. \label{eq:LambdakV-def}
\end{gather}
Note that
\begin{equation} \label{eq:symmetric-exterior-relation}
  S^k(V) = \Lambda^k(\Pi V).
\end{equation}
The graded dimensions of $S^k(V)$ and $\Lambda^k(V)$ can be explicitly computed.  In particular, we have
\begin{equation}
  \sum_{k \in \N} t^k \grdim S^k(V)
  = \prod_{n \in \Z} \left( \frac{1}{1-q^nt} \right)^{\dim V_{n,0}} \prod_{n \in \Z} (1+\pi q^n t)^{\dim V_{n,1}} \in \Z_{q,\pi} \llbracket t \rrbracket.
\end{equation}
Then \eqref{eq:symmetric-exterior-relation} allows us to obtain a similar expression with $S^k(V)$ replaced by $\Lambda^k(V)$.

\begin{prop}[Presentations of $\fh_B$] \label{prop:hB-presentation}
  Recall that, for $i \in \{r+1,\dotsc,\ell\}$ and $n \in 2 \N_+$, $P_i^{(n)}$ and $Q_i^{(n)}$ correspond to polynomials in $P_i^{(m)}$ and $Q_i^{(m)}$, $m \in 2 \N + 1$ (see Remark~\ref{rem:even-type-Q}).  We have the following presentations of $\fh_B$.
  \begin{asparaenum}
    \item \label{lem-item:hB-presentation-nn} The algebra $\fh_B$ is isomorphic, as a $\kk$-module, to the product of polynomial algebras
      \[
        \kk [P_i^{(n)} \mid i=1,\dotsc,\ell,\ n \in N_i] \otimes_\kk \kk [Q_i^{(n)} \mid i=1,\dotsc,\ell,\ n \in N_i].
      \]
      The two polynomial algebras are subalgebras of $\fh_B$, and the relations between the two factors are
      \[
        Q_i^{(n)} P_j^{(m)} = \sum_{k \in \N} \left( \grdim S^k\left(\HOM_B(P_i,P_j)\right) \right) P_j^{(m-k)} Q_i^{(n-k)}.
      \]

    \item \label{lem-item:hB-presentation-1n1n} The algebra $\fh_B$ is isomorphic, as a $\kk$-module, to the product of polynomial algebras
      \[
        \kk [P_i^{(1^n)} \mid i=1,\dotsc,\ell,\ n \in N_i] \otimes_\kk \kk [Q_i^{(1^n)} \mid i=1,\dotsc,\ell,\ n \in N_i].
      \]
      The two polynomial algebras are subalgebras of $\fh_B$, and the relations between the two factors are
      \[
        Q_i^{(1^n)} P_j^{(1^m)} = \sum_{k \in \N} \left( \grdim S^k\left(\HOM_B(P_i,P_j)\right) \right) P_j^{(1^{m-k})} Q_i^{(1^{n-k})}.
      \]

    \item \label{lem-item:hB-presentation-n1n} The algebra $\fh_B$ is isomorphic, as a $\kk$-module, to the product of polynomial algebras
      \[
        \kk [P_i^{(n)} \mid i=1,\dotsc,\ell,\ n \in N_i] \otimes_\kk \kk [Q_i^{(1^n)} \mid i=1,\dotsc,\ell,\ n \in N_i].
      \]
      The two polynomial algebras are subalgebras of $\fh_B$, and the relations between the two factors are
      \[
        Q_i^{(1^n)} P_j^{(m)} = \sum_{k \in \N} \left( \grdim \Lambda^k\left(\HOM_B(P_i,P_j)\right) \right) P_j^{(m-k)} Q_i^{(1^{n-k})}.
      \]

      \item \label{lem-item:hB-presentation-1nn} The algebra $\fh_B$ is isomorphic, as a $\kk$-module, to the product of polynomial algebras
      \[
        \kk [P_i^{(1^n)} \mid i=1,\dotsc,\ell,\ n \in N_i] \otimes_\kk \kk [Q_i^{(n)} \mid i=1,\dotsc,\ell,\ n \in N_i].
      \]
      The two polynomial algebras are subalgebras of $\fh_B$, and the relations between the two factors are
      \[
        Q_i^{(n)} P_j^{(1^m)} = \sum_{k \in \N} \left( \grdim \Lambda^k\left(\HOM_B(P_i,P_j)\right) \right) P_j^{(1^{m-k})} Q_i^{(n-k)}.
      \]
  \end{asparaenum}
\end{prop}

\begin{proof}
  We prove part~\eqref{lem-item:hB-presentation-nn}, since the proofs of the other parts are analogous.  By Lemma~\ref{lem:Heis-double-generators}, it suffices to compute the commutation relations between the $Q_i^{(n)}$ and $P_j^{(m)}$.  Note that, for $i \in \{1,\dotsc,\ell\}$, $n \in \N_+$, and $\lambda = (n)$, we have
  \[
    e_{\lambda \langle i \rangle} = \left( e_i \otimes \dotsb \otimes e_i \right) \frac{1}{n!} \sum_{\tau \in S_n} \tau = \frac{1}{n!} \sum_{\tau \in S_n} \tau (e_i \otimes \dotsb \otimes e_i) \in A_n.
  \]
  Then, as left $A_{n-k} \otimes A_k$-modules, we have
  \[
    P_i^n \otimes_{R_{(n) \langle i \rangle}} L_{(n) \langle i \rangle} = A_n e_{(n) \langle i \rangle} = (Be_i)^{\otimes n} = (Be_i)^{\otimes (n-k)} \boxtimes (Be_i)^{\otimes k}.
  \]
  Therefore,
  \[
    \Res^{A_n}_{A_{n-k} \otimes A_k} \left( P_i^{\otimes n} \otimes_{(n) \langle i \rangle} L_{(n) \langle i \rangle} \right) = \left( P_i^{\otimes (n-k)} \otimes_{(n-k) \langle i \rangle} L_{(n-k) \langle i \rangle} \right) \boxtimes \left( P_i^{\otimes k} \otimes_{(k) \langle i \rangle} L_{(k) \langle i \rangle} \right).
  \]
  Thus, in the Hopf algebras $H^+$ and $H^-$, we have
  \[
    \Delta \left( P_i^{(n)} \right) = \sum_{k=0}^n P_i^{(n-k)} \boxtimes P_i^{(k)} \quad \text{and} \quad
    \Delta \left( Q_i^{(n)} \right) = \sum_{k=0}^n Q_i^{(n-k)} \boxtimes Q_i^{(k)}.
  \]
  Therefore, by Definition~\ref{def:hB}, we have
  \[
    Q_i^{(n)} P_j^{(m)}
    = \sum_{k=0}^n \sum_{\ell=0}^m \left\langle Q_i^{(k)}, P_j^{(\ell)} \right\rangle P_j^{(m-\ell)} Q_i^{(n-k)} \\
    = \sum_{k=0}^n \grdim \HOM(Q_i^{(k)}, P_j^{(k)}) P_j^{(m-k)} Q_i^{(n-k)}.
  \]

  We have isomorphisms of graded vector spaces
  \begin{multline} \label{eq:hom-symmetrizer}
    \HOM \left( Q_i^{(k)}, P_j^{(k)} \right)
    \cong e_{(k) \langle i \rangle} A_k e_{(k) \langle j \rangle}
    = \left\{ \sum_{\tau_1,\tau_2 \in S_n} \tau_1 (e_i b_1 e_j \otimes e_i b_2 e_j \otimes \dotsb \otimes e_i b_k e_j ) \tau_2 \mid b_1,\dotsc, b_k \in B \right\} \\
    = \left( e_i B e_j \right)^{S_k} \sum_{\tau \in S_n} \tau
    \cong \left( e_i B e_j \right)^{S_k}
    \cong S^k\left( e_i B e_j \right)
    \cong S^k \left( \HOM_B(Be_i, Be_j) \right)
    \cong S^k \left( \HOM_B(P_i,P_j) \right),
  \end{multline}
  where $(e_i B e_j)^{S_k}$ denotes the space of $S_k$-invariants of $e_i B e_j$ and, in the third-to-last isomorphism, we use the fact that $\F[S_k]$ is semisimple and so projection onto the trivial isotypic component is isomorphic to the quotient by the other isotypic components.
  \details{
    Proof of part~\eqref{lem-item:hB-presentation-1n1n}: Again, by Lemma~\ref{lem:Heis-double-generators}, it suffices to compute the commutation relations between the $Q_i^{(1^n)}$ and $P_j^{(1^m)}$.  And by \eqref{eq:1n-type-Q-def}, it suffices to assume $1 \le i,j \le r$.  Note that, for $i \in \{1,\dotsc,r\}$ and $n \in \N_+$, we have
    \[
      e_{(1^n) \langle i \rangle} = \left( e_i \otimes \dotsb \otimes e_i \right) \frac{1}{n!} \sum_{\tau \in S_n} (-1)^{\ell(\tau)} \tau = \frac{1}{n!} \sum_{\tau \in S_n} (-1)^{\ell(\tau)} \tau (e_i \otimes \dotsb \otimes e_i) \in A_n.
    \]
    Since $L_{(1^m)}$ is the sign representation of $S_m$ for $m \in \N_+$, it is clear that, for $0 \le k \le n$,
    \[
      \Res^{A_n}_{A_{n-k} \otimes A_k} \left( P_i^{\otimes n} \otimes_{R_{(1^n) \langle i \rangle}} L_{(1^n) \langle i \rangle} \right) = \left( P_i^{\otimes (n-k)} \otimes_{R_{(1^{n-k}) \langle i \rangle}} L_{(1^{n-k}) \langle i \rangle} \right) \boxtimes \left( P_i^{\otimes k} \otimes_{R_{(1^k) \langle i \rangle}} L_{(1^k) \langle i \rangle} \right).
    \]
    Thus, in the Hopf algebras $H^+$ and $H^-$, we have
    \[
      \Delta \left( P_i^{(1^n)} \right) = \sum_{k=0}^n P_i^{(1^{n-k})} \boxtimes P_i^{(1^k)} \quad \text{and} \quad \Delta \left( Q_i^{(1^n)} \right) = \sum_{k=0}^n Q_i^{(1^{n-k})} \boxtimes Q_i^{(1^k)}.
    \]
    Therefore, by Definition~\ref{def:hB}, we have
    \[
      Q_i^{(1^n)} P_j^{(1^m)}
      = \sum_{k=0}^n \sum_{\ell=0}^m \left\langle Q_i^{(1^k)}, P_j^{(1^\ell)} \right\rangle P_j^{(1^{m-\ell})} Q_i^{(1^{n-k})} \\
      = \sum_{k=0}^n \grdim \HOM \left( Q_i^{(1^k)}, P_j^{(1^k)} \right) P_j^{(1^{m-k})} Q_i^{(1^{n-k})}.
    \]
    We have isomorphisms of graded vector spaces
    \begin{multline*}
      \HOM \left( Q_i^{(1^k)}, P_j^{(1^k)} \right)
      \cong e_{(1^k) \langle i \rangle} A_k e_{(1^k) \langle j \rangle} \\
      = \left\{ \sum_{\tau_1,\tau_2 \in S_n} (-1)^{\ell(\tau_1) + \ell(\tau_2)} \tau_1 (e_i b_1 e_j \otimes e_i b_2 e_j \otimes \dotsb \otimes e_i b_k e_j ) \tau_2 \mid b_1,\dotsc, b_k \in B \right\} \\
      = \left\{ \sum_{\tau_1,\tau_2 \in S_n} (-1)^{\ell(\tau_1) + \ell(\tau_2)} \Big( \tau_1 \cdot (e_i b_1 e_j \otimes e_i b_2 e_j \otimes \dotsb \otimes e_i b_k e_j ) \Big) \tau_1 \tau_2 \mid b_1,\dotsc, b_k \in B \right\} \\
      = \left\{ \sum_{\tau_1,\tau \in S_n} (-1)^{\ell(\tau)} \Big( \tau_1 \cdot (e_i b_1 e_j \otimes e_i b_2 e_j \otimes \dotsb \otimes e_i b_k e_j ) \Big) \tau \mid b_1,\dotsc, b_k \in B \right\} \\
      = \left( e_i B e_j \right)^{S_k} \sum_{\tau \in S_n} (-1)^{\ell(\tau)} \tau
      \cong \left( e_i B e_j \right)^{S_k}
      \cong S^k\left( e_i B e_j \right)
      \cong S^k \left( \HOM_B(Be_i, Be_j) \right)
      \cong S^k \left( \HOM_B(P_i,P_j) \right).
    \end{multline*}

    Proof of part~\eqref{lem-item:hB-presentation-n1n}: By Definition~\ref{def:hB}, we have
    \[
      Q_i^{(1^n)} P_j^{(1^m)}
      = \sum_{k=0}^n \sum_{\ell=0}^m \left\langle Q_i^{(1^k)}, P_j^{(1^\ell)} \right\rangle P_j^{(1^{m-\ell})} Q_i^{(1^{n-k})} \\
      = \sum_{k=0}^n \grdim \HOM \left( Q_i^{(1^k)}, P_j^{(1^k)} \right) P_j^{(1^{m-k})} Q_i^{(1^{n-k})}.
    \]
    We have isomorphisms of graded vector spaces
    \begin{multline*}
      \HOM \left( Q_i^{(1^k)}, P_j^{(k)} \right)
      \cong e_{(1^k) \langle i \rangle} A_k e_{(k) \langle j \rangle} \\
      = \left\{ \sum_{\tau_1,\tau_2 \in S_n} (-1)^{\ell(\tau_1)} \tau_1 (e_i b_1 e_j \otimes e_i b_2 e_j \otimes \dotsb \otimes e_i b_k e_j ) \tau_2 \mid b_1,\dotsc, b_k \in B \right\} \\
      = \left( e_i B e_j \right)^\text{sign} \sum_{\tau \in S_n} \tau
      \cong \left( e_i B e_j \right)^\text{sign}
      \cong \Lambda^k \left( e_i B e_j \right)
      \cong \Lambda^k \left( \HOM_B(Be_i, Be_j) \right)
      \cong \Lambda^k \left( \HOM_B(P_i,P_j) \right),
    \end{multline*}
    where $\left( e_i B e_j \right)^\text{sign}$ denotes the isotypic component of $e_i B e_j$ corresponding to the sign representation of $S_n$ and we again use the fact that $\F[S_k]$ is semisimple and so projection onto the sign isotypic component is isomorphic to the quotient by the other isotypic components.

    The proof of part~\eqref{lem-item:hB-presentation-1nn} is almost identical to the proof of part~\eqref{lem-item:hB-presentation-n1n}.
  }
\end{proof}

\begin{rem} \label{rem:lattice-Heisenberg}
  The algebra $\fh_B$ is the \emph{(quantum) lattice Heisenberg algebra} associated to the lattice $K_0(B)$ with bilinear form given by $\langle [P_i], [P_j] \rangle = \grdim \HOM_B(P_i,P_j)$, $i \in \{1,\dotsc,\ell\}$.  See, for example, \cite[\S1.2]{LS12} and \cite[\S7, \S8]{RS15b} for a discussion of (quantum) lattice Heisenberg algebras.  Over a field and when the form is nondegenerate, the (quantum) lattice Heisenberg algebras are all isomorphic to the usual infinite rank Heisenberg algebra over the same field.  In other words, the lattice Heisenberg algebras are integral (more precisely, $\kk$) forms of the usual infinite rank Heisenberg algebra (see \cite[Prop.~8.3]{RS15b}).  Proposition~\ref{prop:hB-presentation} gives a general form of the presentations of these integral forms that appear to be natural from the point of view of categorification.  In special cases, it recovers presentations appearing in the literature and deduced using different methods.  For example, the presentations of \cite[\S1]{Kho14} and \cite[(1)--(5), \S2.2.2]{CL12} are special cases of those of Proposition~\ref{prop:hB-presentation}.  In the case that all gradings are trivial, the relations of Proposition~\ref{prop:hB-presentation} are also closely related to those of \cite[Lem.~1.2 and \S3.2]{Kru15}.
\end{rem}

\begin{rem} \label{rem:missing-relations}
  Proposition~\ref{prop:hB-presentation} does \emph{not} specialize to the presentations of \cite[Prop.~5.1]{CS14} or \cite[Prop.~1]{HS15}.  This is because the presentations in those papers appear to be missing relations.  For example, using \cite[(20)]{CS14}, one can compute that $p^{(2)} = 2 h_{-1/2}^2 = \frac{1}{2} \left( p^{(1)} \right)^2$ in the notation of that paper, but this relation does not follow from the ones given in \cite[Prop.~5.1]{CS14}.  In fact, the relations that are missing there are precisely \eqref{eq:Q-even-reduction} and \eqref{eq:P-even-reduction}.  Those relations are also missing in \cite[Prop.~1]{HS15}.\footnote{After this was brought to their attention, the authors of \cite{HS15} corrected this omission in the published version of their paper.}
\end{rem}

\begin{cor} \label{cor:heis-alg-involution}
  There is an algebra involution $\omega$ of $\fh_B$ uniquely determined by
  \[
    Q_i^{(n)} \mapsto Q_i^{(1^n)},\ P_i^{(n)} \mapsto P_i^{(1^n)},\quad i \in \{1,\dotsc,\ell\},\ n \in \N_i.
  \]
\end{cor}

\section{The category $\cH_B$} \label{sec:cat-def}

In this section we define an additive $\F$-linear monoidal category $\cH_B'$ whose idempotent completion $\cH_B$ categorifies the algebra $\fh_B$.

The objects of $\cH'_B$ are generated by symbols $\{n,\epsilon\} \sP$ and $\{n,\epsilon\} \sQ$, for $n\in \Z$, $\epsilon \in \Z_2$. We think of $\{n,\epsilon\}\sP$ as being a shifted version of $\sP$ and we declare the monoidal structure to be compatible with shifts $\{ \cdot, \cdot \}$, so that, for example, $\{s,\epsilon\} \sQ\otimes \{s', \epsilon'\} \sQ = \{s+s',\epsilon+\epsilon'\} (\sQ\otimes \sQ)$.  We will usually omit the $\otimes$ symbol, and write tensor products as words in $\sP$ and $\sQ$.  Thus an arbitrary object of $\cH'_B$ is a finite direct sum of words in the letters $\sP$ and $\sQ$ where each word has a shift.

The space of morphisms between two objects is the $\F$-algebra generated by suitable planar diagrams modulo local relations.  The diagrams consist of oriented compact one-manifolds immersed into the plane strip $\R \times [0,1]$ modulo certain relations.  The grading on morphisms, which will be specified later in this section, determines the difference in shifts between the domain and codomain.

A single upward oriented strand denotes the identity morphism from $\sP$ to $\sP$ while a downward oriented strand denotes the identity morphism from $\sQ$ to $\sQ$.
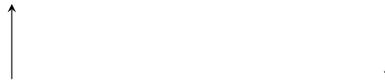
\[
  \begin{tikzpicture}[>=stealth,baseline={([yshift=-.5ex]current bounding box.center)}]
    \draw[->] (0,0) -- (0,1);
    \draw[<-] (5,0) -- (5,1);
  \end{tikzpicture}
\]

Strands are allowed to carry dots labeled by elements of $B$. For example, if $b, b', b'' \in B$, then the diagram
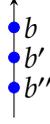
\[
  \begin{tikzpicture}[>=stealth,baseline={([yshift=-.5ex]current bounding box.center)}]
    \draw [->](0,0) -- (0,1.6);
    \filldraw [blue](0,.4) circle (2pt);
    \draw (0,.4) node [anchor=west] [black] {$b''$};
    \filldraw [blue](0,.8) circle (2pt);
    \draw (0,.8) node [anchor=west] [black] {$b'$};
    \filldraw [blue](0,1.2) circle (2pt);
    \draw (0,1.2) node [anchor=west] [black] {$b$};
  \end{tikzpicture}
\]
is an element of $\Hom_{\cH'_B}(\sP,\{-|b|-|b'|-|b''|,\bar b + \bar b' + \bar b''\}\sP)$.  (See below for an explanation of the degree shift.)  Diagrams are linear in the dots in the sense that
\[
  \begin{tikzpicture}[>=stealth,baseline={([yshift=-.5ex]current bounding box.center)}]
    \draw [->](0,0) -- (0,1);
    \bluedot{(0,.5)} node [anchor=west,color=black] {$(z_1b_1 + z_2 b_2)$};
  \end{tikzpicture}
  \ =\ z_1 \left(\
  \begin{tikzpicture}[>=stealth,baseline={([yshift=-.5ex]current bounding box.center)}]
    \draw [->](0,0) -- (0,1);
    \bluedot{(0,.5)} node [anchor=west, color=black] {$b_1$};
  \end{tikzpicture} \right)
  \ +\ z_2 \left(\
  \begin{tikzpicture}[>=stealth,baseline={([yshift=-.5ex]current bounding box.center)}]
    \draw [->](0,0) -- (0,1);
    \bluedot{(0,.5)} node [anchor=west, color=black] {$b_2$};
  \end{tikzpicture} \right)
  \quad \text{for } z_1,z_2 \in \F,\ b_1,b_2 \in B.
\]

Collision of dots is controlled by multiplication in the algebra $B$:

\noindent\begin{minipage}{0.5\linewidth}
  \begin{equation} \label{eq:collision-up}
    \begin{tikzpicture}[>=stealth,baseline={([yshift=-.5ex]current bounding box.center)}]
      \draw[->] (0,0) -- (0,1.5);
      \bluedot{(0,.5)};
      \draw (0,.5) node [anchor=west] [black] {$b'$};
      \bluedot{(0,1)};
      \draw (0,1) node [anchor=west] [black] {$b$};
    \end{tikzpicture}
    = (-1)^{\bar b \bar b'}\
    \begin{tikzpicture}[>=stealth,baseline={([yshift=-.5ex]current bounding box.center)}]
      \draw[->] (0,0) -- (0,1.5);
      \bluedot{(0,.75)};
      \draw (0,.75) node [anchor=west,color=black] {$b'b$};
    \end{tikzpicture}
  \end{equation}
\end{minipage}%
\begin{minipage}{0.5\linewidth}
  \begin{equation} \label{eq:collision-down}
    \begin{tikzpicture}[>=stealth,baseline={([yshift=-.5ex]current bounding box.center)}]
      \draw[->] (0,1.5) -- (0,0);
      \bluedot{(0,.5)};
      \draw (0,.5) node [anchor=west] [black] {$b'$};
      \bluedot{(0,1)};
      \draw (0,1) node [anchor=west] [black] {$b$};
    \end{tikzpicture}
    \ =\
    \begin{tikzpicture}[>=stealth,baseline={([yshift=-.5ex]current bounding box.center)}]
      \draw[->] (0,1.5) -- (0,0);
      \bluedot{(0,.75)};
      \draw (0,.75) node [anchor=west,color=black] {$bb'$};
    \end{tikzpicture}
  \end{equation}
\end{minipage}\par\vspace{\belowdisplayskip}
\noindent The sign in the first relation comes from the fact that composition of dots on upward strands actually corresponds to multiplication in the opposite algebra $B^\op$.

When dots move along cups or caps, in some cases we apply the Nakayama automorphism:

\noindent\begin{minipage}{0.5\linewidth}
  \begin{equation} \label{eq:dotsup-a}
    \begin{tikzpicture}[>=stealth,baseline={([yshift=-.5ex]current bounding box.center)}]
      \draw[->] (1,0.5) -- (1,0) arc (360:180:0.5) -- (0,0.5);
      \bluedot{(0,0)} node [anchor=east,color=black] {$\psiB(b)$};
    \end{tikzpicture}
    \ =\
    \begin{tikzpicture}[>=stealth,baseline={([yshift=-.5ex]current bounding box.center)}]
      \draw[->] (1,0.5) -- (1,0) arc (360:180:0.5) -- (0,0.5);
      \bluedot{(1,0)} node [anchor=west,color=black] {$b$};
    \end{tikzpicture}
  \end{equation}
\end{minipage}%
\begin{minipage}{0.5\linewidth}
  \begin{equation} \label{eq:dotsup-b}
    \begin{tikzpicture}[>=stealth,baseline={([yshift=-.5ex]current bounding box.center)}]
      \draw[->] (0,-0.5) -- (0,0) arc (180:0:0.5) -- (1,-0.5);
      \bluedot{(0,0)} node [anchor=east,color=black] {$b$};
    \end{tikzpicture}
    \ =\
    \begin{tikzpicture}[>=stealth,baseline={([yshift=-.5ex]current bounding box.center)}]
      \draw[->] (0,-0.5) -- (0,0) arc (180:0:0.5) -- (1,-0.5);
      \bluedot{(1,0)} node [anchor=west,color=black] {$b$};
    \end{tikzpicture}
  \end{equation}
\end{minipage}\par\vspace{\belowdisplayskip}
\noindent and

\noindent\begin{minipage}{0.5\linewidth}
  \begin{equation} \label{eq:dotsdown-a}
    \begin{tikzpicture}[>=stealth,baseline={([yshift=-.5ex]current bounding box.center)}]
      \draw[<-] (1,0.5) -- (1,0) arc (360:180:0.5) -- (0,0.5);
      \bluedot{(0,0)} node [anchor=east,color=black] {$b$};
    \end{tikzpicture}
    \ =\
    \begin{tikzpicture}[>=stealth,baseline={([yshift=-.5ex]current bounding box.center)}]
      \draw[<-] (1,0.5) -- (1,0) arc (360:180:0.5) -- (0,0.5);
      \bluedot{(1,0)} node [anchor=west,color=black] {$b$};
    \end{tikzpicture}
  \end{equation}
\end{minipage}%
\begin{minipage}{0.5\linewidth}
  \begin{equation} \label{eq:dotsdown-b}
    \begin{tikzpicture}[>=stealth,baseline={([yshift=-.5ex]current bounding box.center)}]
      \draw[<-] (0,-0.5) -- (0,0) arc (180:0:0.5) -- (1,-0.5);
      \bluedot{(0,0)} node [anchor=east,color=black] {$b$};
    \end{tikzpicture}
    \ =\
    \begin{tikzpicture}[>=stealth,baseline={([yshift=-.5ex]current bounding box.center)}]
      \draw[<-] (0,-0.5) -- (0,0) arc (180:0:0.5) -- (1,-0.5);
      \bluedot{(1,0)} node [anchor=west,color=black] {$\psiB(b)$};
    \end{tikzpicture}
    .
  \end{equation}
\end{minipage}\par\vspace{\belowdisplayskip}

We also allow strands to cross. For example
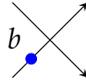
\[
  \begin{tikzpicture}[>=stealth]
    \draw[->] (0,0) -- (1,1);
    \draw[<-] (1,0) -- (0,1);
    \filldraw [blue](.25,.25) circle (2pt);
    \draw (.25,.25) node [anchor=south east] [black] {$b$};
  \end{tikzpicture}
\]
is an element of $\Hom_{\cH'_B}(\sP\sQ,\{-|b|,\bar b\}\sQ\sP)$.  (See below for an explanation of the degree shift.)  Notice that the domain of a 2-morphism is specified at the bottom of the diagram and the codomain is specified at the top, so we compose diagrams by stacking them and reading upwards.

We assign a $\Z\times \Z_2$-grading to the space of planar diagrams by defining
\begin{gather*}
  \deg\
  \begin{tikzpicture}[>=stealth,baseline={([yshift=-.5ex]current bounding box.center)}]
    \draw [->](0,0) -- (1,1);
    \draw [->](1,0) -- (0,1);
  \end{tikzpicture}
  \ = (0, 0),
  \qquad \qquad
  \deg\
  \begin{tikzpicture}[>=stealth,baseline={([yshift=-.5ex]current bounding box.center)}]
    \draw [->](0,0) -- (0,1);
    \bluedot{(0,0.5)} node[anchor=west, color=black] {$b$};
  \end{tikzpicture}
  \ = \deg b,
  \\
  \deg\
  \begin{tikzpicture}[>=stealth,baseline={([yshift=-.5ex]current bounding box.center)}]
    \draw[->] (0,0) arc (180:360:.5);
  \end{tikzpicture}
  \ = (0,0),
  \qquad \qquad
  \deg\
  \begin{tikzpicture}[>=stealth,baseline={([yshift=-.5ex]current bounding box.center)}]
    \draw[<-] (6.5,-.5) arc (0:180:.5);
  \end{tikzpicture}
  \ = (0,0),
  \\
  \deg\
  \begin{tikzpicture}[>=stealth,baseline={([yshift=-.5ex]current bounding box.center)}]
    \draw[<-] (0,0) arc (180:360:.5);
  \end{tikzpicture}
  \ =(\delta,\sigma),
  \qquad \qquad
  \deg\
  \begin{tikzpicture}[>=stealth,baseline={([yshift=-.5ex]current bounding box.center)}]
    \draw[->] (6.5,-.5) arc (0:180:.5);
  \end{tikzpicture}
  \ = (-\delta,\sigma).
\end{gather*}

We consider diagrams up to super isotopy fixing the vertical coordinates of the endpoints of strands.  By \emph{super isotopy}, we mean the following:
\begin{enumerate}
  \item Any isotopy changing the diagram only in a horizontal strip of the form $\R \times [a,b]$, with $0 \le a < b \le 1$, containing no dots or odd cups/caps (i.e.\ left cups/caps if $\sigma = 1$).  Any time two odd cups or caps (i.e.\ left cups/caps if $\sigma = 1$) move past each other vertically, we multiply the diagram by $-1$.  We may also use the relations
      \begin{equation} \label{rel:left-down-zigzag}
        \begin{tikzpicture}[>=stealth,baseline={([yshift=-.5ex]current bounding box.center)}]
          \draw[->] (1,1) .. controls (1,-1) and (0,-1) .. (0,0) .. controls (0,1) and (-1,1) .. (-1,-1);
        \end{tikzpicture}
        \ = (-1)^\sigma \ \
        \begin{tikzpicture}[>=stealth,baseline={([yshift=-.5ex]current bounding box.center)}]
          \draw[->] (0,1) -- (0,-1);
        \end{tikzpicture}
        \qquad \text{and} \qquad
        \begin{tikzpicture}[>=stealth,baseline={([yshift=-.5ex]current bounding box.center)}]
          \draw[->] (1,-1) .. controls (1,1) and (0,1) .. (0,0) .. controls (0,-1) and (-1,-1) .. (-1,1);
        \end{tikzpicture}
        \ = \
        \begin{tikzpicture}[>=stealth,baseline={([yshift=-.5ex]current bounding box.center)}]
          \draw[<-] (0,1) -- (0,-1);
        \end{tikzpicture}
      \end{equation}
      to straighten out left zigzags.

  \item Dots may slide along strands, except that they may pick up a Nakayama automorphism when passing along cups or caps according to \eqref{eq:dotsup-a}--\eqref{eq:dotsdown-b}, and when a dot of odd parity moves vertically past a dot, cup, or cap of odd parity, we multiply the diagram by $-1$.  In particular, dots on distinct strands supercommute when they move past each other:
      \begin{equation} \label{eq:supercomm}
        \begin{tikzpicture}[>=stealth,baseline={([yshift=-.5ex]current bounding box.center)}]
          \draw[->] (0,0) -- (0,1.5);
          \bluedot{(0,.5)} node [anchor=east,color=black] {$b$};
          \draw (.5,.75) node {$\ldots$};
          \draw[->] (1,0)--(1,1.5);
          \bluedot{(1,1)} node [anchor=west,color=black] {$b'$};
        \end{tikzpicture}
        = \ (-1)^{\bar{b}\bar{b'}}\
        \begin{tikzpicture}[>=stealth,baseline={([yshift=-.5ex]current bounding box.center)}]
          \draw[->] (6,0) --(6,1.5);
          \bluedot{(6,1)} node [anchor=east, color=black] {$b$};
          \draw (6.5,.75) node {$\ldots$};
          \bluedot{(7,.5)} node [anchor=west,color=black] {$b'$};
          \draw[->] (7,0) -- (7,1.5);
        \end{tikzpicture}.
      \end{equation}
\end{enumerate}
Because of the sign changes involved in super isotopy invariance, we do not allow diagrams in which odd cups, caps, or dots appear at the same height (i.e.\ have the same vertical coordinate).

The local relations we impose are the following:

\noindent\begin{minipage}{0.5\linewidth}
  \begin{equation}\label{eq:rel0a}
    \begin{tikzpicture}[>=stealth,baseline={([yshift=-.5ex]current bounding box.center)}]
      \draw[->] (0,0) -- (1,1);
      \draw[->] (1,0) -- (0,1);
      \bluedot{(.25,.25)} node [anchor=south east, color=black] {$b$};
    \end{tikzpicture}
    \ =\
    \begin{tikzpicture}[>=stealth,baseline={([yshift=-.5ex]current bounding box.center)}]
      \draw[->](2,0) -- (3,1);
      \draw[->](3,0) -- (2,1);
      \bluedot{(2.75,.75)} node [anchor=north west, color=black] {$b$};
    \end{tikzpicture}
  \end{equation}
\end{minipage}%
\begin{minipage}{0.5\linewidth}
  \begin{equation}\label{eq:rel0b}
    \begin{tikzpicture}[>=stealth,baseline={([yshift=-.5ex]current bounding box.center)}]
      \draw[->] (0,0) -- (1,1);
      \draw[->] (1,0) -- (0,1);
      \bluedot{(.75,.25)} node [anchor=south west, color=black] {$b$};
    \end{tikzpicture}
    \ =\
    \begin{tikzpicture}[>=stealth,baseline={([yshift=-.5ex]current bounding box.center)}]
      \draw[->] (2,0) -- (3,1);
      \draw[->] (3,0) -- (2,1);
      \bluedot{(2.25,.75)} node [anchor=north east, color=black] {$b$};
    \end{tikzpicture}
  \end{equation}
\end{minipage}\par\vspace{\belowdisplayskip}

\noindent\begin{minipage}{0.5\linewidth}
  \begin{equation}\label{eq:rel1a}
    \begin{tikzpicture}[>=stealth,baseline={([yshift=-.5ex]current bounding box.center)}]
      \draw (0,0) -- (2,2)[->];
      \draw (2,0) -- (0,2)[->];
      \draw[->] (1,0) .. controls (0,1) .. (1,2);
    \end{tikzpicture}
    \ =\
    \begin{tikzpicture}[>=stealth,baseline={([yshift=-.5ex]current bounding box.center)}]
      \draw (3,0) -- (5,2)[->];
      \draw (5,0) -- (3,2)[->];
      \draw[->] (4,0) .. controls (5,1) .. (4,2);
    \end{tikzpicture}
  \end{equation}
\end{minipage}%
\begin{minipage}{0.5\linewidth}
  \begin{equation}\label{eq:rel1b}
    \begin{tikzpicture}[>=stealth,baseline={([yshift=-.5ex]current bounding box.center)}]
      \draw[->] (0,0) .. controls (1,1) .. (0,2);
      \draw[->] (1,0) .. controls (0,1) .. (1,2);
    \end{tikzpicture}
    \ =\
    \begin{tikzpicture}[>=stealth,baseline={([yshift=-.5ex]current bounding box.center)}]
      \draw[->] (2,0) --(2,2);
      \draw[->] (3,0) -- (3,2);
    \end{tikzpicture}
  \end{equation}
\end{minipage}\par\vspace{\belowdisplayskip}

\noindent\begin{minipage}{0.5\linewidth}
  \begin{equation}\label{eq:rel2a}
    \begin{tikzpicture}[>=stealth,baseline={([yshift=-.5ex]current bounding box.center)}]
      \draw[<-] (0,0) .. controls (1,1) .. (0,2);
      \draw[->] (1,0) .. controls (0,1) .. (1,2);
    \end{tikzpicture}
    \ =\
    \begin{tikzpicture}[>=stealth,baseline={([yshift=-.5ex]current bounding box.center)}]
      \draw[<-] (2,0) --(2,2);
      \draw[->] (3,0) -- (3,2);
    \end{tikzpicture}
    \ - \sum_{b \in \cB} \
    \begin{tikzpicture}[>=stealth,baseline={([yshift=-.5ex]current bounding box.center)}]
      \draw (4,1.75) arc (180:360:.5);
      \draw (4,2) -- (4,1.75);
      \draw[<-] (5,2) -- (5,1.75);
      \draw (5,.25) arc (0:180:.5);
      \bluedot{(5,1.66)} node [anchor=west,color=black] {$b^\vee$};
      \bluedot{(4,.33)} node [anchor=west,color=black] {$b$};
      \draw (5,0) -- (5,.25);
      \draw[<-] (4,0) -- (4,.25);
    \end{tikzpicture}
  \end{equation}
\end{minipage}%
\begin{minipage}{0.5\linewidth}
  \begin{equation}\label{eq:rel2b}
    \begin{tikzpicture}[>=stealth,baseline={([yshift=-.5ex]current bounding box.center)}]
      \draw[->] (0,0) .. controls (1,1) .. (0,2);
      \draw[<-] (1,0) .. controls (0,1) .. (1,2);
    \end{tikzpicture}
    \ =\
    \begin{tikzpicture}[>=stealth,baseline={([yshift=-.5ex]current bounding box.center)}]
      \draw[->] (2,0) --(2,2);
      \draw[<-] (3,0) -- (3,2);
    \end{tikzpicture}
  \end{equation}
\end{minipage}\par\vspace{\belowdisplayskip}

\noindent\begin{minipage}{0.5\linewidth}
  \begin{equation}\label{eq:rel3a}
    \begin{tikzpicture}[>=stealth,baseline={([yshift=-.5ex]current bounding box.center)}]
      \draw [->](0,0) arc (180:360:0.5cm);
      \draw (1,0) arc (0:180:0.5cm);
      \bluedot{(0,0)} node [anchor=east,color=black] {$b$};
    \end{tikzpicture}
    \ = \tr_B(b)
  \end{equation}
\end{minipage}%
\begin{minipage}{0.5\linewidth}
  \begin{equation}\label{eq:rel3b}
    \begin{tikzpicture}[>=stealth,baseline={([yshift=-.5ex]current bounding box.center)}]
      \draw (0,0) .. controls (0,.5) and (.7,.5) .. (.9,0);
      \draw (0,0) .. controls (0,-.5) and (.7,-.5) .. (.9,0);
      \draw (1,-1) .. controls (1,-.5) .. (.9,0);
      \draw[->] (.9,0) .. controls (1,.5) .. (1,1);
      \draw (1.5,0) node {$=$};
      \draw (2,0) node {$0$};
    \end{tikzpicture}
  \end{equation}
\end{minipage}\par\vspace{\belowdisplayskip}
\noindent Relation \eqref{eq:rel2a} is independent of the choice of basis ${\cB}$ by Lemma~\ref{lem:casimir-properties}.  By \emph{local relation}, we mean that relations \eqref{eq:rel0a}--\eqref{eq:rel3b} can be applied to any subdiagram contained in horizonal strip of the form $\R \times [a,b]$, with $0 < a < b < 1$, where the remainder of the strip (to the right and left of the subdiagram) consists only of strands without dots, local minima, or local maxima.  If we wish to perform a local relation on a subdiagram in a vertical strip containing dots, local minima, and/or local minima, we should first super isotope the diagram so that the vertical strip containing the subdiagram no longer contains such elements.

The monoidal structure on morphisms is given by horizontal juxtaposition of diagrams.  However, it is important that we specify that, when juxtaposing diagrams, any dots, local minima, or local maxima in the left-hand diagram should be placed higher (vertically) than dots, local minima, or local maxima in the right-hand diagram.  We always super isotope the diagrams before juxtaposing them to ensure that this is the case.

Note that all of the graphical relations are graded.   Thus, the morphism spaces of $\cH_B'$ are $\Z\times\Z_2$-graded vector spaces and composition of morphisms is compatible with the grading. We denote by $\{ \cdot,\cdot\}$ the grading shift in $\cH'_B$.

To illustrate the concept of super isotopy, consider the diagrams
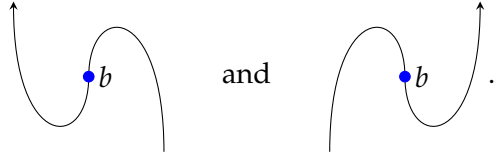
\[
  \begin{tikzpicture}[>=stealth,baseline={([yshift=-.5ex]current bounding box.center)}]
    \draw[<-] (-1,1) .. controls (-1,-1) and (0,-1) .. (0,0) .. controls (0,1) and (1,1) .. (1,-1);
    \bluedot{(0,0)} node [anchor=west, color=black] {$b$};
  \end{tikzpicture}
  \qquad \text{and} \qquad
  \begin{tikzpicture}[>=stealth,baseline={([yshift=-.5ex]current bounding box.center)}]
    \draw[<-] (1,1) .. controls (1,-1) and (0,-1) .. (0,0) .. controls (0,1) and (-1,1) .. (-1,-1);
    \bluedot{(0,0)} node [anchor=west, color=black] {$b$};
  \end{tikzpicture}\ .
\]
These diagrams are not equivalent in general.  Indeed, using relations \eqref{eq:dotsup-a}--\eqref{eq:dotsdown-b}, we see that they are equal, respectively, to
\[
  \begin{tikzpicture}[>=stealth,baseline={([yshift=-.5ex]current bounding box.center)}]
    \draw[<-] (0,1) to (0,0);
    \bluedot{(0,0.5)} node [anchor=west, color=black] {$(-1)^{\sigma \bar b} \psi(b)$};
  \end{tikzpicture}
  \qquad \text{and} \qquad
  \begin{tikzpicture}[>=stealth,baseline={([yshift=-.5ex]current bounding box.center)}]
    \draw[<-] (0,1) to (0,0);
    \bluedot{(0,0.5)} node [anchor=west, color=black] {$b$};
  \end{tikzpicture}.
\]
Thus, they are equal precisely when $\psiB(b) = (-1)^{\sigma \bar b} b$.

We define the category $\cH_B$ to be the idempotent completion (also known as the Karoubi envelope) of $\cH'_B$.  By definition, the objects of $\cH_B$ consist of pairs $(R,e)$, where $R$ is an object of $\cH'_B$ and $e \colon R \to R$ is an idempotent morphism.  The space of morphisms in $\cH_B$ from $(R,e)$ to $(R',e')$ is the subspace of morphisms in $\cH'_B$ consisting of $g \colon R \to R'$ such that $ge = g = e'g$.  The idempotent $e$ defines the identity morphism of $(R,e)$ in $\cH_B$.

Note that $\cH_B$ inherits the $\Z\times \Z_2$-grading from $\cH'_B$.  This means that the split Grothendieck group $K_0'(\cH_B)$ of $\cH_B$ is a $\Z_{q,\pi}$-algebra, where the grading shift $\{1,0\}$ corresponds to multiplication by $q$ and $\{0,1\}$ corresponds to multiplication by $\pi$.  We define $K_0(\cH_B) = K_0'(\cH_B) \otimes_{\Z_{q,\pi}} \kk$ (see \eqref{eq:k-def}).  The monoidal structure in $\cH_B$ descends to a multiplication in $K_0(\cH_B)$.

\begin{rem} \label{rem:2-category-viewpoint}
   We could have chosen to define $\cH_B'$ and $\cH_B$ as 2-categories.  More precisely, we have a two category $\mathscr{H}_B'$ whose set of objects is $\Z$.  For $n,m \in \Z$, the morphisms from $m$ to $n$ are given by the subcategory of $\cH_B'$ whose objects are sequences of $\sP$'s and $\sQ$'s such that the total number of $\sP$'s minus the total number of $\sQ$'s is equal to $n-m$.  Vertical composition is composition in the category $\cH_B'$ and horizontal composition is given by the monoidal structure on $\cH_B'$.  Then $\mathscr{H}_B$ is again the idempotent completion of $\mathscr{H}_B'$.
\end{rem}

\begin{rem} \label{rem:special-cases}
  Several categories appearing in the literature are special cases of the category $\cH_B$ (or 2-category $\mathscr{H}_B$) for appropriate choices of Frobenius algebra $B$.
  \begin{asparaenum}
    \item If $B = \F$, then the category $\cH_\F$ is precisely the category $\cH$ of \cite{Kho14}.

    \item The 2-category $\mathscr{H}_\Gamma'$ of \cite{CL12} is a special case of the 2-category $\mathscr{H}_B'$ of Remark~\ref{rem:2-category-viewpoint} for the choice $B = \Lambda^*(V) \rtimes \Gamma$, where $\Gamma$ is a finite subgroup of $\mathrm{SL}_2(\F)$ and $V$ is the natural two-dimensional representation of $\Gamma$.  In this case, $\delta=2$, $\sigma = 0$, and $\psi_B = \id_B$ (i.e.\ $B$ is a symmetric algebra).  In addition, the object $P$ of $\mathscr{H}_\Gamma$' is the object $\{-1\}\sP$ of $\mathscr{H}_B'$.   As a result, the degrees of the cups and caps in \cite{CL12} are shifted relative to their degrees in the current paper.

    \item \label{rem-item:special-cases-CLaux} The 2-category $\mathscr{H}^\Gamma$ of \cite[\S 6.1]{CL12} is also a special case of the 2-category $\mathscr{H}_B$, where $B$ is a skew-zigzag algebra of \cite{HK01}.  The particular skew-zigzag algebra needed to recover $\mathscr{H}^\Gamma$ is described in \cite[\S2.3]{CLS14}.  An upward strand labeled $i$ in $\mathscr{H}^\Gamma$ is the (identity morphism of the) element $(\sP,e_i)$ of the Karoubi envelope $\cH_B$, where $e_i$ is the idempotent in the skew-zigzag algebra corresponding to a path of length zero at vertex $i$.

    \item The category $\cH^t$ of \cite[\S 6]{CS14} is the special case of the category $\cH_B$ for $B = \Cl$.

    \item The category $\cH_\Gamma^-$ of \cite{HS15} is the special case of the category $\cH_B$ for $B = (\Lambda^*(V) \rtimes \Gamma) \otimes \Cl$.
  \end{asparaenum}
\end{rem}

\section{Action of $\cH_B$ on module categories} \label{sec:actions}

In this section we define an action of $\cH_B$ on the category of modules $\bigoplus_{n \ge 0} A_n \md$.  More precisely, for each $n \in \N$, we will define a functor
\[
  \bF_n \colon \cH_B \to \bigoplus_{m \in \Z} (A_m,A_n)\bimod.
\]
Then the composition
\begin{multline} \label{eq:functor-A-def}
  \bF \colon \cH_B
  \xrightarrow{\bigoplus_{n \in \N} \bF_n} \bigoplus_{m,n \in \N} (A_m,A_n)\bimod
  \xrightarrow{M \mapsto M \otimes -} \bigoplus_{m,n \in \N} \Fun (A_n\md,A_m\md) \\
  \to \Fun \left( \bigoplus_{n \in \N} A_n\md, \bigoplus_{n \in \N} A_n\md \right).
\end{multline}
defines a representation of the category $\cH_B$ on the category $\bigoplus_{n \in \N} A_n\md$.  This will categorify the Fock space representation of the Heisenberg algebra $\fh_B$.

\subsection{Objects}

Fix $n \in \N$.  We define the functor $\bF_n$ on objects as follows.  The object $\sP$ is sent to the bimodule $(k+1)_k$, and the object $\sQ$ is sent to the bimodule $\prescript{}{k}(k+1)$, where $k$ is uniquely determined by the fact that our functor should respect the tensor product and be a functor to the category $\bigoplus_{m \in \Z} (A_m,A_n)\bimod$.  For example,
\[
  \bF_n(\sP\sP\sQ\sP\sQ\sQ\sP) = (n+1)_n(n)_{n-1}(n)_n(n)_{n-1}(n)_n(n+1)_{n+1}(n+1)_n
\]
We define the $\bF_n$ to be zero on any object for which the indices $k$ become negative.  For example, $\bF_1(\sQ\sQ)=0$.  Finally, we declare $\bF_n$ to commute with degree shifts.

\subsection{Morphisms}

We now define the functor on morphisms.  Let $D$ be a morphism in $\cH_B$ consisting of a single planar diagram.  We label the rightmost region of the diagram $D$ with $n$.  We then label all other regions of the diagram with integers such that, as we move from right to left across the diagram, labels increase by one when we cross upward pointing strands and decrease by one when we cross downward pointing strands.  It is clear that there is a unique way to label the regions of the diagram in this way.  For instance, the following diagram is labelled as indicated (where we have omitted the labels of the dots).
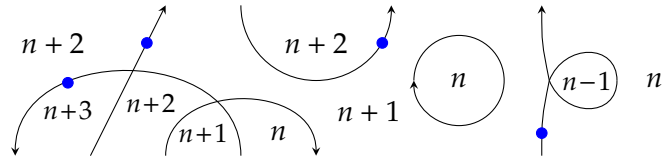
\[
  \begin{tikzpicture}[>=stealth]
    \draw[->] (0,0) to (1,2);
    \draw[->] (1,0) .. controls (1,1) and (3,1) .. (3,0);
    \draw[->] (2,0) .. controls (2,1.5) and (-1,1.5) .. (-1,0);
    \draw[->] (2,2) arc(180:360:1);
    \bluedot{(0.75,1.5)};
    \bluedot{(-.3,0.97)};
    \draw (7,1) .. controls (7,1.5) and (6.3,1.5) .. (6.1,1);
    \draw (7,1) .. controls (7,.5) and (6.3,.5) .. (6.1,1);
    \draw (6,0) .. controls (6,.5) .. (6.1,1);
    \draw (6.1,1) .. controls (6,1.5) .. (6,2) [->];
    \bluedot{(6,0.3)};
    \draw[->] (4.3,1) arc(180:-180:.6);
    \filldraw [blue] (3.88,1.5) circle (2pt);
    \draw (7.5,1) node {$n$};
    \draw (6.6,1) node {\small{$n\!-\!1$}};
    \draw (3.7,.6) node {$n+1$};
    \draw (4.9,1) node {$n$};
    \draw (3,1.5) node {$n+2$};
    \draw (-.5,1.5) node {$n+2$};
    \draw (-.3,.6) node {\small$n\!+\!3$};
    \draw (.8,.7) node {\small$n\!+\!2$};
    \draw (1.5,.3) node {\small$n\!+\!1$};
    \draw (2.5,.3) node {$n$};
  \end{tikzpicture}
\]
Each diagram is a composition of dots, crossings, cups and caps.  Thus, we define the functors $\bF_n$, $n \in \N$, on such atoms.  Since, on such atoms, the functor $\bF_n$ is independent of $n$, we will drop the index and describe the functor $\bF$.

\subsection{Cups and caps}

We define $\bF$ on cups and caps by

\noindent\begin{minipage}{0.5\linewidth}
  \begin{equation} \label{eq:right-cap-cup-maps}
    \begin{tikzpicture}[>=stealth,baseline={([yshift=-.5ex]current bounding box.center)}]
      \draw[<-] (0,-.5) arc (0:180:.5);
      \draw (0.7,-0.3) node {$n+1$};
    \end{tikzpicture}
    \quad \mapsto \varepsilon_\rR,
    \end{equation}
\end{minipage}%
\begin{minipage}{0.5\linewidth}
  \begin{equation}
    \begin{tikzpicture}[>=stealth,baseline={([yshift=-.5ex]current bounding box.center)}]
      \draw[->] (0,0) arc (180:360:.5);
      \draw (1.3,-0.3) node {$n$};
    \end{tikzpicture}
    \quad \mapsto \eta_\rR,
  \end{equation}
\end{minipage}\par\vspace{\belowdisplayskip}

\noindent\begin{minipage}{0.5\linewidth}
  \begin{equation}
    \begin{tikzpicture}[>=stealth,baseline={([yshift=-.5ex]current bounding box.center)}]
      \draw[->] (0,-.5) arc (0:180:.5);
      \draw (0.3,-0.3) node {$n$};
    \end{tikzpicture}
    \quad \mapsto \varepsilon_\rL,
    \end{equation}
\end{minipage}%
\begin{minipage}{0.5\linewidth}
  \begin{equation} \label{eq:left-cap-cup-maps}
    \begin{tikzpicture}[>=stealth,baseline={([yshift=-.5ex]current bounding box.center)}]
      \draw[<-] (0,0) arc (180:360:.5);
      \draw (1.5,-0.3) node {$n+1$};
    \end{tikzpicture}
    \quad \mapsto \eta_\rL,
  \end{equation}
\end{minipage}\par\vspace{\belowdisplayskip}

\noindent where $\varepsilon_\rR$, $\eta_\rR$, $\varepsilon_\rL$, and $\eta_\rL$ are the bimodule homomorphisms of Proposition~\ref{prop:adjunction-maps}.  Note that the maps \eqref{eq:right-cap-cup-maps}--\eqref{eq:left-cap-cup-maps} are degree preserving.

\subsection{Dots} \label{subsec:dots}

It is straightforward to verify that the maps
\begin{gather}
  \phi_n \colon B^\op \rightarrow \END \left( (n+1)_n \right), \quad \phi(b) = \pr{\left(1_B^{\otimes n} \otimes b\right)}, \quad \text{and} \label{eq:phi-def} \\
  \phi'_n \colon B \rightarrow \END \left( \prescript{}{n}(n+1) \right),\quad \phi'(b) = \pl{\left(1_B^{\otimes n} \otimes b \right)}, \label{eq:phi'-def}
\end{gather}
are algebra homomorphisms.  Since the value of $n$ will usually be clear from the context, we will sometimes omit the subscript on $\phi_n$ and $\phi_n'$.
\details{
  For $b \in B$, $\phi(b)$ is a homomorphism of left $A_{n+1}$-modules because, for $\bb_1,\bb_2 \in B^{\otimes (n+1)}$ and $\tau_1, \tau _2 \in S_{n+1}$, we have
  \[
    (\bb_2 \tau_2) \left( \phi(b) \cdot (\bb_1 \tau_1) \right)
    = (-1)^{\bar b \bar \bb_1} (\bb_2 \tau_2) ({\bb_1}\tau_1) (1_B^{\otimes n} \otimes b) \\
    = (-1)^{\bar b \bar \bb_2} \phi(b) \cdot \left((\bb_2 \tau_2) (\bb_1 \tau_1)\right).
  \]
  It is a homomorphism of right $A_n$-modules since, for $\bb_1 \in B^{\otimes (n+1)}$, $\bb_2 \in B^{\otimes n}$, $\tau_1 \in S_{n+1}$, $\tau_2 \in S_n$, we have
  \begin{multline*}
    \left(\phi(b) \cdot (\bb_1 \tau_1) \right) (\bb_2 \tau_2)
    = (-1)^{\bar b \bar \bb_1} (\bb_1 \tau_1) (1_B^{\otimes n} \otimes b) (\bb_2 \tau_2) \\
    = (-1)^{\bar b \bar \bb_1 + \bar b \bar \bb_2} (\bb_1 \tau_1) (\bb_2 \tau_2) (1_B^{\otimes n} \otimes b)
    = \phi(b) \cdot \left((\bb_1 \tau_1) (\bb_2 \tau_2) \right),
  \end{multline*}
  where, in the second equality, we use that $\tau_2 \in S_n$, so $\tau_2 \cdot (1_B^{\otimes n} \otimes b) = (1_B^{\otimes n} \otimes b)$.  It is clear that the left and right actions commute.  Thus $\phi(b)$ is a homomorphism of $(A_{n+1},A_n)$-bimodules.

  Finally, $\phi$ is an algebra homomorphism since it is clearly linear, maps the identity to the identity, and, for $b_1, b_2 \in B$,
  \[
    \phi \left( (-1)^{\bar b_1 \bar b_2} b_2 b_1 \right)
    = \pr{\left( 1^{\otimes n} \otimes (-1)^{\bar b_1 \bar b_2} b_2 b_1 \right)}
    = \pr{\left(1^{\otimes n} \otimes b_1 \right)}\, \pr{\left(1^{\otimes n} b_2\right)}
    = \phi(b_1) \phi(b_2).
  \]

  The proof that $\phi'$ is an algebra homomorphism is similar.
}
We thus define $\bF$ on dots by

\noindent\begin{minipage}{0.5\linewidth}
  \begin{equation} \label{eq:dot-maps-up}
    \begin{tikzpicture}[>=stealth,baseline={([yshift=-.5ex]current bounding box.center)}]
      \draw[->] (0,0) -- (0,1.6);
      \bluedot{(0,.8)};
      \draw (0,.8) node [anchor=west] {$b$};
      \draw (0.5,0.3) node {$n$};
    \end{tikzpicture}
    \mapsto \phi_n(b) = \pr{\left(1_B^{\otimes n} \otimes b\right)},
  \end{equation}
\end{minipage}%
\begin{minipage}{0.5\linewidth}
  \begin{equation} \label{eq:dot-maps-down}
    \begin{tikzpicture}[>=stealth,baseline={([yshift=-.5ex]current bounding box.center)}]
      \draw[<-] (0,0) -- (0,1.6);
      \bluedot{(0,.8)};
      \draw (0,.8) node [anchor=west] {$b$};
      \draw (-0.4,0.3) node {$n$};
    \end{tikzpicture}
    \mapsto \phi'_n(b)  = \pl{\left(1_B^{\otimes n} \otimes b \right)}.
  \end{equation}
\end{minipage}\par\vspace{\belowdisplayskip}

\noindent Note that the maps \eqref{eq:dot-maps-up} and \eqref{eq:dot-maps-down} preserve degree.

\subsection{Crossings} \label{sec:cross}

We define $\bF$ on upwards crossings by the degree preserving map
\begin{equation} \label{eq:up-crossing-map}
  \begin{tikzpicture}[>=stealth,baseline={([yshift=-.5ex]current bounding box.center)}]
    \draw[->] (0,0) -- (1,1);
    \draw[->] (1,0) -- (0,1);
    \draw (1.5,0.5) node {$n$};
  \end{tikzpicture}
  \ \mapsto \Big( \pr{s_{n+1}} \colon (n+2)_n \to (n+2)_n \Big).
\end{equation}
Note that $\pr{s_{n+1}}$ is a homomorphism of $(A_{n+2},A_n)$-bimodules since $s_{n+1}$ commutes with the elements of $A_n$, viewed as a subalgebra of $A_{n+2}$.

\begin{lem} \label{lem:downcross-invariance}
  The images under $\bF$ of the diagrams
  \begin{equation} \label{eq:up-down-crossing-twist}
    \begin{tikzpicture}[>=stealth,baseline={([yshift=-.5ex]current bounding box.center)}]
      \draw[->] (-1.5,1.5) .. controls (-1.5,0.5) and (-1,-1) .. (0,0) .. controls (1,1) and (1.5,-0.5) .. (1.5,-1.5);
      \draw[->] (-2,1.5) .. controls (-2,-2) and (1.5,-1.5) .. (0,0) .. controls (-1.5,1.5) and (2,2) .. (2,-1.5);
      \draw (2.3,0.5) node {$n$};
    \end{tikzpicture}
    \qquad \text{and} \qquad
    \begin{tikzpicture}[>=stealth,baseline={([yshift=-.5ex]current bounding box.center)}]
      \draw[->] (1.5,1.5) .. controls (1.5,0.5) and (1,-1) .. (0,0) .. controls (-1,1) and (-1.5,-0.5) .. (-1.5,-1.5);
      \draw[->] (2,1.5) .. controls (2,-2) and (-1.5,-1.5) .. (0,0) .. controls (1.5,1.5) and (-2,2) .. (-2,-1.5);
      \draw (1.9,-0.6) node {$n$};
    \end{tikzpicture}
  \end{equation}
  are both $\pl{s_{n-1}}$.
\end{lem}

\begin{proof}
  First note that $\pl{s_{n-1}} \colon \prescript{}{n-2}(n) \to \prescript{}{n-2}(n)$ is a bimodule homomorphism.  The leftmost map of \eqref{eq:up-down-crossing-twist} is the composition
  \[
    \prescript{}{n-2}(n)
    \cong (n-2)_{n-2}(n)
    \xrightarrow{\eta_\rR^2 \otimes \id} \prescript{}{n-2}(n)_{n-2}(n)
    \xrightarrow{\pr{s_{n-1}} \otimes \id} \prescript{}{n-2}(n)_{n-2}(n)
    \xrightarrow{\varepsilon_\rR^2} \prescript{}{n-2}(n).
  \]
  This map is uniquely determined by the image of $1_n$ and we compute
  \[
    1_n
    \mapsto 1_{n-2} \otimes 1_n
    \mapsto 1_n \otimes 1_n
    \mapsto s_{n-1} \otimes 1_n
    \mapsto s_{n-1}
    = \pl{s_{n-1}}(1_n).
  \]
  On the other hand, the rightmost map of \eqref{eq:up-down-crossing-twist} is the composition
  \begin{multline*}
    \prescript{}{n-2}(n)
    \xrightarrow{\eta_\rL} \prescript{}{n-2}(n)_{n-1}(n)
    \cong \prescript{}{n-2}(n)_{n-1}(n-1)_{n-1}(n)
    \xrightarrow{\id \otimes \eta_\rL \otimes \id} \prescript{}{n-2}(n)_{n-1}(n-1)_{n-2}(n-1)_{n-1}(n) \\
    \cong \prescript{}{n-2}(n)_{n-2}(n)
    \xrightarrow{\pr{s_{n-1}} \otimes \id} \prescript{}{n-2}(n)_{n-2}(n)
    \xrightarrow{\varepsilon_\rL^2 \otimes \id} (n-2)_{n-2}(n)
    \cong \prescript{}{n-2}(n).
  \end{multline*}
  This map is also uniquely determined by the image of $1_n$ and we compute
  \begin{align*}
    1_n
    &\mapsto \sum_{\substack{b \in \cB \\ i \in \{1,\dotsc,n\}}} (-1)^{\sigma \bar b^\vee} s_i \dotsm s_{n-1} (1_B^{\otimes (n-1)} \otimes b^\vee) \otimes (1_B^{\otimes (n-1)} \otimes b) s_{n-1} \dotsm s_i \\
    &\mapsto \sum_{\substack{b,c \in \cB \\ i \in \{1,\dotsc,n\} \\ j \in \{1,\dotsc,n-1\}}} (-1)^{\sigma \bar c^\vee} s_i \dotsm s_{n-1} (1_B^{\otimes (n-1)} \otimes b^\vee) \otimes s_j \dotsm s_{n-2} (1_B^{\otimes (n-2)} \otimes c^\vee) \otimes \\
    & \pushright{\otimes (1_B^{\otimes (n-2)} \otimes c) s_{n-2} \dotsm s_j \otimes (1_B^{\otimes (n-1)} \otimes b) s_{n-1} \dotsm s_i} \\
    &\mapsto \sum_{\substack{b,c \in \cB \\ i \in \{1,\dotsc,n\} \\ j \in \{1,\dotsc,n-1\}}} (-1)^{\bar b \bar c^\vee} s_i \dotsm s_{n-1} s_j \dotsm s_{n-2} (1_B^{\otimes (n-2)} \otimes c^\vee \otimes b^\vee) \otimes (1_B^{\otimes (n-2)} \otimes c \otimes b) s_{n-2} \dotsm s_j s_{n-1} \dotsm s_i \\
    &\mapsto \sum_{\substack{b,c \in \cB \\ i \in \{1,\dotsc,n\} \\ j \in \{1,\dotsc,n-1\}}} (-1)^{\bar b \bar c^\vee} s_i \dotsm s_{n-1} s_j \dotsm s_{n-2} (1_B^{\otimes (n-2)} \otimes c^\vee \otimes b^\vee) s_{n-1} \otimes (1_B^{\otimes (n-2)} \otimes c \otimes b) s_{n-2} \dotsm s_j s_{n-1} \dotsm s_i \\
    &\mapsto \sum_{b,c \in \cB} \tr_B(c^\vee) \tr_B(b^\vee) (1_B^{\otimes (n-2)} \otimes c \otimes b) s_{n-1}
    \stackrel{\eqref{eq:dual-bases-decomp}}{=} s_{n-1} = \pl{s_{n-1}}(1_n),
  \end{align*}
  where we have used the fact that $s_i \dotsm s_{n-1} s_j \dotsm s_{n-1} \in S_{n-2}$ if and only if $i=j=n-1$ and that $\tr_B(c^\vee) \tr_B(b^\vee)=0$ unless $\bar b \bar c^\vee = 0$.
\end{proof}

In light of Lemma~\ref{lem:downcross-invariance}, we define $\bF$ on downwards crossings by the degree preserving map
\begin{equation}
  \begin{tikzpicture}[>=stealth,baseline={([yshift=-.5ex]current bounding box.center)}]
    \draw[<-] (0,0) -- (1,1);
    \draw[<-] (1,0) -- (0,1);
    \draw (1.5,0.5) node {$n$};
  \end{tikzpicture}
  \quad \mapsto \pl{s_{n-1}}.
\end{equation}

\begin{lem} \label{lem:rightcross-invariance}
  The images under $\bF$ of the diagrams
  \[
    \begin{tikzpicture}[>=stealth,baseline={([yshift=-.5ex]current bounding box.center)}]
    \draw[->] (1,0) .. controls (1,0.5) and (0,0.5) .. (0,1);
    \draw[->] (-0.5,1) .. controls (-0.5,0) and (0,0) .. (0.5,0.5) .. controls (1,1) and (1.5,1) .. (1.5,0);
    \draw (2,0.5) node {$n$};
    \end{tikzpicture}
    \quad \text{and} \quad
    \begin{tikzpicture}[>=stealth,baseline={([yshift=-.5ex]current bounding box.center)}]
      \draw[<-] (-1,0) .. controls (-1,0.5) and (0,0.5) .. (0,1);
      \draw[<-] (0.5,1) .. controls (0.5,0) and (0,0) .. (-0.5,0.5) .. controls (-1,1) and (-1.5,1) .. (-1.5,0);
      \draw (1,0.5) node {$n$};
    \end{tikzpicture}
  \]
  are both the bimodule homomorphism
  \[
    (n)_{n-1}(n) \to \prescript{}{n}(n+1)_n,\ a \otimes a' \mapsto as_na'.
  \]
\end{lem}

\begin{proof}
  This follows from a straightforward computation, which will be omitted.
  \details{
    The image under $\bF$ of the left-hand diagram is the composition
    \[
      (n)_{n-1}(n) \xrightarrow{\eta_\rR \otimes \id} \prescript{}{n}(n+1)_{n-1}(n) \xrightarrow{\pr{s_n} \otimes \id} \prescript{}{n}(n+1)_{n-1}(n) \xrightarrow{\id \otimes \varepsilon_\rR} \prescript{}{n}(n+1)_n.
    \]
    For $a \in (n)_{n-1}$ and $a' \in \prescript{}{n-1}(n)$, we have
    \[
      a \otimes a' \mapsto a \otimes a' \mapsto as_n \otimes a' \mapsto a s_n a'.
    \]
    Similarly, the image under $\bF$ of the right-hand diagram is the composition
    \[
      (n)_{n-1}(n) \xrightarrow{\id \otimes \eta_\rR} (n)_{n-1}(n+1)_n \xrightarrow{\id \otimes \pl{s_n}} (n)_{n-1}(n+1)_n \xrightarrow{\varepsilon_\rR \otimes \id} \prescript{}{n}(n+1)_n.
    \]
    For $a \in (n)_{n-1}$ and $a' \in \prescript{}{n-1}(n)$, we have
    \[
      a \otimes a' \mapsto a \otimes a' \mapsto a \otimes s_n a' \mapsto a s_n a'. \qedhere
    \]
  }
\end{proof}

In light of Lemma~\ref{lem:rightcross-invariance}, we define $\bF$ on rightwards crossings by the degree preserving map
\begin{equation} \label{eq:right-crossing-map}
  \begin{tikzpicture}[>=stealth,baseline={([yshift=-.5ex]current bounding box.center)}]
    \draw[->] (0,0) -- (1,1);
    \draw[<-] (1,0) -- (0,1);
    \draw (1.5,0.5) node {$n$};
  \end{tikzpicture}
  \ \mapsto
  \Big( (n)_{n-1}(n) \to \prescript{}{n}(n+1)_n,\ a \otimes a' \mapsto a s_n a' \Big).
\end{equation}

\begin{lem} \label{lem:leftcross-invariance}
  The images under $\bF$ of the diagrams
  \[
    \begin{tikzpicture}[>=stealth,baseline={([yshift=-.5ex]current bounding box.center)}]
      \draw[<-] (1,0) .. controls (1,0.5) and (0,0.5) .. (0,1);
      \draw[<-] (-0.5,1) .. controls (-0.5,0) and (0,0) .. (0.5,0.5) .. controls (1,1) and (1.5,1) .. (1.5,0);
      \draw (2,0.5) node {$n$};
    \end{tikzpicture}
    \quad \text{and} \quad
    \begin{tikzpicture}[>=stealth,baseline={([yshift=-.5ex]current bounding box.center)}]
      \draw[->] (-1,0) .. controls (-1,0.5) and (0,0.5) .. (0,1);
      \draw[->] (0.5,1) .. controls (0.5,0) and (0,0) .. (-0.5,0.5) .. controls (-1,1) and (-1.5,1) .. (-1.5,0);
      \draw (1,0.5) node {$n$};
    \end{tikzpicture}
  \]
  are both the maps uniquely determined by
  \[
    \prescript{}{n}(n+1)_n \to (n)_{n-1}(n),\quad
    1_B^{\otimes n} \otimes b \mapsto 0,\
    (1_B^{\otimes n} \otimes b)s_n \mapsto (-1)^\sigma 1_B^{\otimes n} \otimes (1_B^{\otimes (n-1)} \otimes b),\quad b \in B.
  \]
\end{lem}

\begin{proof}
  Since $\prescript{}{n}(n+1)_n$ is generated, as an $(A_n,A_n)$-bimodule, by $1_B^{\otimes n} \otimes b$ and $(1_B^{\otimes n} \otimes b)s_n$, $b \in B$, it suffices to compute the images of these elements.  The image under $\bF$ of the left-hand diagram is the composition
  \[
    \prescript{}{n}(n+1)_n \xrightarrow{\eta_\rL} (n)_{n-1}(n+1)_n \xrightarrow{\id \otimes \pl{s_n}} (n)_{n-1}(n+1)_n \xrightarrow{\id \otimes \varepsilon_\rL} (n)_{n-1}(n).
  \]
  We have
  \begin{multline*}
    1_B^{\otimes n} \otimes b
    \mapsto \sum_{\substack{c \in \cB \\ i \in \{1,\dotsc,n\}}} (-1)^{\sigma \bar c^\vee} s_i \dotsm s_{n-1} (1_B^{\otimes (n-1)} \otimes c^\vee) \otimes (1_B^{\otimes (n-1)} \otimes c \otimes 1_B) s_{n-1} \dotsm s_i (1_B^{\otimes n} \otimes b) \\
    \mapsto \sum_{\substack{c \in \cB \\ i \in \{1,\dotsc,n\}}} (-1)^{\sigma \bar c^\vee} s_i \dotsm s_{n-1} (1_B^{\otimes (n-1)} \otimes c^\vee) \otimes (1_B^{\otimes n} \otimes c) s_n \dotsm s_i (1_B^{\otimes n} \otimes b) \mapsto 0,
  \end{multline*}
  and
  \begin{multline*}
    (1_B^{\otimes n} \otimes b)s_n
    \mapsto \sum_{\substack{c \in \cB \\ i \in \{1,\dotsc,n\}}} (-1)^{\sigma \bar c^\vee} s_i \dotsm s_{n-1} (1_B^{\otimes (n-1)} \otimes c^\vee) \otimes (1_B^{\otimes (n-1)} \otimes c \otimes 1_B) s_{n-1} \dotsm s_i (1_B^{\otimes n} \otimes b)s_n \\
    \mapsto \sum_{\substack{c \in \cB \\ i \in \{1,\dotsc,n\}}} (-1)^{\sigma \bar c^\vee} s_i \dotsm s_{n-1} (1_B^{\otimes (n-1)} \otimes c^\vee) \otimes (1_B^{\otimes n} \otimes c) s_n \dotsm s_i (1_B^{\otimes n} \otimes b) s_n
    \mapsto (-1)^\sigma 1_B^{\otimes n} \otimes (1_B^{\otimes (n-1)} \otimes b),
  \end{multline*}
  where, in the last map, we used \eqref{eq:dual-bases-decomp} and the fact that $\tr(c^\vee)=0$ unless $\bar c^\vee = \sigma$.  The proof for the right-hand diagram is similar.

  \details{
    The image under $\bF$ of the right-hand diagram is the composition
    \[
      \prescript{}{n}(n+1)_n \xrightarrow{\eta_\rL} \prescript{}{n}(n+1)_{n-1}(n) \xrightarrow{\pr{s_n} \otimes \id} \prescript{}{n}(n+1)_{n-1}(n) \xrightarrow{\varepsilon_\rL \otimes \id} (n)_{n-1}(n),
    \]
    and we have
    \begin{multline*}
      1_B^{\otimes n} \otimes b
      \mapsto \sum_{\substack{c \in \cB \\ i \in \{1,\dotsc,n\}}} (-1)^{\sigma (\bar b + \bar c^\vee)} (1_B^{\otimes n} \otimes b) s_i \dotsm s_{n-1} (1_B^{\otimes (n-1)} \otimes c^\vee \otimes 1_B) \otimes (1_B^{\otimes (n-1)} \otimes c) s_{n-1} \dotsm s_i \\
      \mapsto \sum_{\substack{c \in \cB \\ i \in \{1,\dotsc,n\}}} (-1)^{\sigma (\bar b + \bar c^\vee)} (1_B^{\otimes n} \otimes b) s_i \dotsm s_n (1_B^{\otimes n} \otimes c^\vee) \otimes (1_B^{\otimes (n-1)} \otimes c) s_{n-1} \dotsm s_i \mapsto 0,
    \end{multline*}
    and
    \begin{multline*}
      (1_B^{\otimes n} \otimes b) s_n
      \mapsto \sum_{\substack{c \in \cB \\ i \in \{1,\dotsc,n\}}} (-1)^{\sigma (\bar b + \bar c^\vee)} (1_B^{\otimes n} \otimes b) s_n s_i \dotsm s_{n-1} (1_B^{\otimes (n-1)} \otimes c^\vee \otimes 1_B) \otimes (1_B^{\otimes (n-1)} \otimes c) s_{n-1} \dotsm s_i \\
      \mapsto \sum_{\substack{c \in \cB \\ i \in \{1,\dotsc,n\}}} (-1)^{\sigma (\bar b + \bar c^\vee)} (1_B^{\otimes n} \otimes b) s_n s_i \dotsm s_n (1_B^{\otimes n} \otimes c^\vee) \otimes (1_B^{\otimes (n-1)} \otimes c) s_{n-1} \dotsm s_i
      \mapsto (-1)^\sigma 1_B^{\otimes n} \otimes (1_B^{\otimes (n-1)} \otimes b),
    \end{multline*}
    where, in the last map, we used \eqref{eq:dual-bases-decomp} and the fact that $\tr(bc^\vee)=0$ unless $\bar b + \bar c^\vee = \sigma$.
  }
\end{proof}

In light of Lemma~\ref{lem:leftcross-invariance} and the sign of $(-1)^\sigma$ appearing in \eqref{rel:left-down-zigzag}, we define $\bF$ on leftwards crossings by the degree preserving map
\begin{equation} \label{eq:left-crossing-map}
  \begin{tikzpicture}[>=stealth,baseline={([yshift=-.5ex]current bounding box.center)}]
    \draw[<-] (0,0) -- (1,1);
    \draw[->] (1,0) -- (0,1);
    \draw (1.5,0.5) node {$n$};
  \end{tikzpicture}
  \ \mapsto \Big( \prescript{}{n}(n+1)_n \to (n)_{n-1}(n),\quad 1_B^{\otimes n} \otimes b \mapsto 0,\ (1_B^{\otimes n} \otimes b)s_n \mapsto 1_B^{\otimes n} \otimes (1_B^{\otimes (n-1)} \otimes b) \Big).
\end{equation}

\subsection{Action}

We now prove one of our main results: that the functors $\bF_n$ are well defined.  We assume some knowledge of cyclic biadjointness and its relation to planar diagrammatics for bimodules.  We refer the reader to \cite{Kho10} for an overview of this topic.

\begin{theo} \label{theo:action-functor}
  The above maps give a well-defined functor
  \[
    \bF_n \colon \cH_B \to \bigoplus_{m \in \Z} (A_m,A_n)\bimod
  \]
  for each $n \in \N$ and hence define an action of $\cH_B$ on $\bigoplus_{n \in \N} A_n\md$.
\end{theo}

\begin{proof}
  Since $(A_m,A_n)\bimod$ is idempotent complete for all $m,n \in \N$, any functor from $\cH_B'$ to $\bigoplus_{m \in \Z} (A_m,A_n)\bimod$ naturally induces a functor $\cH_B \to \bigoplus_{m \in \Z} (A_m,A_n)\bimod$.  Thus, it suffices to consider the category $\cH_B'$.

  By Proposition~\ref{prop:adjunction-maps}, the images $(n+1)_n$ and $\prescript{}{n}(n+1)$ of the objects $\sP$ and $\sQ$ under $\bF$ are biadjoint up to shift.  In particular, the relations
  \begin{equation} \label{eq:relsnake}
    \begin{tikzpicture}[>=stealth,baseline={([yshift=-.5ex]current bounding box.center)}]
      \draw[->] (0,0) to (0,1) arc(180:0:.5) arc(180:360:.5) to (2,2);
    \end{tikzpicture}
    \ =\
    \begin{tikzpicture}[>=stealth,baseline={([yshift=-.5ex]current bounding box.center)}]
      \draw[->] (3,0) -- (3,2);
    \end{tikzpicture}
    \ =\
    \begin{tikzpicture}[>=stealth,baseline={([yshift=-.5ex]current bounding box.center)}]
      \draw[<-] (4,2) to (4,1) arc(180:360:.5) arc(180:0:.5) to (6,0);
    \end{tikzpicture}
    \qquad \text{and} \qquad (-1)^\sigma\
    \begin{tikzpicture}[>=stealth,baseline={([yshift=-.5ex]current bounding box.center)}]
      \draw[<-] (0,0) to (0,1) arc(180:0:.5) arc(180:360:.5) to (2,2);
    \end{tikzpicture}
    \ =\
    \begin{tikzpicture}[>=stealth,baseline={([yshift=-.5ex]current bounding box.center)}]
      \draw[<-] (3,0) -- (3,2);
    \end{tikzpicture}
    \ =\
    \begin{tikzpicture}[>=stealth,baseline={([yshift=-.5ex]current bounding box.center)}]
      \draw[->] (4,2) to (4,1) arc(180:360:.5) arc(180:0:.5) to (6,0);
    \end{tikzpicture}
  \end{equation}
  are preserved by $\bF$.  The fact that $\bF$ preserves invariance under local super isotopies not involving dots then follows from Lemmas~\ref{lem:downcross-invariance}, \ref{lem:rightcross-invariance}, and \ref{lem:leftcross-invariance}.  It remains to show that $\bF$ preserves the relations \eqref{eq:collision-up}--\eqref{eq:rel3b}.

  \medskip

  \paragraph{\emph{Relations \eqref{eq:collision-up}, \eqref{eq:collision-down}, \eqref{eq:supercomm}}.} Relations \eqref{eq:collision-up} and \eqref{eq:collision-down} follow immediately from the fact that the maps \eqref{eq:phi-def} and \eqref{eq:phi'-def} are algebra homomorphisms.  Relation \eqref{eq:supercomm} is satisfied since the tensor product of algebras satisfies the supercommutation relation.

  \medskip

  \paragraph{\emph{Relations \eqref{eq:dotsup-a}, \eqref{eq:dotsup-b}}.}  If we label the rightmost region of \eqref{eq:dotsup-a} by $n$, then the image under $\bF$ of the left-hand side corresponds to the composition
  \[
    (n) \xrightarrow{\eta_L} (n)_{n-1}(n) \xrightarrow{\phi(\psiB(b)) \otimes \id} (n)_{n-1}(n).
  \]
  Since the bimodule $(n)$ is generated by $1_n$, it suffices to consider the image of this element.  We have
  \begin{align*}
    1_n &\mapsto \sum_{\substack{c \in \cB \\ i \in \{1,\dotsc,n\}}} (-1)^{\sigma \bar c^\vee} s_i \dotsm s_{n-1} (1_B^{\otimes n} \otimes c^\vee) \otimes (1_B^{\otimes n} \otimes c) s_{n-1} \dotsm s_i \\
    &\mapsto \sum_{\substack{c \in \cB \\ i \in \{1,\dotsc,n\}}} (-1)^{\bar b^\vee \bar c^\vee} s_i \dotsm s_{n-1} (1_B^{\otimes n} \otimes c^\vee \psiB(b)) \otimes (1_B^{\otimes n} \otimes c) s_{n-1} \dotsm s_i \\
    &\!\! \stackrel{\eqref{eq:dual-bases-decomp}}{=} \sum_{\substack{c,d \in \cB \\ i \in \{1,\dotsc,n\}}} (-1)^{\bar b^\vee \bar c^\vee} s_i \dotsm s_{n-1} \left( 1_B^{\otimes n} \otimes \tr_B(d c^\vee \psiB(b)) d^\vee \right) \otimes (1_B^{\otimes n} \otimes c) s_{n-1} \dotsm s_i \\
    &= \sum_{\substack{c,d \in \cB \\ i \in \{1,\dotsc,n\}}} (-1)^{\bar b^\vee \bar c^\vee} s_i \dotsm s_{n-1} ( 1_B^{\otimes n} \otimes d^\vee ) \otimes \left( 1_B^{\otimes n} \otimes \tr_B(d c^\vee \psiB(b)) c \right) s_{n-1} \dotsm s_i \\
    &= \sum_{\substack{c,d \in \cB \\ i \in \{1,\dotsc,n\}}} (-1)^{\sigma \bar c^\vee + \bar b \bar d} s_i \dotsm s_{n-1} ( 1_B^{\otimes n} \otimes d^\vee ) \otimes \left( 1_B^{\otimes n} \otimes \tr_B(b d c^\vee) c \right) s_{n-1} \dotsm s_i \\
    &\!\! \stackrel{\eqref{eq:dual-bases-decomp}}{=} \sum_{\substack{d \in \cB \\ i \in \{1,\dotsc,n\}}} (-1)^{\bar b^\vee \bar d^\vee} s_i \dotsm s_{n-1} ( 1_B^{\otimes n} \otimes d^\vee ) \otimes (1_B^{\otimes n} \otimes b d) s_{n-1} \dotsm s_i \\
    &= \pl{\big( 1_{n+1} \otimes (1_B^{\otimes n} \otimes b) \big)} \sum_{\substack{d \in \cB \\ i \in \{1,\dotsc,n\}}} (-1)^{\sigma \bar d^\vee} s_i \dotsm s_{n-1} ( 1_B^{\otimes n} \otimes d^\vee ) \otimes (1_B^{\otimes n} \otimes d) s_{n-1} \dotsm s_i.
  \end{align*}
  The last expression above is the image of $1_n$ under the composition
  \[
    (n) \xrightarrow{\eta_L} (n)_{n-1}(n) \xrightarrow{\id \otimes \phi'(b)} (n)_{n-1}(n),
  \]
  which is the image under $\bF$ of the right-hand side of \eqref{eq:dotsup-a}.  The proof of \eqref{eq:dotsup-b} is straightforward and will be omitted.

  \details{
    If we label the rightmost region of \eqref{eq:dotsup-b} by $n$, then the left-hand side is mapped by $\bF$ to the composition
    \[
      (n)_{n-1}(n) \xrightarrow{\phi(b) \otimes \id} (n)_{n-1}(n) \xrightarrow{\varepsilon_\rR} (n).
    \]
    Since the bimodule $(n)_{n-1}(n)$ is generated by $1_n \otimes 1_n$, it suffices to consider the image of this element.  We have
    \[
      1_n \otimes 1_n \mapsto \left( 1_B^{\otimes (n-1)} \otimes b \right) \otimes 1_n \mapsto 1_B^{\otimes (n-1)} \otimes b.
    \]
    This is precisely the image of $1_n \otimes 1_n$ under the composition
    \[
      (n)_{n-1}(n) \xrightarrow{\id \otimes \phi'(b)} (n)_{n-1}(n) \xrightarrow{\varepsilon_\rR} (n).
    \]
    corresponding to the image under $\bF$ of the right-hand side of \eqref{eq:dotsup-b}.
  }

  \medskip

  \paragraph{\emph{Relations \eqref{eq:dotsdown-a}, \eqref{eq:dotsdown-b}}} The relation \eqref{eq:dotsdown-a} is a straightforward computation and will be omitted.
  \details{
    The left-hand side corresponds to the composition
    \[
      (n) \xrightarrow{\eta_\rR} \prescript{}{n}(n+1)_n \xrightarrow{\phi'(b) \otimes \id} \prescript{}{n}(n+1)_n,\quad 1_n \mapsto 1_{n+1} \mapsto 1_B^{\otimes n} \otimes b,
    \]
    and the right-hand side corresponds to the composition
    \[
      (n) \xrightarrow{\eta_\rR} \prescript{}{n}(n+1)_n \xrightarrow{\id \otimes \phi(b)} \prescript{}{n}(n+1)_n,\quad 1_n \mapsto 1_{n+1} \mapsto 1_B^{\otimes n} \otimes b.
    \]
  }
  If we label the rightmost region $n$, then the left-hand side of \eqref{eq:dotsdown-b} is mapped by $\bF$ to the composition
  \[
    \prescript{}{n}(n+1)_n \xrightarrow{\phi'(b)} \prescript{}{n}(n+1)_n \xrightarrow{\varepsilon_\rL} (n).
  \]
  Since the bimodule $\prescript{}{n}(n+1)_n$ is generated by the elements $1_B^{\otimes n} \otimes c$ and $(1_B^{\otimes n} \otimes c)s_n$, for $c \in B$, it suffices to consider the images of these elements.  We have
  \[
    1_B^{\otimes n} \otimes c \mapsto 1_B^{\otimes n} \otimes bc \mapsto \tr_B(bc) 1_B^{\otimes n},\quad (1_B^{\otimes n} \otimes c) s_n \mapsto (1_B^{\otimes n} \otimes bc)s_n \mapsto 0.
  \]
  On the other hand, the right-hand side of \eqref{eq:dotsdown-b} is mapped by $\bF$ to the composition
  \[
    \prescript{}{n}(n+1)_n \xrightarrow{\phi(\psiB(b))} \prescript{}{n}(n+1)_n \xrightarrow{\varepsilon_\rL} (n),
  \]
  and we have
  \begin{gather*}
    1_B^{\otimes n} \otimes c \mapsto (-1)^{\bar b \bar c} 1_B^{\otimes n} \otimes c\psiB(b) \mapsto (-1)^{\bar b \bar c} \tr_B(c\psiB(b)) 1_B^{\otimes n} = \tr_B(bc) 1_B^{\otimes n},\\
    (1_B^{\otimes n} \otimes c) s_n \mapsto (-1)^{\bar b \bar c} (1_B^{\otimes n} \otimes c) s_n (1_B^{\otimes n} \otimes \psiB(b)) \mapsto 0.
  \end{gather*}
  Thus the two maps are equal.

  \medskip

  \paragraph{\emph{Relations \eqref{eq:rel0a}--\eqref{eq:rel1b}}.}  For $k \in \N_+$, $\bF(\sP^{\otimes k}) = \prescript{}{n+k}(n+k)_n$.  We have a natural map
  \[
    A_k^\op \to \END \prescript{}{n+k}(n+k)_n,\quad a \mapsto \pr{\iota(a)},
  \]
  where $\iota \colon A_k \hookrightarrow A_n \otimes A_k \hookrightarrow A_{n+k}$ is the natural inclusion of $A_k$ into $A_{n+k}$ on the right.  Thus, the fact that $\bF$ preserves relations \eqref{eq:rel0a}--\eqref{eq:rel1b} follows from the corresponding relations in $A_2$ and $A_3$.

  \medskip

  \paragraph{\emph{Relations \eqref{eq:rel2a}, \eqref{eq:rel2b}}}  We prove \eqref{eq:rel2a}.  Relation \eqref{eq:rel2b} is similar but much easier.  If we label the rightmost region $n$, then both sides of \eqref{eq:rel2a} are mapped by $\bF$ to bimodule homomorphisms $\prescript{}{n}(n+1)_n \to \prescript{}{n}(n+1)_n$.  Since the bimodule $\prescript{}{n}(n+1)_n$ is generated by the elements $1_B^{\otimes n} \otimes c$ and $(1_B^{\otimes n} \otimes c)s_n$, for $c \in B$, it suffices to consider the images of these elements.  For the left-hand side, we compute
  \begin{gather*}
    1_B^{\otimes n} \otimes c \mapsto 0, \\
    (1_B^{\otimes n} \otimes c)s_n
    \mapsto 1_B^{\otimes n} \otimes (1_B^{\otimes (n-1)} \otimes c)
    \mapsto 1_B^{\otimes (n+1)} s_n(1_B^{\otimes (n-1)} \otimes c \otimes 1_B)
    = (1_B^{\otimes n} \otimes c)s_n.
  \end{gather*}
  We now consider the right-hand side.  The first term yields the identity map.  For the second term, the summand corresponding to a fixed $b \in \cB$ yields the compositions
  \begin{gather*}
    1_B^{\otimes n} \otimes c
    \mapsto 1_B^{\otimes n} \otimes b c
    \mapsto \tr_B(b c) 1_B^{\otimes n}
    \mapsto \tr_B(b c) 1_B^{\otimes (n+1)}
    \mapsto \tr_B(b c) 1_B^{\otimes n} \otimes b^\vee, \\
    (1_B^{\otimes n} \otimes c)s_n
    \mapsto (1_B^{\otimes n} \otimes b c) s_n
    \mapsto 0.
  \end{gather*}
  Thus we see that the right-hand side of \eqref{eq:rel2a} yields
  \begin{gather*}
    1_B^{\otimes n} \otimes c \mapsto 1_B^{\otimes n} \otimes c - \sum_{b \in \cB} \tr_B(b c) 1_B^{\otimes n} \otimes b^\vee \stackrel{\eqref{eq:dual-bases-decomp}}{=} 0, \\
    (1_B^{\otimes n} \otimes c)s_n \mapsto (1_B^{\otimes n} \otimes c)s_n.
  \end{gather*}

  \details{
    The left-hand side of \eqref{eq:rel2b} gives the composition
    \[
      (n)_{n-1}(n) \to \prescript{}{n}(n+1)_n \to (n)_{n-1}(n),\quad 1_n \otimes 1_n \mapsto s_n \mapsto 1_n \otimes 1_n,
    \]
    and so is equal to the identity map of the right-hand side.
  }

  \medskip

  \paragraph{\emph{Relations \eqref{eq:rel3a}, \eqref{eq:rel3b}}}  The proofs of \eqref{eq:rel3a} and \eqref{eq:rel3b} are straightforward computations and will be omitted.
  \details{
    The diagram
    \[
      \begin{tikzpicture}[>=stealth]
        \draw [->](0,0) arc (180:360:0.5cm);
        \draw (1,0) arc (0:180:0.5cm);
        \filldraw [blue](0,0) circle (2pt);
        \draw (0,0) node [anchor=east] {$b$};
        \draw (1.2,-0.3) node {$n$};
      \end{tikzpicture}
    \]
    is mapped by $\bF$ to the composition
   \begin{gather*}
      (n) \xrightarrow{\eta_\rR} \prescript{}{n}(n+1)_n \xrightarrow{\phi_n'(b)} \prescript{}{n}(n+1)_n \xrightarrow{\varepsilon_\rL} (n), \\
      a \mapsto a \mapsto (1_B^{\otimes n} \otimes b)a = (-1)^{\bar a \bar b} a \otimes b \mapsto \tr_B(b) a,\quad a \in (n),
    \end{gather*}
    where we have used the fact that $\tr_B(b)=0$ unless $\bar b = \sigma$.

    The diagram
    \[
      \begin{tikzpicture}[>=stealth]
        \draw (0,0) .. controls (0,.5) and (.7,.5) .. (.9,0);
        \draw (0,0) .. controls (0,-.5) and (.7,-.5) .. (.9,0);
        \draw (1,-1) .. controls (1,-.5) .. (.9,0);
        \draw[->] (.9,0) .. controls (1,.5) .. (1,1);
        \draw (1.3,-0.5) node {$n$};
      \end{tikzpicture}
    \]
    is mapped by $\bF$ to the composition
    \begin{gather*}
      (n+1)_n \xrightarrow{\eta_\rR} \prescript{}{n+1}(n+2)_n \xrightarrow{\pr{s_{n+1}}} \prescript{}{n+1}(n+2)_n \xrightarrow{\varepsilon_\rL} (n+1)_n, \\
      a \mapsto a \mapsto a s_{n+1} \mapsto 0,\quad a \in (n+1)_n.
    \end{gather*}
  }
\end{proof}

\section{Morphism spaces} \label{sec:morphism-spaces}

In this section, we examine the morphism spaces of the category $\cH'_B$.   We will freely use super isotopy invariance, including  relation \eqref{eq:supercomm}, without mention.

\subsection{Right curls}

While relation \eqref{eq:rel3b} asserts that a left curl is equal to zero, right curls are interesting morphisms in our category.  We will use an open dot to denote a right curl and an open dot labeled $d \in \N_+$ to denote $d$ right curls.
\[
  \begin{tikzpicture}[>=stealth,baseline={([yshift=-.5ex]current bounding box.center)}]
    \draw[->] (0,0) to (0,2);
    \redcircle{(0,1)};
  \end{tikzpicture}
  \ =\
  \begin{tikzpicture}[>=stealth,baseline={([yshift=-.5ex]current bounding box.center)}]
    \draw (2,1) .. controls (2,1.5) and (1.3,1.5) .. (1.1,1);
    \draw (2,1) .. controls (2,.5) and (1.3,.5) .. (1.1,1);
    \draw (1,0) .. controls (1,.5) .. (1.1,1);
    \draw[->] (1.1,1) .. controls (1,1.5) .. (1,2);
  \end{tikzpicture}
  \qquad \qquad
  \begin{tikzpicture}[>=stealth,baseline={([yshift=-.5ex]current bounding box.center)}]
    \draw[->] (0,0) to (0,2);
    \redcircle{(0,1)} node[anchor=west] {$d$};
  \end{tikzpicture}
  \ =\
  \begin{tikzpicture}[>=stealth,baseline={([yshift=-.5ex]current bounding box.center)}]
    \draw[->] (1.5,0) to (1.5,2);
    \redcircle{(1.5,.5)};
    \redcircle{(1.5,.8)};
    \redcircle{(1.5,1.1)};
    \redcircle{(1.5,1.4)};
    \draw (2.5,1) node {$\Bigg\}\ d$ dots};
  \end{tikzpicture}
\]
It follows from super isotopy invariance that open dots can be moved freely along strands, including through cups and caps.

\begin{lem}
  We have the following equalities of diagrams:
  \begin{equation} \label{eq:curl-Nakayama}
    \begin{tikzpicture}[>=stealth,baseline={([yshift=-.5ex]current bounding box.center)}]
      \draw[->] (0,0) to (0,2);
      \bluedot{(0,0.6)} node[color=black, anchor=east] {$b$};
      \redcircle{(0,1.4)};
    \end{tikzpicture}
    \ = (-1)^{\sigma \bar b} \ \
    \begin{tikzpicture}[>=stealth,baseline={([yshift=-.5ex]current bounding box.center)}]
      \draw[->] (0,0) to (0,2);
      \bluedot{(0,1.4)} node[color=black, anchor=west] {$\psiB(b)$};
      \redcircle{(0,0.6)};
    \end{tikzpicture}
    \qquad \qquad \qquad
    \begin{tikzpicture}[>=stealth,baseline={([yshift=-.5ex]current bounding box.center)}]
      \draw[<-] (0,0) to (0,2);
      \bluedot{(0,1.4)} node[color=black, anchor=east] {$b$};
      \redcircle{(0,0.6)};
    \end{tikzpicture}
    \ = (-1)^{\sigma \bar b} \ \
    \begin{tikzpicture}[>=stealth,baseline={([yshift=-.5ex]current bounding box.center)}]
      \draw[<-] (0,0) to (0,2);
      \bluedot{(0,0.6)} node[color=black, anchor=west] {$\psiB(b)$};
      \redcircle{(0,1.4)};
    \end{tikzpicture}
  \end{equation}
\end{lem}

\begin{proof}
  This follows immediately from \eqref{eq:dotsup-a}--\eqref{eq:dotsdown-b}, \eqref{eq:rel0a}, and \eqref{eq:rel0b}.  The sign comes from exchanging the heights of the dot labelled $b$ and the left cup involved in the right curl.
\end{proof}

\begin{lem}
  The relation
  \begin{equation} \label{eq:triple-point}
    \begin{tikzpicture}[>=stealth,baseline={([yshift=-.5ex]current bounding box.center)}]
      \draw (0,0) -- (2,2);
      \draw (2,0) -- (0,2);
      \draw (1,0) .. controls (0,1) .. (1,2);
    \end{tikzpicture}
    \ = \
    \begin{tikzpicture}[>=stealth,baseline={([yshift=-.5ex]current bounding box.center)}]
      \draw (3,0) -- (5,2);
      \draw (5,0) -- (3,2);
      \draw (4,0) .. controls (5,1) .. (4,2);
    \end{tikzpicture}
  \end{equation}
  holds for all possible orientations of the strands.
\end{lem}

\begin{proof}
  The case where all strands are pointing up is \eqref{eq:rel1a}.  The case where all strands are pointing down follows by super isotopy invariance.  The case where the left and middle strands are pointing up and the right strand is pointing down follows easily from \eqref{eq:rel1a} and \eqref{eq:rel1b}.  (Here and in what follows, the ``left strand'' is the strand whose top endpoint is on the left, etc.)
  \details{
    \[
      \begin{tikzpicture}[>=stealth,baseline={([yshift=-.5ex]current bounding box.center)}]
        \draw[<-] (0,0) to (2,2);
        \draw[->] (2,0) to (0,2);
        \draw[<-] (1,2) .. controls (0,1) .. (1,0);
      \end{tikzpicture}
      \ \stackrel{\eqref{eq:rel1b}}{=}\
      \begin{tikzpicture}[>=stealth,baseline={([yshift=-.5ex]current bounding box.center)}]
        \draw[<-] (0,0) to (2,2);
        \draw[->] (2,0) to (0,2);
        \draw[<-] (1,2) .. controls (0,1) and (0.5,-.1) .. (1.3,.7) ..
        controls (2.1,1.5) and (2.3,.8) .. (1,0);
      \end{tikzpicture}
      \stackrel{\eqref{eq:rel1a}}{=}
      \begin{tikzpicture}[>=stealth,baseline={([yshift=-.5ex]current bounding box.center)}]
        \draw[<-] (0,0) to (2,2);
        \draw[->] (2,0) to (0,2);
        \draw[<-] (1,2) .. controls (0.2,1.5) and (-0.1,.5) .. (0.8,1.2) ..
        controls (1.7,2.1) and (2,1) .. (1,0);
      \end{tikzpicture}
      \stackrel{\eqref{eq:rel1b}}{=}
      \begin{tikzpicture}[>=stealth,baseline={([yshift=-.5ex]current bounding box.center)}]
        \draw[<-] (0,0) to (2,2);
        \draw[->] (2,0) to (0,2);
        \draw[<-] (1,2) .. controls (2,1) .. (1,0);
      \end{tikzpicture}
    \]
  }
  We then have
  \begin{multline*}
    \begin{tikzpicture}[>=stealth,baseline={([yshift=-.5ex]current bounding box.center)}]
      \draw[->] (0,0) to (2,2);
      \draw[->] (2,0) to (0,2);
      \draw[->] (1,2) .. controls (0,1) .. (1,0);
    \end{tikzpicture}
    \ \stackrel{\eqref{eq:rel2a}}{=}\
    \begin{tikzpicture}[>=stealth,baseline={([yshift=-.5ex]current bounding box.center)}]
      \draw[->] (3.5,0) to (5.5,2);
      \draw[->] (5.5,0) to (3.5,2);
      \draw[->] (4.5,2) .. controls (3.5,1) and (4,-.1) .. (4.8,.7) ..
      controls (5.6,1.5) and (5.8,.8) .. (4.5,0);
    \end{tikzpicture}
    + \sum_{b \in \cB}
    \begin{tikzpicture}[>=stealth,baseline={([yshift=-.5ex]current bounding box.center)}]
      \draw[->] (6.5,0) to (8.5,2);
      \draw[->] (8.5,0) .. controls (8,.5) .. (7.5,0);
      \draw[->] (7.5,2) .. controls (6.5,1) and (7.2,0) .. (7.7,.5) .. controls (8.2,1) and (7,1.5) .. (6.5,2);
      \bluedot{(7.8,0.7)} node[color=black, anchor=west] {$b^\vee$};
      \bluedot{(7.7,0.2)} node[color=black, anchor=north west] {$b$};
    \end{tikzpicture}
    \\
    \stackrel{\substack{\eqref{eq:rel0b},\eqref{eq:rel1b}}}{=}
    \begin{tikzpicture}[>=stealth,baseline={([yshift=-.5ex]current bounding box.center)}]
      \draw[->] (3.5,0) to (5.5,2);
      \draw[->] (5.5,0) to (3.5,2);
      \draw[->] (4.5,2) .. controls (3.7,1.5) and (3.4,.5) .. (4.3,1.2) ..
      controls (5.2,2.1) and (5.5,1) .. (4.5,0);
    \end{tikzpicture}
    \ + \sum_{b \in \cB}\
    \begin{tikzpicture}[>=stealth,baseline={([yshift=-.5ex]current bounding box.center)}]
      \draw[->] (7,0) to (9,2);
      \draw[->] (9,0) arc(0:180:.5);
      \draw[->] (8,2) .. controls (7,1.5) and (7.4,1.1) .. (7.7,1.1)
      .. controls (7.9,1.1) and (8,1.5) .. (7,2);
      \bluedot{(7.3,1.85)} node[color=black, anchor=south west] {$b^\vee$};
      \bluedot{(8.1,0.3)} node[color=black, anchor=east] {$b$};
    \end{tikzpicture}
    \stackrel{\eqref{eq:rel2b},\eqref{eq:rel3b}}{=}
    \begin{tikzpicture}[>=stealth,baseline={([yshift=-.5ex]current bounding box.center)}]
      \draw[->] (1,0) to (3,2);
      \draw[->] (3,0) to (1,2);
      \draw[->] (2,2) .. controls (3,1) .. (2,0);
    \end{tikzpicture}
  \end{multline*}
  where in the second equality we use the case where the left and middle strands are pointing up and the right strand is pointing down (and super isotopy invariance).  The case where the middle strand is pointing up and the outer left and right strands are pointing down follows by super isotopy invariance.  The remaining cases are proved similarly.
\end{proof}

\begin{lem} \label{lem:dot-crossing-moves}
  We have the following equalities of diagrams:
  \begin{equation} \label{eq:circle-leftup-slide}
    \begin{tikzpicture}[>=stealth,baseline={([yshift=-.5ex]current bounding box.center)}]
      \draw[->] (0,0) to (2,2);
      \draw[->] (2,0) to (0,2);
      \redcircle{(0.5,1.5)};
    \end{tikzpicture}
    \ =\
    \begin{tikzpicture}[>=stealth,baseline={([yshift=-.5ex]current bounding box.center)}]
      \draw[->] (3,0) to (5,2);
      \draw[->] (5,0) to (3,2);
      \redcircle{(4.5,.5)};
    \end{tikzpicture}
    +\ \sum_{b \in \cB}\
    \begin{tikzpicture}[>=stealth,baseline={([yshift=-.5ex]current bounding box.center)}]
      \draw[->] (8,0) to (8,2);
      \draw[->] (9,0) to (9,2);
      \bluedot{(8,1.2)} node[color=black,anchor=east] {$b^\vee$};
      \bluedot{(9,0.8)} node[color=black,anchor=west] {$b$};
    \end{tikzpicture}
  \end{equation}
  \begin{equation} \label{eq:circle-rightup-slide}
    \begin{tikzpicture}[>=stealth,baseline={([yshift=-.5ex]current bounding box.center)}]
      \draw[->] (0,0) to (2,2);
      \draw[->] (2,0) to (0,2);
      \redcircle{(0.5,0.5)};
    \end{tikzpicture}
    \ =\
    \begin{tikzpicture}[>=stealth,baseline={([yshift=-.5ex]current bounding box.center)}]
      \draw[->] (3,0) to (5,2);
      \draw[->] (5,0) to (3,2);
      \redcircle{(4.5,1.5)};
    \end{tikzpicture}
    +\ \sum_{b \in \cB}\
    \begin{tikzpicture}[>=stealth,baseline={([yshift=-.5ex]current bounding box.center)}]
      \draw[->] (7.5,0) to (7.5,2);
      \draw[->] (8.5,0) to (8.5,2);
      \bluedot{(7.5,0.8)} node[color=black,anchor=east] {$b$};
      \bluedot{(8.5,1.2)} node[color=black,anchor=west] {$b^\vee$};
    \end{tikzpicture}
  \end{equation}
\end{lem}

\begin{proof}
  We prove \eqref{eq:circle-leftup-slide}.  The proof of \eqref{eq:circle-rightup-slide} is analogous.
  \begin{multline*}
    \begin{tikzpicture}[>=stealth,baseline={([yshift=-.5ex]current bounding box.center)}]
      \draw[->] (3,0) to (5,2);
      \draw[->] (5,0) to (4,1) .. controls (3.5,1.5) and (3.7,1.8) ..
      (4,1.5) .. controls (4.3,1.2) and (3.7,1.2) .. (3,2);
    \end{tikzpicture}
    \ \stackrel{\eqref{eq:dotsdown-b},\eqref{eq:rel2a}}{=}
    \begin{tikzpicture}[>=stealth,baseline={([yshift=-.5ex]current bounding box.center)}]
      \draw[->] (9,0) .. controls (8,.5) and (7.2,1.5) .. (8.2,1.5) ..
      controls (9.2,1.5) and (8.5,-.25) .. (7,2);
      \draw[->] (7,0) .. controls (8,.5) and (9,.3) .. (8.2,1.3) ..
      controls (7.6,1.8) and (8.5,1.75) .. (9,2);
    \end{tikzpicture}
    +\ \sum_{b \in \cB}
    \begin{tikzpicture}[>=stealth,baseline={([yshift=-.5ex]current bounding box.center)}]
      \draw[->] (11.2,0) .. controls (12.2,1) .. (11.2,2);
      \draw[->] (12.2,0) .. controls (11.2,1) .. (12.2,2);
      \bluedot{(11.95,1.76)} node[color=black,anchor=south] {$b^\vee$};
      \bluedot{(11.95,1)} node[color=black,anchor=west] {$\psiB(b)$};
    \end{tikzpicture}
    \\
    \stackrel{\eqref{eq:rel0a},\eqref{eq:triple-point}}{=}\
    \begin{tikzpicture}[>=stealth,baseline={([yshift=-.5ex]current bounding box.center)}]
      \draw[->] (9,0) .. controls (8,.5) and (7.2,1.5) .. (8.2,1.5) ..
      controls (9.2,1.5) and (8.5,-.25) .. (7,2);
      \draw[->] (7,0) .. controls (7,1) and (7.5,1.6) .. (8.2,1.3) ..
      controls (8.9,1) and (9,1.5) .. (9,2);
    \end{tikzpicture}
    +\ \sum_{b \in \cB}
    \begin{tikzpicture}[>=stealth,baseline={([yshift=-.5ex]current bounding box.center)}]
      \draw[->] (11.2,0) .. controls (12.2,1) .. (11.2,2);
      \draw[->] (12.2,0) .. controls (11.2,1) .. (12.2,2);
      \bluedot{(11.95,1.76)} node[color=black,anchor=south] {$b^\vee$};
      \bluedot{(11.5,0.3)} node[color=black,anchor=south east] {$\psiB(b)$};
    \end{tikzpicture}
    \ \stackrel{\eqref{eq:rel1b}}{=}
    \begin{tikzpicture}[>=stealth,baseline={([yshift=-.5ex]current bounding box.center)}]
      \draw[->] (14,0) .. controls (13.5,.5) and (13.7,.8) ..
      (14,.5) .. controls (14.3,.2) and (13.7,.2) .. (13,1) to (12,2);
      \draw[->] (12,0) to (14,2);
    \end{tikzpicture}
    +\ \sum_{b \in \cB}
    \begin{tikzpicture}[>=stealth,baseline={([yshift=-.5ex]current bounding box.center)}]
      \draw[->] (8,0) to (8,2);
      \draw[->] (9,0) to (9,2);
      \bluedot{(8,0.8)} node[color=black,anchor=east] {$\psiB(b)$};
      \bluedot{(9,1.2)} node[color=black,anchor=west] {$b^\vee$};
    \end{tikzpicture}
  \end{multline*}
  Relation \eqref{eq:circle-leftup-slide} then follows from \eqref{eq:b-bvee-swap} and the fact that diagrams are linear in the dots.
\end{proof}

\begin{cor}
  We have the following equalities of diagrams:
  \begin{equation} \label{eq:circle-mult-leftup-slide}
    \begin{tikzpicture}[>=stealth,baseline={([yshift=-.5ex]current bounding box.center)}]
      \draw[->] (0,0) to (2,2);
      \draw[->] (2,0) to (0,2);
      \redcircle{(0.5,1.5)} node[color=black,anchor=north east] {$d$};
    \end{tikzpicture}
    \ =\
    \begin{tikzpicture}[>=stealth,baseline={([yshift=-.5ex]current bounding box.center)}]
      \draw[->] (3,0) to (5,2);
      \draw[->] (5,0) to (3,2);
      \redcircle{(4.5,.5)} node[color=black,anchor=north east] {$d$};
    \end{tikzpicture}
    \ +\ \sum_{k=0}^{d-1} \sum_{b \in \cB}\
    \begin{tikzpicture}[>=stealth,baseline={([yshift=-.5ex]current bounding box.center)}]
      \draw[->] (0,0) to (0,2);
      \draw[->] (1,0) to (1,2);
      \bluedot{(0,1.2)} node[color=black,anchor=east] {$b^\vee$};
      \bluedot{(1,0.8)} node[color=black,anchor=west] {$b$};
      \redcircle{(0,1.6)} node[color=black,anchor=east] {$k$};
      \redcircle{(1,0.4)} node[color=black,anchor=west] {$d-k-1$};
    \end{tikzpicture}
  \end{equation}
  \begin{equation} \label{eq:circle-mult-rightup-slide}
    \begin{tikzpicture}[>=stealth,baseline={([yshift=-.5ex]current bounding box.center)}]
      \draw[->] (0,0) to (2,2);
      \draw[->] (2,0) to (0,2);
      \redcircle{(0.5,0.5)} node[color=black,anchor=south east] {$d$};
    \end{tikzpicture}
    \ =\
    \begin{tikzpicture}[>=stealth,baseline={([yshift=-.5ex]current bounding box.center)}]
      \draw[->] (3,0) to (5,2);
      \draw[->] (5,0) to (3,2);
      \redcircle{(4.5,1.5)} node[color=black,anchor=south east] {$d$};
    \end{tikzpicture}
    \ +\ \sum_{k=0}^{d-1} \sum_{b \in \cB}\
    \begin{tikzpicture}[>=stealth,baseline={([yshift=-.5ex]current bounding box.center)}]
      \draw[->] (0,0) to (0,2);
      \draw[->] (1,0) to (1,2);
      \bluedot{(0,0.8)} node[color=black,anchor=east] {$b$};
      \bluedot{(1,1.2)} node[color=black,anchor=west] {$b^\vee$};
      \redcircle{(0,0.4)} node[color=black,anchor=east] {$k$};
      \redcircle{(1,1.6)} node[color=black,anchor=west] {$d-k-1$};
    \end{tikzpicture}
  \end{equation}
\end{cor}

\begin{proof}
  This follows easily from repeated application of \eqref{eq:circle-leftup-slide} and \eqref{eq:circle-rightup-slide}.
\end{proof}

Under the functor $\bF$, the image of
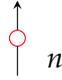
\[
  \begin{tikzpicture}[>=stealth,baseline={([yshift=-.5ex]current bounding box.center)}]
    \draw[->] (0,0) to (0,1);
    \draw (0.5,0.2) node {$n$};
    \redcircle{(0,0.5)};
  \end{tikzpicture}
\]
is the zero map for $n \le 0$ and, for $n \in \N_+$, is the degree $(\delta,\sigma)$ bimodule homomorphism
\begin{gather}
  \pr{J_n} \colon (n+1)_n \to \{-\delta,\sigma\}(n+1)_n,\quad \text{where} \nonumber \\
  J_n := \sum_{\substack{b \in \cB \\ i \in \{1,\dotsc,n\}}} (-1)^{\sigma \bar b^\vee} \left( 1_B^{\otimes(i-1)} \otimes b^\vee \otimes 1_B^{\otimes(n-i)} \otimes b \right) (i,n+1),
\end{gather}
where $(i,n+1)$ denotes the transposition of $i$ and $n+1$.  By convention, we set $J_n = 0$ for $n \le 0$.
\details{
  The result for $n \le 0$ follows immediately from the fact that the diagram corresponding to the right curl has a region with a negative label in this case.  Now assume that $n \in \N_+$.  Then the right curl is the composition
  \begin{multline*}
    (n+1)_n \cong (n+1)_n(n)
    \xrightarrow{\id \otimes \eta_\rL} (n+1)_n(n)_{n-1}(n)
    \cong (n+1)_{n-1}(n)
    \xrightarrow{\pr{s_n} \otimes \id} (n+1)_{n-1}(n) \\
    \cong (n+1)_n(n)_{n-1}(n)
    \xrightarrow{\id \otimes \varepsilon_\rR} (n+1)_n(n)
    \cong (n+1)_n.
  \end{multline*}
  Under this composition, for $a \in (n+1)_n$, we have
  \begin{multline*}
    a \mapsto \sum_{\substack{b \in \cB \\ i \in \{1,\dotsc,n\}}} (-1)^{\sigma (\bar a + \bar b^\vee)} a s_i \dotsm s_{n-1} \left( 1_B^{\otimes (n-1)} \otimes b^\vee \otimes 1_B \right) \otimes \left( 1_B^{\otimes (n-1)} \otimes b \right) s_{n-1} \dotsm s_i \\
    \mapsto
    \sum_{\substack{b \in \cB \\ i \in \{1,\dotsc,n\}}} (-1)^{\sigma (\bar a + \bar b^\vee)} a s_i \dotsm s_{n-1} \left( 1_B^{\otimes (n-1)} \otimes b^\vee \otimes 1_B \right) s_n \otimes \left( 1_B^{\otimes (n-1)} \otimes b \right) s_{n-1} \dotsm s_i.
  \end{multline*}
  The result then follows from the fact that the degree of $J_n$ is $(\delta,\sigma)$ and that $s_i \dotsm s_n \dotsm s_i$ is the transposition $(i,n+1)$.
}
It follows that the image of
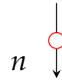
\[
  \begin{tikzpicture}[>=stealth,baseline={([yshift=-.5ex]current bounding box.center)}]
    \draw[<-] (0,0) to (0,1);
    \draw (-0.5,0.2) node {$n$};
    \redcircle{(0,0.5)};
  \end{tikzpicture}
\]
is the zero map for $n \le 0$ and, for $n \in \N+$, is the degree $(\delta,\sigma)$ bimodule homomorphism
\begin{gather}
  \pl{J_n} \colon \prescript{}{n}(n+1) \to \{-\delta,\sigma\} \prescript{}{n}(n+1).
\end{gather}

\begin{rem} \label{rem:Jucys-Murphy}
  Note that in the special case $B=\F$, we have that $J_n$ is the Jucys--Murphy element, which plays a crucial role in the representation theory of the symmetric group.  When $B = \Cl$, we obtain the analogous elements of the Sergeev algebra.  (See, for example, \cite[(13.22)]{Kle05}.)  One should think of $J_n$ as a generalization of the Jucys--Murphy element to more general wreath product algebras.  It would be interesting to investigate which properties of the usual Jucys--Murphy elements hold in the more general setting.
\end{rem}

\subsection{Bubbles}

For $b \in B$ and $d \in \N$, define the \emph{bubbles}
\begin{equation}
  \tilde c_{b,d} =
  \begin{tikzpicture}[>=stealth,baseline={([yshift=-.5ex]current bounding box.center)}]
    \draw[->] (0,0) arc (90:470:.5);
    \bluedot{(-0.45,-0.3)} node[color=black,anchor=east] {$b$};
    \redcircle{(0.45,-0.7)} node[color=black,anchor=west] {$d$};
  \end{tikzpicture}
  \qquad \text{and} \qquad
  c_{b,d} =
  \begin{tikzpicture}[>=stealth,baseline={([yshift=-.5ex]current bounding box.center)}]
    \draw[<-] (0,0) arc (90:470:.5);
    \bluedot{(-0.45,-0.3)} node[color=black,anchor=east] {$b$};
    \redcircle{(0.45,-0.7)} node[color=black,anchor=west] {$d$};
  \end{tikzpicture}
  \ .
\end{equation}

\begin{lem} \label{lem:circle-index-relations}
  We have
  \begin{equation} \label{eq:circle-Nakayama}
    \tilde{c}_{b,d} = \tilde{c}_{\psiB(b),d},\quad c_{b,d} = c_{\psiB(b),d} \quad \text{for all } b \in B,\ d \in \N.
  \end{equation}
  and, if $\sigma =1$, then
  \begin{equation} \label{eq:circle-d-deven}
    \tilde{c}_{b,d} = 0 = c_{b,d} \quad \text{for all } b \in B,\ d \in 2\N_+.
  \end{equation}
\end{lem}

\begin{proof}
  Equations~\eqref{eq:circle-Nakayama} follow immediately from \eqref{eq:dotsup-a}--\eqref{eq:dotsdown-b} and the fact that the open dots can be slid around the circles.  For $b \in B$ and $d \ge 2$, we have
  \begin{multline*}
    c_{b,d} =
    \begin{tikzpicture}[>=stealth,baseline={([yshift=-.5ex]current bounding box.center)}]
      \draw[<-] (0,0) arc (90:470:.5);
      \bluedot{(-0.45,-0.3)} node[color=black,anchor=east] {$b$};
      \redcircle{(0.45,-0.7)} node[color=black,anchor=west] {$d$};
    \end{tikzpicture}
    \ = (-1)^{\sigma b}
    \begin{tikzpicture}[>=stealth,baseline={([yshift=-.5ex]current bounding box.center)}]
      \draw[<-] (0,0) arc (90:470:.5);
      \bluedot{(-0.5,-0.5)} node[color=black,anchor=east] {$b$};
      \redcircle{(0.36,-0.85)} node[color=black,anchor=west] {$d-1$};
      \redcircle{(-0.36,-0.15)};
    \end{tikzpicture}
    \ =\
    \begin{tikzpicture}[>=stealth,baseline={([yshift=-.5ex]current bounding box.center)}]
      \draw[<-] (0,0) arc (90:470:.5);
      \bluedot{(-0.36,-0.15)} node[color=black,anchor=south east] {$\psiB(b)$};
      \redcircle{(0.36,-0.85)} node[color=black,anchor=west] {$d-1$};
      \redcircle{(-0.5,-0.5)};
    \end{tikzpicture}
    \ = (-1)^{(d-1)\sigma}\
    \begin{tikzpicture}[>=stealth,baseline={([yshift=-.5ex]current bounding box.center)}]
      \draw[<-] (0,0) arc (90:470:.5);
      \bluedot{(-0.36,-0.15)} node[color=black,anchor=south east] {$\psiB(b)$};
      \redcircle{(0.5,-0.5)} node[color=black,anchor=west] {$d-1$};
      \redcircle{(-0.36,-0.85)};
    \end{tikzpicture}
    \\
    = (-1)^{(d-1)\sigma}\
    \begin{tikzpicture}[>=stealth,baseline={([yshift=-.5ex]current bounding box.center)}]
      \draw[<-] (0,0) arc (90:470:.5);
      \bluedot{(-0.45,-0.3)} node[color=black,anchor=east] {$\psiB(b)$};
      \redcircle{(0.45,-0.7)} node[color=black,anchor=west] {$d$};
    \end{tikzpicture}
    \ = (-1)^{(d-1)\sigma} c_{\psiB(b),d} \stackrel{\eqref{eq:circle-Nakayama}}{=} (-1)^{(d-1) \sigma} c_{b,d},
  \end{multline*}
  from which it follows that $c_{b,d}=0$ when $\sigma=1$ and $d \in 2\N_+$.  The proof of the first equality in \eqref{eq:circle-d-deven} is analogous.
\end{proof}

\begin{prop} \label{prop:circle-conversion}
  For $b \in B$ and $d \in \N$, we have
  \begin{equation}
    \tilde{c}_{b,d+1} = \sum_{k=0}^{d-1} \sum_{f \in \cB} (-1)^{\bar f + \bar b \bar f + (k+1) \sigma \bar f + (d+kd+1)\sigma} \, \tilde{c}_{f \psiB (b),k} c_{f^\vee,d-k-1}.
  \end{equation}
  Note, in particular, that this implies $\tilde c_{b,1}=0$.
\end{prop}

\begin{proof}
  \begin{multline*}
    \tilde c_{b,d+1}=
    \begin{tikzpicture}[>=stealth,baseline={([yshift=-.5ex]current bounding box.center)}]
      \draw[-<] (1,1) .. controls (1,2) and (1.5,2) .. (2,1) .. controls (2.5,0) and (3,0) ..
      (3,1);
      \draw (3,1) .. controls (3,2) and (2.5,2) .. (2,1) .. controls (1.5,0) and (1,0) .. (1,1);
      \redcircle{(1.02,0.7)} node[color=black,anchor=east] {$d$};
      \bluedot{(1.02,1.3)} node[color=black,anchor=east] {$b$};
    \end{tikzpicture}
    \ \stackrel{\eqref{eq:circle-mult-rightup-slide}}{=}
    \begin{tikzpicture}[>=stealth,baseline={([yshift=-.5ex]current bounding box.center)}]
      \draw[-<] (4,1) .. controls (4,2) and (4.5,2) .. (5,1) .. controls (5.5,0) and (6,0) ..
      (6,1);
      \draw (6,1) .. controls (6,2) and (5.5,2) .. (5,1) .. controls (4.5,0) and (4,0) .. (4,1);
      \redcircle{(5.18,1.3)} node[color=black, anchor=south] {$d$};
      \bluedot{(4.1,1.55)} node[color=black,anchor=east] {$b$};
    \end{tikzpicture}
    \ + \sum_{k=0}^{d-1} \sum_{f \in \cB}\
    \begin{tikzpicture}[>=stealth,baseline={([yshift=-.5ex]current bounding box.center)}]
      \draw[->] (9.5,1) ellipse (0.5 and 0.75);
      \draw[->] (9.5,0.25) to (9.51,0.25);
      \draw (12.5,1) ellipse (0.5 and 0.75);
      \draw[->] (12.49,1.75) to (12.5,1.75);
      \redcircle{(9.75,0.35)} node[color=black,anchor=north west] {$k$};
      \bluedot{(9.2,1.6)} node[color=black,anchor=east] {$b$};
      \bluedot{(9.92,.6)} node[color=black,anchor=west] {$f$};
      \redcircle{(12.05,1.3)} node[color=black,anchor=east] {$d-k-1$};
      \bluedot{(12.01,0.9)} node[color=black,anchor=east] {$f^\vee$};
    \end{tikzpicture}
    \\
    \stackrel{\substack{\eqref{eq:dotsdown-b} \\ \eqref{eq:rel0b}, \eqref{eq:collision-down}}}{=} (-1)^{d \sigma \bar b}\
    \begin{tikzpicture}[>=stealth,baseline={([yshift=-.5ex]current bounding box.center)}]
      \draw[-<] (4,1) .. controls (4,2) and (4.5,2) .. (5,1) .. controls (5.5,0) and (6,0) ..
      (6,1);
      \draw (6,1) .. controls (6,2) and (5.5,2) .. (5,1) .. controls (4.5,0) and (4,0) .. (4,1);
      \redcircle{(5.3,1.5)} node[color=black, anchor=south] {$d$};
      \bluedot{(5.3,0.5)} node[color=black,anchor=north east] {$\psi(b)$};
    \end{tikzpicture}
    \ + \sum_{k=0}^{d-1} \sum_{f \in \cB} (-1)^{\bar f + \bar b \bar f + \sigma \bar f + (d-k-1)\sigma}\
    \begin{tikzpicture}[>=stealth,baseline={([yshift=-.5ex]current bounding box.center)}]
      \draw[->] (9.5,1) ellipse (0.5 and 0.75);
      \draw[->] (9.5,0.25) to (9.51,0.25);
      \draw (11.5,1) ellipse (0.5 and 0.75);
      \draw[->] (11.49,1.75) to (11.5,1.75);
      \redcircle{(9.75,0.35)} node[color=black,anchor=north west] {$k$};
      \bluedot{(9.2,1.6)} node[color=black,anchor=east] {$\psiB^{-1}(f)b$};
      \redcircle{(11.99,0.9)} node[color=black,anchor=west] {$d-k-1$};
      \bluedot{(11.05,1.3)} node[color=black,anchor=east] {$f^\vee$};
    \end{tikzpicture}
    \\
    \stackrel{\eqref{eq:rel3b}}{=}
    \sum_{k=0}^{d-1} \sum_{f \in \cB} (-1)^{\bar f + \bar b \bar f + (k+1) \sigma \bar f + (d+kd+1)\sigma}\
    \begin{tikzpicture}[>=stealth,baseline={([yshift=-.5ex]current bounding box.center)}]
      \draw[->] (9.5,0.5) arc(-90:270:.5);
      \draw[-<] (11.5,0) arc(90:450:.5);
      \redcircle{(9.7,0.55)} node[color=black,anchor=north] {$k$};
      \bluedot{(9,1)} node[color=black,anchor=east] {$\psiB^{-1}(f) b$};
      \redcircle{(12,-0.5)} node[color=black,anchor=west] {$d-k-1$};
      \bluedot{(11.1,-0.2)} node[color=black,anchor=east] {$f^\vee$};
    \end{tikzpicture}
  \end{multline*}
  The result then follows from \eqref{eq:circle-Nakayama}.
\end{proof}

\begin{prop} \label{prop:bubble-slide}
  We have the following equality of diagrams:
  \begin{multline*}
    \begin{tikzpicture}[>=stealth,baseline={([yshift=-.5ex]current bounding box.center)}]
      \draw[<-] (0,0) arc (90:470:.5);
      \draw[->] (1,-1.5) to (1,0.5);
      \bluedot{(-0.45,-0.3)} node[color=black,anchor=east] {$b$};
      \redcircle{(0.45,-0.7)} node[color=black,anchor=west] {$d$};
    \end{tikzpicture}
    \ =\
    \begin{tikzpicture}[>=stealth,baseline={([yshift=-.5ex]current bounding box.center)}]
      \draw[<-] (0,0) arc (90:470:.5);
      \draw[->] (-1.1,-1.5) to (-1.1,0.5);
      \bluedot{(-0.45,-0.3)} node[color=black,anchor=east] {$b$};
      \redcircle{(0.45,-0.7)} node[color=black,anchor=west] {$d$};
    \end{tikzpicture}
    \ +\ \sum_{k=0}^{d-1} \sum_{f \in \cB} (-1)^{\bar f + \sigma(dk-k+1 + d \bar b + \bar f)}\
    \begin{tikzpicture}[>=stealth,baseline={([yshift=-.5ex]current bounding box.center)}]
      \draw[->] (0,0) to (0,2);
      \bluedot{(0,0.4)} node[color=black,anchor=west] {$f^\vee$};
      \bluedot{(0,1)} node[color=black,anchor=west] {$\psiB^{k-d}(b) \psiB^{-d}(f)$};
      \redcircle{(0,1.6)} node[color=black,anchor=west] {$d$};
    \end{tikzpicture}
    \ +\ \sum_{f \in \cB} (-1)^{\bar f + d \sigma \bar b + \sigma}\
    \begin{tikzpicture}[>=stealth,baseline={([yshift=-.5ex]current bounding box.center)}]
      \draw[->] (0,0) to (0,2);
      \bluedot{(0,1)} node[color=black,anchor=west] {$b \psiB^{-d}(f)$};
      \redcircle{(0,1.6)} node[color=black,anchor=west] {$d$};
      \bluedot{(0,0.4)} node[color=black,anchor=west] {$f^\vee$};
    \end{tikzpicture}
    \\
    \qquad \qquad \qquad
    \ -\ \sum_{k=0}^{d-1} \sum_{j = 0}^{d-k-2} \sum_{f,g \in \cB}  (-1)^{\bar f + \bar \sigma (dk-k+1 + (d+1)\bar f)}\
    \begin{tikzpicture}[>=stealth,baseline={([yshift=-.5ex]current bounding box.center)}]
      \draw[->] (0,0) to (0,3);
      \draw[<-] (1,1.5) arc (90:470:0.5);
      \bluedot{(0,2.7)} node[color=black,anchor=east] {$\psiB^k(b) f$};
      \redcircle{(0,2.2)} node[color=black,anchor=east] {$j+k$};
      \bluedot{(0,1.7)} node[color=black,anchor=east] {$g^\vee$};
      \bluedot{(0,0.2)} node[color=black,anchor=east] {$f^\vee$};
      \bluedot{(0.5,1)} node[color=black,anchor=west] {$g$};
      \redcircle{(1.2,0.55)};
      \draw (1.6,0.5) node[color=black,anchor=north] {$d-k-j-2$};
    \end{tikzpicture}
  \end{multline*}
\end{prop}

\begin{proof}
  \begin{multline*}
     (-1)^{d \sigma \bar b}\
    \begin{tikzpicture}[>=stealth,baseline={([yshift=-.5ex]current bounding box.center)}]
      \draw[<-] (0,0) arc (90:470:.5);
      \draw[->] (1,-1.5) to (1,0.5);
      \bluedot{(-0.45,-0.3)} node[color=black,anchor=east] {$b$};
      \redcircle{(0.45,-0.7)} node[color=black,anchor=west] {$d$};
    \end{tikzpicture}
    \ \stackrel{\eqref{eq:rel2a}}{=}\
    \begin{tikzpicture}[>=stealth,baseline={([yshift=-.5ex]current bounding box.center)}]
      \draw[->] (0,0) .. controls (0,0.5) and (1.5,1) .. (2,1) .. controls (2.5,1) and (2.5,-1) .. (2,-1) .. controls (1.5,-1) and (0,-0.5) .. (0,0);
      \draw[->] (1,-1.2) .. controls (1.5,-0.5) and (1.5,0.5) .. (1,1.2);
      \redcircle{(0.6,0.6)} node[color=black,anchor=south] {$d$};
      \bluedot{(0.6,-0.6)} node[color=black,anchor=north] {$b$};
    \end{tikzpicture}
    \ + \sum_{f \in \cB}\
    \begin{tikzpicture}[>=stealth,baseline={([yshift=-.5ex]current bounding box.center)}]
      \draw[->] (1,-1.2) .. controls (1,.5) and (0.5,-1) .. (0.3,-1) .. controls (0,-1) and (0,1) .. (0.3,1) .. controls (0.5,1) and (1,-.5) .. (1,1.2);
      \redcircle{(0.15,0.8)} node[color=black,anchor=east] {$d$};
      \bluedot{(0.16,-0.8)} node[color=black,anchor=east] {$b$};
      \bluedot{(0.98,0.6)} node[color=black,anchor=west] {$f^\vee$};
      \bluedot{(0.66,-0.53)} node[color=black,anchor=north] {$f$};
    \end{tikzpicture}
    \\
    \stackrel{\substack{\eqref{eq:circle-mult-rightup-slide} \\ \eqref{eq:dotsdown-b}, \eqref{eq:rel0b}}}{=}\
    \begin{tikzpicture}[>=stealth,baseline={([yshift=-.5ex]current bounding box.center)}]
      \draw[->] (0,0) .. controls (0,0.5) and (1.5,1) .. (2,1) .. controls (2.5,1) and (2.5,-1) .. (2,-1) .. controls (1.5,-1) and (0,-0.5) .. (0,0);
      \draw[->] (1,-1.2) .. controls (1.5,-0.5) and (1.5,0.5) .. (1,1.2);
      \redcircle{(1.6,0.95)} node[color=black,anchor=south] {$d$};
      \bluedot{(1.5,-0.93)} node[color=black,anchor=north] {$b$};
    \end{tikzpicture}
    \ +\ \sum_{k=0}^{d-1} \sum_{f \in \cB}\
    \begin{tikzpicture}[>=stealth,baseline={([yshift=-.5ex]current bounding box.center)}]
      \draw[->] (0,0) .. controls (0.5,3) and (1,2) .. (1,1) .. controls (1,0) and (-0.5,1) .. (-0.5,2.5);
      \bluedot{(0.4,0.7)} node[color=black,anchor=north] {$b$};
      \redcircle{(-0.05,1.1)} node[color=black,anchor=north east] {$k$};
      \bluedot{(-0.24,1.4)} node[color=black,anchor=east] {$f$};
      \bluedot{(0.39,1.6)} node[anchor=south, color=black] {$f^\vee$};
      \redcircle{(0.9,1.8)} node[color=black,anchor=west] {$d-k-1$};
    \end{tikzpicture}
    \ +\ \sum_{f \in \cB} (-1)^{d \sigma (\bar f + \sigma) + \bar b(\bar f + \sigma) + \sigma \bar f}\
    \begin{tikzpicture}[>=stealth,baseline={([yshift=-.5ex]current bounding box.center)}]
      \draw[->] (0,0) to (0,2);
      \bluedot{(0,1.7)} node[color=black,anchor=west] {$f^\vee$};
      \redcircle{(0,1.2)} node[color=black,anchor=west] {$d$};
      \bluedot{(0,0.7)} node[color=black,anchor=west] {$b$};
      \bluedot{(0,0.2)} node[color=black,anchor=west] {$\psiB(f)$};
    \end{tikzpicture}
    \\
    \stackrel{\substack{\eqref{eq:collision-up}, \eqref{eq:rel0a} \\ \eqref{eq:rel0b},\eqref{eq:curl-Nakayama}}}{=}
    \begin{tikzpicture}[>=stealth,baseline={([yshift=-.5ex]current bounding box.center)}]
      \draw[->] (0,0) .. controls (0,0.5) and (1.5,1) .. (2,1) .. controls (2.5,1) and (2.5,-1) .. (2,-1) .. controls (1.5,-1) and (0,-0.5) .. (0,0);
      \draw[->] (1,-1.2) .. controls (1.5,-0.5) and (1.5,0.5) .. (1,1.2);
      \redcircle{(1.6,0.95)} node[color=black,anchor=south] {$d$};
      \bluedot{(1.5,-0.93)} node[color=black,anchor=north] {$b$};
    \end{tikzpicture}
    \ +\ \sum_{k=0}^{d-1} \sum_{f \in \cB} (-1)^{\bar f + \sigma (dk-k+1 + (d+1)\bar f + d \bar b)}\
    \begin{tikzpicture}[>=stealth,baseline={([yshift=-.5ex]current bounding box.center)}]
      \draw[->] (0,0) .. controls (0.2,3) and (1,1.5) .. (1,1) .. controls (1,0) and (-0.5,1) .. (-0.5,2.5);
      \bluedot{(0.03,0.3)} node[color=black,anchor=east] {$f^\vee$};
      \redcircle{(0.5,0.63)} node[color=black,anchor=north west] {$d-k-1$};
      \redcircle{(-0.2,1.3)} node[color=black,anchor=east] {$k$};
      \bluedot{(-0.4,1.8)} node[color=black,anchor=east] {$\psiB^k(b) f$};
    \end{tikzpicture}
    \ +\ \sum_{f \in \cB} (-1)^{\sigma \bar f}\
    \begin{tikzpicture}[>=stealth,baseline={([yshift=-.5ex]current bounding box.center)}]
      \draw[->] (0,0) to (0,2);
      \bluedot{(0,1)} node[color=black,anchor=west] {$b \psiB^{-d}(f^\vee)$};
      \redcircle{(0,1.6)} node[color=black,anchor=west] {$d$};
      \bluedot{(0,0.4)} node[color=black,anchor=west] {$\psiB(f)$};
    \end{tikzpicture}
    \\
    \stackrel{\eqref{eq:rel1b},\eqref{eq:circle-mult-leftup-slide}}{=}
    \begin{tikzpicture}[>=stealth,baseline={([yshift=-.5ex]current bounding box.center)}]
      \draw[<-] (0,0) arc (90:470:.5);
      \draw[->] (-1.1,-1.5) to (-1.1,0.5);
      \bluedot{(-0.45,-0.7)} node[color=black,anchor=east] {$b$};
      \redcircle{(0.45,-0.3)} node[color=black,anchor=west] {$d$};
    \end{tikzpicture}
    \ +\ \sum_{k=0}^{d-1} \sum_{f \in \cB} (-1)^{\bar f + \sigma (dk-k+1 + (d+1)\bar f + d \bar b)}\
    \left(
      \begin{tikzpicture}[>=stealth,baseline={([yshift=-.5ex]current bounding box.center)}]
        \draw[->] (0,0) .. controls (0,1.5) and (1,1.5) .. (1,1) .. controls (1,0.5) and (0,0) .. (0,2);
        \bluedot{(0.01,0.3)} node[color=black,anchor=east] {$f^\vee$};
        \bluedot{(0.02,1.7)} node[color=black,anchor=east] {$\psiB^k(b) f$};
        \redcircle{(0.05,1.2)} node[color=black,anchor=east] {$d-1$};
      \end{tikzpicture}
      \ -\ \sum_{j = 0}^{d-k-2} \sum_{g \in \cB}\
      \begin{tikzpicture}[>=stealth,baseline={([yshift=-.5ex]current bounding box.center)}]
        \draw[->] (0,0) to (0,3);
        \draw[<-] (1,1.5) arc (90:470:0.5);
        \bluedot{(0,2.7)} node[color=black,anchor=east] {$\psiB^k(b) f$};
        \redcircle{(0,2.2)} node[color=black,anchor=east] {$j+k$};
        \bluedot{(0,1.7)} node[color=black,anchor=east] {$g^\vee$};
        \bluedot{(0,0.2)} node[color=black,anchor=east] {$f^\vee$};
        \bluedot{(0.5,1)} node[color=black,anchor=west] {$g$};
        \redcircle{(1.2,0.55)};
        \draw (1.6,0.5) node[color=black,anchor=north] {$d-k-j-2$};
      \end{tikzpicture}
    \right)
    \\
    \qquad \qquad \qquad +\ \sum_{f \in \cB} (-1)^{\sigma \bar f}\
    \begin{tikzpicture}[>=stealth,baseline={([yshift=-.5ex]current bounding box.center)}]
      \draw[->] (0,0) to (0,2);
      \bluedot{(0,1)} node[color=black,anchor=west] {$b \psiB^{-d}(f^\vee)$};
      \redcircle{(0,1.6)} node[color=black,anchor=west] {$d$};
      \bluedot{(0,0.4)} node[color=black,anchor=west] {$\psiB(f)$};
    \end{tikzpicture}
    \\
    \ \stackrel{\eqref{eq:b-bvee-swap}}{=}\
    \begin{tikzpicture}[>=stealth,baseline={([yshift=-.5ex]current bounding box.center)}]
      \draw[<-] (0,0) arc (90:470:.5);
      \draw[->] (-1.1,-1.5) to (-1.1,0.5);
      \bluedot{(-0.45,-0.7)} node[color=black,anchor=east] {$b$};
      \redcircle{(0.45,-0.3)} node[color=black,anchor=west] {$d$};
    \end{tikzpicture}
    \ +\ \sum_{k=0}^{d-1} \sum_{f \in \cB} (-1)^{\bar f + \sigma (dk-k+1 + (d+1)\bar f + d \bar b)}\
    \left(
      \begin{tikzpicture}[>=stealth,baseline={([yshift=-.5ex]current bounding box.center)}]
        \draw[->] (0,0) to (0,2);
        \bluedot{(0,0.4)} node[color=black,anchor=east] {$f^\vee$};
        \bluedot{(0,1.6)} node[color=black,anchor=east] {$\psiB^k(b)f$};
        \redcircle{(0,1)} node[color=black,anchor=east] {$d$};
      \end{tikzpicture}
      \ -\ \sum_{j = 0}^{d-k-2} \sum_{g \in \cB}\
      \begin{tikzpicture}[>=stealth,baseline={([yshift=-.5ex]current bounding box.center)}]
        \draw[->] (0,0) to (0,3);
        \draw[<-] (1,1.5) arc (90:470:0.5);
        \bluedot{(0,2.7)} node[color=black,anchor=east] {$\psiB^k(b) f$};
        \redcircle{(0,2.2)} node[color=black,anchor=east] {$j+k$};
        \bluedot{(0,1.7)} node[color=black,anchor=east] {$g^\vee$};
        \bluedot{(0,0.2)} node[color=black,anchor=east] {$f^\vee$};
        \bluedot{(0.5,1)} node[color=black,anchor=west] {$g$};
        \redcircle{(1.2,0.55)};
        \draw (1.6,0.5) node[color=black,anchor=north] {$d-k-j-2$};
      \end{tikzpicture}
    \right)
    \\
    \qquad \qquad \qquad +\ \sum_{f \in \cB} (-1)^{\bar f + \sigma}\
    \begin{tikzpicture}[>=stealth,baseline={([yshift=-.5ex]current bounding box.center)}]
      \draw[->] (0,0) to (0,2);
      \bluedot{(0,1)} node[color=black,anchor=west] {$b \psiB^{-d}(f)$};
      \redcircle{(0,1.6)} node[color=black,anchor=west] {$d$};
      \bluedot{(0,0.4)} node[color=black,anchor=west] {$f^\vee$};
    \end{tikzpicture}
    \\
    \ = (-1)^{d \sigma \bar b}\
    \begin{tikzpicture}[>=stealth,baseline={([yshift=-.5ex]current bounding box.center)}]
      \draw[<-] (0,0) arc (90:470:.5);
      \draw[->] (-1.1,-1.5) to (-1.1,0.5);
      \bluedot{(-0.45,-0.3)} node[color=black,anchor=east] {$b$};
      \redcircle{(0.45,-0.7)} node[color=black,anchor=west] {$d$};
    \end{tikzpicture}
    \ +\ \sum_{k=0}^{d-1} \sum_{f \in \cB} (-1)^{\bar f + \sigma(dk-k+1 + \bar f)}\
    \begin{tikzpicture}[>=stealth,baseline={([yshift=-.5ex]current bounding box.center)}]
      \draw[->] (0,0) to (0,2);
      \bluedot{(0,0.4)} node[color=black,anchor=west] {$f^\vee$};
      \bluedot{(0,1)} node[color=black,anchor=west] {$\psiB^{k-d}(b) \psiB^{-d}(f)$};
      \redcircle{(0,1.6)} node[color=black,anchor=west] {$d$};
    \end{tikzpicture}
    \ +\ \sum_{f \in \cB} (-1)^{\bar f + \sigma}\
    \begin{tikzpicture}[>=stealth,baseline={([yshift=-.5ex]current bounding box.center)}]
      \draw[->] (0,0) to (0,2);
      \bluedot{(0,1)} node[color=black,anchor=west] {$b \psiB^{-d}(f)$};
      \redcircle{(0,1.6)} node[color=black,anchor=west] {$d$};
      \bluedot{(0,0.4)} node[color=black,anchor=west] {$f^\vee$};
    \end{tikzpicture}
    \\
    \qquad \qquad \qquad
    \ -\ \sum_{k=0}^{d-1} \sum_{j = 0}^{d-k-2} \sum_{f,g \in \cB}  (-1)^{\bar f + \sigma (dk-k+1 + (d+1)\bar f + d \bar b)}\
    \begin{tikzpicture}[>=stealth,baseline={([yshift=-.5ex]current bounding box.center)}]
      \draw[->] (0,0) to (0,3);
      \draw[<-] (1,1.5) arc (90:470:0.5);
      \bluedot{(0,2.7)} node[color=black,anchor=east] {$\psiB^k(b) f$};
      \redcircle{(0,2.2)} node[color=black,anchor=east] {$j+k$};
      \bluedot{(0,1.7)} node[color=black,anchor=east] {$g^\vee$};
      \bluedot{(0,0.2)} node[color=black,anchor=east] {$f^\vee$};
      \bluedot{(0.5,1)} node[color=black,anchor=west] {$g$};
      \redcircle{(1.2,0.55)};
      \draw (1.6,0.5) node[color=black,anchor=north] {$d-k-j-2$};
    \end{tikzpicture} \qedhere
  \end{multline*}
\end{proof}

Under the functor $\bF$, the image of
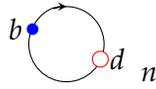
\[
  \begin{tikzpicture}[>=stealth,baseline={([yshift=-.5ex]current bounding box.center)}]
    \draw[<-] (0,0) arc (90:470:.5);
    \draw (1.1,-0.9) node {$n$};
    \bluedot{(-0.45,-0.3)} node[color=black,anchor=east] {$b$};
    \redcircle{(0.45,-0.7)} node[color=black,anchor=west] {$d$};
\end{tikzpicture}
\]
for $n \in \N$, is the degree $(|b|+(d+1)\delta,\bar b + (d+1)\sigma)$ bimodule homomorphism $(n) \to (n)$ given by
\begin{equation} \label{eq:bubble-endo}
  \pr{\left( \sum_{\substack{f \in \cB \\ i \in \{1,\dotsc,n\}}} (-1)^{(d+1) \sigma \bar f^\vee + \bar b \bar f^\vee} s_i \dotsm s_{n-1} \left( 1_B^{\otimes (n-1)} \otimes f^\vee b \right) J_{n-1}^d \left( 1_B^{\otimes (n-1)} \otimes f \right) s_{n-1} \dotsm s_i \right)}.
\end{equation}
\details{
  The image of the given bubble under $\bF$ is the composition
  \[
    (n) \xrightarrow{\eta_\rL} (n)_{n-1}(n) \xrightarrow{\id \otimes \pl{J_{n-1}^d}} (n)_{n-1}(n) \xrightarrow{\phi_{n-1}(b) \otimes \id} (n)_{n-1}(n) \xrightarrow{\varepsilon_\rR} (n).
  \]
  Under this composition, we have
  \begin{multline*}
     1_n \mapsto \sum_{\substack{f \in \cB \\ i \in \{1,\dotsc,n\}}} (-1)^{\sigma \bar f^\vee} s_i \dotsm s_{n-1} \left( 1_B^{\otimes (n-1)} \otimes f^\vee \right) \otimes \left( 1_B^{\otimes (n-1)} \otimes f \right) s_{n-1} \dotsm s_i \\
    \mapsto \sum_{\substack{f \in \cB \\ i \in \{1,\dotsc,n\}}} (-1)^{(d+1) \sigma \bar f^\vee} s_i \dotsm s_{n-1} \left( 1_B^{\otimes (n-1)} \otimes f^\vee \right) \otimes J_{n-1}^d \left( 1_B^{\otimes (n-1)} \otimes f \right) s_{n-1} \dotsm s_i \\
    \mapsto \sum_{\substack{f \in \cB \\ i \in \{1,\dotsc,n\}}} (-1)^{(d+1) \sigma \bar f^\vee + \bar b \bar f^\vee} s_i \dotsm s_{n-1} \left( 1_B^{\otimes (n-1)} \otimes f^\vee b \right) \otimes J_{n-1}^d \left( 1_B^{\otimes (n-1)} \otimes f \right) s_{n-1} \dotsm s_i.
  \end{multline*}
}

\subsection{Endomorphisms of the identity}

We now turn our attention to endomorphisms of the identity.  Define
\begin{equation}
  B_\psi = B/\langle b - \psiB(b) \mid b \in B\rangle.
\end{equation}
For $\beta \in B_\psi$ and $d \in \N$, we define $c_{\beta,d} = c_{b,d}$ where $b$ is any representative of $\beta$.  This is well defined by \eqref{eq:circle-Nakayama}.

Let
\[
  \widetilde \N :=
  \begin{cases}
    \N & \text{if } \sigma=0, \\
    2\N + 1 & \text{if } \sigma=1,
  \end{cases}\
\]
and let $C$ be the $\F$-vector space with basis $v_d$, $d \in \widetilde \N$.  We consider this as a $\N \times \Z_2$-graded vector space by declaring $\deg v_d = (d+1)(\delta,\sigma)$.  So
\begin{equation}
  C = \bigoplus_{d \in \widetilde \N} \F v_d,\quad \deg v_d = (d+1)(\delta,\sigma).
\end{equation}
Note that, for any $\Z \times \Z_2$-graded vector space $V$,
\[
  S(V) := \bigoplus_{k=0}^\infty S^k(V)
\]
is a $\Z \times \Z_2$-graded algebra, with product induced from the product in the tensor algebra.

\begin{prop} \label{prop:identity-end-surjection}
  We have a surjective algebra homomorphism
  \begin{equation} \label{eq:identity-end-surjection}
    S(B_\psi \otimes_\F C) \twoheadrightarrow \END_{\cH'_B} \one,\quad \beta \otimes v_d \mapsto c_{\beta,d},\quad \beta \in B_\psi,\ d \in \widetilde \N.
  \end{equation}
\end{prop}

\begin{proof}
  Endomorphisms of the object $\one$ of $\cH'_B$ correspond to linear combinations of closed diagrams.  Using the local relations, closed diagrams can be converted to linear combinations of diagrams consisting of nested circles decorated with elements of $B$ and open dots (right curls).  Proposition~\ref{prop:bubble-slide} can be used to split apart nested circles.  Then Proposition~\ref{prop:circle-conversion} can be used to convert counterclockwise circles into linear combinations of products of clockwise circles.  Then the result follows from Lemma~\ref{lem:circle-index-relations}.
\end{proof}

\begin{conj} \label{conj:id-end}
  The map \eqref{eq:identity-end-surjection} is an isomorphism.
\end{conj}

\begin{rem}
  If $B = \F$, Conjecture~\ref{conj:id-end} is proved in \cite[Prop.~3]{Kho14}.  The proof there relies on the fact that, when $B = \F$, the composition
  \begin{equation} \label{eq:circle-action}
    S(B_\psi \otimes_\F C) \twoheadrightarrow \END_{\cH'_B}(\one) \xrightarrow{\bigoplus \bF_n} \bigoplus_{n=0}^\infty Z(A_n)
  \end{equation}
  is injective.  However, this composition is \emph{not} injective for general $B$. For example, the composition \eqref{eq:circle-action} maps $c_{b,0}$ to a bimodule endomorphism involving $\sum_{f \in \cB} f^\vee b f$, which has $\Z$-degree $|b| + \delta$.  This is zero if $|b| > 0$, since the maximal degree of elements of $B$ is $\delta$.  For the particular choice of $B$ made in \cite{CL12}, Conjecture~\ref{conj:id-end} reduces to \cite[Conj.~3]{CL12}.
\end{rem}

\subsection{Endomorphisms of $\sP^n$}

For $m \in \N_+$, let $\F[x_1,\dotsc,x_m]$ denote the associative $\F$-algebra generated by $x_1,\dotsc,x_m$ with relations $x_i x_j = (-1)^\sigma x_j x_i$ for $i \ne j$, where we declare each $x_i$, $1 \le i \le m$, to be of degree $(\delta,\sigma)$.

\begin{defin}[The algebra $D_m$]
  For $m \ge 1$, define the $\N$-graded super vector space
  \[
    D_m := \F[x_1,\dotsc,x_m] \otimes_\F A_m^\op.
  \]
  Then we give $D_m$ the structure of an associative $\F$-algebra by declaring $\F[x_1,\dotsc,x_m]$ and $A_m$ to be subalgebras and
  \begin{gather}
    x_i (b_1 \otimes \dotsb \otimes b_m) = (-1)^{\sigma (\bar b_1 + \dotsb + \bar b_m)} (b_1 \otimes \dotsb \otimes b_{i-1} \otimes \psiB(b_i) \otimes b_{i+1} \otimes \dotsb \otimes b_m) x_i, \\
    x_i s_i = s_i x_{i+1} + \sum_{b \in \cB} 1_B^{\otimes(i-1)} \otimes b^\vee \otimes b \otimes 1_B^{\otimes (m-i-1)}, \\
    s_i x_i = x_{i+1} s_i + \sum_{b \in \cB} (-1)^{\bar b (\bar b + \sigma)} 1_B^{\otimes (i-1)} \otimes b \otimes b^\vee \otimes 1_B^{\otimes (m-i-1)},
  \end{gather}
  for all $b_1,\dotsc,b_m \in B$ and $1 \le i \le m-1$.  It follows that $D_m$ is a graded superalgebra.
\end{defin}

Note that, if $B = \F$, then $D_m$ is the degenerate affine Hecke algebra of rank $m$.  By \eqref{eq:rel0a}--\eqref{eq:rel1b}, \eqref{eq:curl-Nakayama}, \eqref{eq:circle-leftup-slide}, and \eqref{eq:circle-rightup-slide}, we have a homomorphism of algebras
\begin{equation} \label{eq:chi-m}
  \chi_m \colon D_m \otimes \END_{\cH'_B} \one \to \END_{\cH'_B} \sP^m,
\end{equation}
sending $s_i$, $1 \le i \le n-1$, to a crossing of strands $i$ and $i+1$, sending $x_i$, $1 \le i \le m$, to a right curl on the $i$-th strand, and sending $1_B^{\otimes (i-1)} \otimes b \otimes 1_B^{\otimes (m-i)}$, $b \in B$, $1 \le i \le m$, to a dot labelled $b$ on the $i$-th strand.  Here we label the strands $1,2,\dotsc,n$ from the left and read planar diagrams from bottom to top (i.e.\ an element $z_1 z_2$, with $z_1,z_2 \in D_m \otimes \END_{\cH'_B} \one$ is mapped to the diagram corresponding to $z_1$ placed on top of the diagram corresponding to $z_2$).  Elements of $\END_{\cH'_B} \one$, which consist of closed diagrams, are placed to the right of the diagrams corresponding to elements of $D_m$.

Define the \emph{$m$-disturbance} of a permutation $\tau \in S_{n+m}$ to be the number of integers between 1 and $n$ that are not fixed by $\tau$.  Precisely, the $m$-disturbance of $\tau$ is
\[
  \dist_m \tau = | \{i \mid 1 \le i \le n,\ \tau(i) \ne i \} |.
\]
We define the $m$-disturbance of elements of $B^{\otimes (n+m)}$ to be zero.  This yields a filtration on $A_{n+m}^\op$ by taking the $k$-th step in the increasing filtration to be the subspace spanned by elements of total $m$-disturbance at most $k$.  Analogously, we define a filtration on $D_m$ by setting the $m$-disturbance of elements of $A_m^\op$ to be zero, and the $m$-disturbance of $x_i$ to be 1 for $1 \le i \le m$.

For $i,j \in \{1,\dotsc,n\}$ with $i < j$, define the degree $(\delta,\sigma)$ element of $A_n$
\begin{equation}
  t_{i,j} =
  \sum_{b \in \cB} (-1)^{\sigma \bar b + \sigma} \left( 1_B^{\otimes (i-1)} \otimes b^\vee \otimes 1_B^{\otimes (j-i-1)} \otimes b \otimes 1_B^{\otimes (n-j)} \right) (i,j).
\end{equation}
Here, and in what follows, we use cycle notation for permutations.  Then, for $n \ge 2$, we have
\begin{equation}
  J_{n-1} = \sum_{i=1}^{n-1} t_{i,n}.
\end{equation}

\begin{lem} \label{lem:tij-products}
  Suppose $i_1,\dotsc,i_\ell \in \{1,\dotsc,n\}$ and $n < j \le n + m$.  Then, in the associated graded algebra corresponding to the above filtration on $D_m$, we have $t_{i_\ell,j} \dotsm t_{i_2,j} t_{i_1,j}=0$ if the $i_1,\dotsc,i_\ell$ are not all distinct and, if $i_1 < i_2 < \dotsb < i_\ell$, then
  \[
    t_{i_\ell,j} \dotsm t_{i_2,j} t_{i_1,j} = (-1)^{\sigma \ell (\ell+1)/2} \sum_{b_1,\dotsc,b_\ell \in \cB} (-1)^{\sigma \bar b_\ell + \sum_{s=1}^{\ell-1} \bar b_s \bar b_{s+1}} \beta_{b_1,\dotsc,b_\ell} (i_1,i_2,\dotsc,i_\ell,j),
  \]
  where $\beta_{b_1,\dotsc,b_\ell} \in B^{\otimes n}$ is a simple tensor with $k$-th component equal to
  \begin{itemize}
    \item $b_1^\vee$ if $k=i_1$,
    \item $b_p^\vee b_{p-1}$ if $k = i_p$ for some $2 \le p \le \ell$,
    \item $b_\ell$ if $k=j$,
    \item $1_B$ if $k \not \in \{i_1,\dotsc,i_\ell,j\}$.
  \end{itemize}
\end{lem}

\begin{proof}
  The transposition $(i_k,j)$, $1 \le k \le \ell$, has $m$-disturbance one.  If the $i_k$ are not all distinct, then the product $(i_\ell,j) \dotsm (i_2,j) (i_1,j)$ has disturbance less than $\ell$ and so is zero in the associated graded algebra.  The case where $i_1 < i_2 < \dotsb < i_\ell$ follows by direct computation.
\end{proof}

\begin{prop} \label{prop:chim-properties}
  The map $\chi_m$ of \eqref{eq:chi-m} is a surjective algebra homomorphism.  Its restriction to $D_m$ is injective.
\end{prop}

\begin{proof}
  Any diagram representing an element of $\END_{\cH'_B} P^m$ can be inductively simplified to a linear combination of standard diagrams consisting of (the image under $\chi_m$) of an element of $D_m$ and a closed diagram to the right:
  
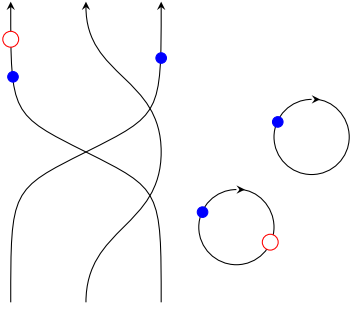
\[
    \begin{tikzpicture}[>=stealth,baseline={([yshift=-.5ex]current bounding box.center)}]
      \draw[->] (0,0) .. controls (0,1.5) .. (1,2) .. controls (2,2.5) .. (2,4);
      \draw[->] (2,0) .. controls (2,1.5) .. (1,2) .. controls (0,2.5) .. (0,4);
      \draw[->] (1,0) .. controls (1,1) and (2,1) .. (2,2) .. controls (2,3) and (1,3) .. (1,4);
      \bluedot{(0.03,3)};
      \redcircle{(0,3.5)};
      \bluedot{(2,3.25)};
      \draw[-<] (4,2.7) arc(90:450:.5);
      \bluedot{(3.55,2.4)};
      \draw[-<] (3,1.5) arc(90:450:.5);
      \bluedot{(2.55,1.2)};
      \redcircle{(3.45,0.8)};
    \end{tikzpicture}
  \]
  Surjectivity of $\chi_m$ follows.

  It remains to show that the restriction of $\chi_m$ to $D_m$ is injective.  For $n \in \N$, we have the composition
  \begin{gather*}
    \chi_{m,n} \colon D_m \xrightarrow{\chi_m} \END_{\cH'_B} (\sP^m) \xrightarrow{\bF_n} \END((n+m)_n), \\
    \chi_{m,n}(x_i) = \pr{J_{n+m-i}},\quad \chi_{m,n}(a) = \pr{\shift_n(a)},\quad a \in A_m,\ i \in \{1,\dotsc,m\},
  \end{gather*}
  \details{
    Note that the subscript of $n+m-i$ in the equality $\chi_{m,n}(x_i) = \pr{J_{n+m-i}}$ comes from the fact that the region to the right of the $i$-th upward pointing strand from the left is $n+m-i$.
  }
  where $\shift_n{a}$ denotes the shift of an element $a \in A_m$ given by
  \begin{gather*}
    \shift_n(\beta) = 1_B^{\otimes n} \otimes \beta,\quad \beta \in B^{\otimes n}, \\
    \shift_n(\tau)(i+n) = \tau(i) + n,\quad 1 \le i \le m,\quad \shift_n(\tau)(i) = i,\quad 1 \le i \le n.
  \end{gather*}
  Now, $\END ((n+m)_n) \subseteq A_{n+m}^\op$, since the endomorphism algebra of an algebra $R$ viewed as a left $R$-module is naturally isomorphic to $R^\op$, i.e., $\END_R (\prescript{}{R}R, \prescript{}{R}R) \cong R^\op$.  Thus, $\chi_{m,n}$ corresponds to a map
  \[
    \chi_{m,n}' \colon D_m \to A_{n+m}^\op.
  \]

  The map $\chi_{m,n}'$ is a homomorphism of filtered algebras.  To show that the restriction of $\chi_m$ to $D_m$ is injective, it suffices to show that $\chi_{m,n}'$ is injective for some $n$. We will take $n$ large compared to $m$.  Note that $D_m$ has a basis given by the elements
  \[
    x(\ba,\beta,\tau) := x_1^{a_1} \dotsm x_m^{a_m} \beta \tau,\quad \ba =(a_1,\dotsc,a_m) \in \N^m,\ \beta \in \cB^{\otimes m},\ \tau \in S_m.
  \]
  Consider a nonzero element
  \begin{equation} \label{eq:chim-kernel-element}
    y = \sum_{\ba,\beta,\tau} d_{\ba,\beta,\tau}\, x(\ba,\beta,\tau) \in D_m,
  \end{equation}
  where $d_{\ba,\beta,\tau} \in \F \setminus \{0\}$ and the sum is over finitely many triples $(\ba,\beta,\tau)$.  The element $\chi_{m,n}'(x(\ba,\beta,\tau))$ belongs to the $k$-th step of the filtration of $A_{m+n}$, where
  \[
    k = a_1 + \dotsb + a_m,
  \]
  but not to the $(k-1)$-st step.  Let $k_0$ denote the largest value of $k$ such that $\chi_{m,n}'(y)$ does not lie in the $(k-1)$-st step in the filtration of $A_{m+n}^\op$.  It suffices to show that, as we sum over only the terms in \eqref{eq:chim-kernel-element} lying in the $k_0$-th step of the filtration, the image of $\chi_{m,n}'( \sum d_{\ba,\beta,\tau} x(\ba,\beta,\tau))$ in the associated graded ring of $A_{m+n}^\op$ relative to the $m$-disturbance filtration is nonzero.  In other words, it is enough to show that, for some sufficiently large $n$, the $B^{\otimes n}$-coefficients of permutations of disturbance $k_0$ in
  \begin{equation} \label{eq:chi-image}
    \chi_{m,n}' \left( \sum_{\ba,\beta,\tau} d_{\ba,\beta,\tau}\, x(\ba,\beta,\tau) \right) \in A_{m+n}
  \end{equation}
  are not all zero.

  Choose $\ba,\beta = \beta_1 \otimes \dotsb \otimes \beta_m,\tau$ such that $a_1 + \dotsb + a_m = k_0$ and $d_{\ba, \beta, \tau} \ne 0$.  Replacing $y$ by $y \tau^{-1}$, we may assume that $\tau = 1$.  Then we have
  \begin{equation} \label{eq:Dm-injective-leading-term}
    \chi_{m,n}'(x(\ba,\beta,1)) = \pm \shift_n(\beta) J_{n}^{a_m} J_{n+1}^{a_{m-1}} \dotsb J_{n+m-1}^{a_1}.
  \end{equation}
  Choose natural numbers
  \[
    i_{1,m} < i_{2,m} < \dotsb i_{a_m,m} < i_{1,m-1} < \dotsb < i_{a_{m-1},m-1} < \dotsb < i_{1,1} < \dotsb < i_{a_1,1} < n.
  \]
  (Recall that we take $n$ to be large.)  It follows from Lemma~\ref{lem:tij-products} that the $B^{\otimes m}$-coefficient of
  \begin{equation} \label{eq:highest-disturbance-term}
    (i_{1,m},i_{2,m},\dotsc,i_{a_m,m},n+1) \dotsm (i_{1,2}, i_{2,2}, \dotsc, i_{a_2,2}, n+m-1) (i_{1,1}, i_{2,1}, \dotsc, i_{a_1,1},n+m)
  \end{equation}
  in \eqref{eq:chi-image} is equal to the coefficent of \eqref{eq:highest-disturbance-term} in \eqref{eq:Dm-injective-leading-term} and is equal to
  \begin{equation} \label{eq:leading-term-coef}
    \pm \sum_{\substack{b_{s,r} \in \cB \\ 1 \le r \le m,\ 1 \le s \le a_r}} (-1)^{\sum_{t=1}^m \left( \sigma \bar b_{a_t,t} + \sum_{\ell=1}^{a_t-1} \bar b_{\ell,t} \bar b_{\ell+1,t} \right)} \shift_n(\beta) \beta',
  \end{equation}
  where $\beta' \in B^{\otimes (m+n)}$ is the simple tensor with $k$-th component equal to
  \begin{itemize}
    \item $b_{1,r}^\vee$ if $k=i_{1,r}$ for some $r = 1,\dotsc,m$,
    \item $b_{s,r}^\vee b_{s-1,r}$ if $k=i_{s,r}$ for some $s = 2, \dotsc, a_r$, $r = 1,\dotsc,m$,
    \item $b_{a_r,r}$ if $k=n+m+1-r$ for some $r=1,\dotsc,m$ (where $b_{a_r,r}=1_B$ if $a_r=0$),
    \item $1_B$ otherwise.
  \end{itemize}
  Note that, in \eqref{eq:leading-term-coef} and in what follows, we have absorbed all signs not depending on the $b_{s,r}$ into the leading $\pm$.  It suffices to prove that \eqref{eq:leading-term-coef} is nonzero.  Note that, for $b \in B$, we have
  \begin{equation} \label{eq:contraction}
    (\tr_B \otimes \id) \left ( \sum_{c \in \cB} (-1)^{\bar b \bar c} c^\vee \otimes bc \right)
    = \sum_{c \in \cB} \tr_B(c^\vee) \otimes b c \\
    = 1_\F \otimes b \left( \sum_{c \in \cB} \tr_B(c^\vee) c \right)
    \stackrel{\eqref{eq:dual-bases-decomp}}{=} 1_\F \otimes b,
  \end{equation}
  where, in the first equality, we used the fact that $\tr_B(c^\vee)=0$ unless $\bar c = 0$ by \eqref{eq:trace-parity-condition}.  Thus, applying the map $\tr_B$ to the $i_{s,r}$-th components of \eqref{eq:leading-term-coef} for $s=1,\dotsc,a_r-1$ and $r=1,\dotsc,m$, we obtain
  \begin{equation} \label{eq:leading-term-coef2}
    \pm \sum_{b_{a_1,1},\dotsc,b_{a_m,m} \in \cB} (-1)^{\sigma \sum_{r=1}^m \bar b_r} \shift_n(\beta) \beta'',
  \end{equation}
  where $\beta'' \in B^{\otimes (m+n)}$ is the simple tensor with $k$-th component equal to
  \begin{itemize}
    \item $b_{a_r,r}^\vee$ if $k=i_{a_r,r}$ for some $r=1,\dotsc,m$,
    \item $b_{a_r,r}$ if $k=n+m+1-r$ for some $r=1,\dotsc,m$ (where $b_{a_r,r}=1_B$ if $a_r=0$),
    \item $1_B$ or $1_\F$ otherwise.
  \end{itemize}
  Therefore, \eqref{eq:leading-term-coef2} is equal to
  \begin{equation} \label{eq:leading-term-coef3}
    \pm \sum_{b_1,\dotsc,b_m \in \cB} (-1)^{\sum_{r=1}^m \bar b_r (\sigma + \bar \beta_1 + \dotsb + \bar \beta_{m+1-r})} \beta''',
  \end{equation}
  where $\beta''' \in B^{\otimes (m+n)}$ is the simple tensor with $k$-th component equal to
  \begin{itemize}
    \item $b_r^\vee$ if $k=i_{a_r,r}$ for some $r=1,\dotsc,m$,
    \item $\beta_{m+1-r} b_r$ if $k=n+m+1-r$ for some $r=1,\dotsc,m$ (where $b_r=1_B$ if $a_r=0$),
    \item $1_B$ or $1_\F$ otherwise.
  \end{itemize}
  Finally, applying the map $\tr_B$ to the $i_{a_r,r}$-the components of \eqref{eq:leading-term-coef3} for $r=1,\dotsc,m$ gives
  \[
    \pm \left( 1^{\otimes n} \otimes \beta \right),
  \]
  which is nonzero.  (As in \eqref{eq:contraction}, we use the fact that $\tr_B(b_r^\vee)=0$ unless $\bar b_r = 0$.)  It follows that \eqref{eq:leading-term-coef} is nonzero, as desired.
\end{proof}

\begin{cor} \label{cor:CL-conjecture}
  Conjecture~1 of \cite{CL12} holds.
\end{cor}

\begin{proof}
  Recall that the category $\mathscr{H}^\Gamma$ of \cite{CL12} is a special case of the 2-category version $\mathscr{H}_B$ of the category $\cH_B$ defined in the current paper, where $B$ is skew-zigzag algebra defined in \cite[\S2.3]{CLS14} (see Remark~\ref{rem:special-cases}\eqref{rem-item:special-cases-CLaux}).  The subscript $i$ of the algebra $H_i^n$ of \cite[\S10.3]{CL12} is a node of a Dynkin diagram of affine ADE type corresponding to $\Gamma$.  If $e_i$ is the corresponding idempotent of $B$, then $e_iBe_i$ is two-dimensional, spanned by $e_i$ and $p_i$, where $p_i$ is the class of a path of length two (starting and ending at the vertex $i$) in the double quiver associated to $\Gamma$.

  The algebra $H_i^n$ of \cite{CL12} is precisely $(1 \otimes e_i^{\otimes n}) D_n (1 \otimes e_i^{\otimes n})$.   For $k \in \{1,\dotsc,n\}$, the $y_k$ of \cite{CL12} are the $x_k$ of the current paper, and the $z_k$ of \cite{CL12} are the elements $1_B^{\otimes (k-1)} \otimes p_i \otimes 1_B^{\otimes (n-k)} \in A_n^\op \subseteq D_n$.  Finally, as explained in Remark~\ref{rem:special-cases}\eqref{rem-item:special-cases-CLaux}, the 1-morphism $P_i^n$ of $\mathscr{H}^\Gamma$, which is $n$ upward pointing strands labeled $i$, corresponds to the object $(\sP^n,e_i^{\otimes n})$ of $\cH_B$.  Thus, Conjecture~1 of \cite{CL12} follows immediately from Proposition~\ref{prop:chim-properties}.
\end{proof}

\begin{conj} \label{conj:chim-isom}
  The map $\chi_m$ of \eqref{eq:chi-m} is an isomorphism of algebras for all $m \in \N$.
\end{conj}

\begin{rem}
  In the case $B=\F$, Conjecture~\ref{conj:chim-isom} is proven in \cite[Prop.~4]{Kho14}.   The map $\chi_m$ is also conjectured to be an isomorphism in \cite[Conj.~2]{CL12} for the special case of $B$ considered there.  If Conjectures~\ref{conj:id-end} and~\ref{conj:chim-isom} both hold, then $\END_{\cH'_B} \sP^m \cong S(B_\psi \otimes C) \otimes D_m$.
\end{rem}

\section{Key isomorphisms in $\cH_B$} \label{sec:key-isoms}

It follows from relations \eqref{eq:collision-up}, \eqref{eq:rel0a}--\eqref{eq:rel1b} that we have an algebra homomorphism
\begin{equation} \label{eq:map-Bn-Pn}
  A_n^{\op}\to \END_{\cH'_B}(\sP^n).
\end{equation}
Using the adjunctions we then also have an analogous map
\begin{equation} \label{eq:map-Bn-Qn}
  A_n \to \END_{\cH'_B}(\sQ^n).
\end{equation}
For $n \in \N_+$ and an idempotent $f \in B$, let
\begin{equation}
  e_{f,(n)}:=(f\otimes f\otimes \dotsb\otimes f)\frac{1}{n!}\sum_{w\in S_n}w,\quad
  e_{f,(1^n)}:=(f\otimes f\otimes \dotsb\otimes f)\frac{1}{n!}\sum_{w\in S_n} (-1)^{\ell(w)} w.
\end{equation}
By abuse of notation, we also use $e_{f,(n)}$ and $e_{f,(1^n)}$ to denote the images of these elements under the maps \eqref{eq:map-Bn-Pn} and \eqref{eq:map-Bn-Qn}.  Then, for an idempotent $f \in B$ and $\lambda = (n)$ or $\lambda = (1^n)$, $n \in \N_+$, we define the following objects in the category $\cH_B$:
\[
  \sP^\lambda_f:= \left( \sP^{|\lambda|},e_{f,\lambda} \right),\quad \sQ^\lambda_f:= \left( \sQ^{|\lambda|},e_{f,\lambda} \right).
\]
The inclusion morphisms $\sP_f^\lambda \hookrightarrow \sP^{|\lambda|}$ and $\sQ_f^\lambda \hookrightarrow \sQ_f^{|\lambda|}$ will be denoted, respectively, by
\begin{equation}
  \begin{tikzpicture}[>=stealth,baseline={([yshift=-.5ex]current bounding box.center)}]
    \draw(-1,-0.3) rectangle (0,0.2) node[midway]{$f,\lambda$};
    \draw[->] (-0.5,0.2) to (-0.5,0.8);
  \end{tikzpicture}
  \qquad \text{and} \qquad
  \begin{tikzpicture}[>=stealth,baseline={([yshift=-.5ex]current bounding box.center)}]
    \draw(-1,-0.3) rectangle (0,0.2) node[midway]{$f,\lambda$};
    \draw[<-] (-0.5,0.2) to (-0.5,0.8);
  \end{tikzpicture}
  \ ,
\end{equation}
while the projections $\sP^{|\lambda|} \twoheadrightarrow \sP_f^\lambda$ and $\sQ_f^{|\lambda|} \twoheadrightarrow \sQ_f^\lambda$ will be denoted, respectively, by
\begin{equation}
  \begin{tikzpicture}[>=stealth,baseline={([yshift=-.5ex]current bounding box.center)}]
    \draw(-1,-0.3) rectangle (0,0.2) node[midway]{$f,\lambda$};
    \draw[->] (-0.5,-0.9) to (-0.5,-0.3);
  \end{tikzpicture}
  \qquad \text{and} \qquad
  \begin{tikzpicture}[>=stealth,baseline={([yshift=-.5ex]current bounding box.center)}]
    \draw(-1,-0.3) rectangle (0,0.2) node[midway]{$f,\lambda$};
    \draw[<-] (-0.5,-0.9) to (-0.5,-0.3);
  \end{tikzpicture}
  \ .
\end{equation}
Finally, the compositions $\sP_f^\lambda \hookrightarrow \sP^{|\lambda|} \twoheadrightarrow \sP_f^\lambda$ and $\sQ_f^\lambda \hookrightarrow \sQ_f^{|\lambda|} \twoheadrightarrow \sQ_f^\lambda$ will be denoted, respectively, by
\begin{equation}
  \begin{tikzpicture}[>=stealth,baseline={([yshift=-.5ex]current bounding box.center)}]
    \draw(-1,-0.3) rectangle (0,0.2) node[midway]{$f,\lambda$};
    \draw[->] (-0.5,0.2) to (-0.5,0.8);
    \draw[->] (-0.5,-0.9) to (-0.5,-0.3);
  \end{tikzpicture}
  \qquad \text{and} \qquad
  \begin{tikzpicture}[>=stealth,baseline={([yshift=-.5ex]current bounding box.center)}]
    \draw(-1,-0.3) rectangle (0,0.2) node[midway]{$f,\lambda$};
    \draw[<-] (-0.5,0.2) to (-0.5,0.8);
    \draw[<-] (-0.5,-0.9) to (-0.5,-0.3);
  \end{tikzpicture}
  \ .
\end{equation}
As above, for simplicity, when drawing diagrams involving such boxes, we will sometimes draw one curve to represent multiple strands.  When there is potential for confusion, we indicate the number of strands beside the curve.

Note that, for all idempotents $f\in B$ and $n \geq k$, the fact that $\tau \sum_{w \in S_n} w = \left( \sum_{w \in S_n} w \right) \tau$ for all $\tau \in S_n$ implies that
\begin{equation} \label{eq:absorb-crossing}
  \begin{tikzpicture}[>=stealth,baseline={([yshift=-.5ex]current bounding box.center)}]
    \draw (-0.75,0) rectangle (0.75,0.5) node[midway] {$f,(n)$};
    \draw (-0.45,0.5) -- (-0.45,1.5);
    \draw (0.45,0.5) -- (0.45,1.5);
    \draw (-0.15,0.5) -- (0.15,1.5);
    \draw (0.15,0.5) -- (-0.15,1.5);
  \end{tikzpicture}
  \ =\
  \begin{tikzpicture}[>=stealth,baseline={([yshift=-.5ex]current bounding box.center)}]
    \draw (-0.75,0) rectangle (0.75,0.5) node[midway] {$f,(n)$};
    \draw (-0.45,0.5) -- (-0.45,1.5);
    \draw (0.45,0.5) -- (0.45,1.5);
    \draw (0.15,0.5) -- (0.15,1.5);
    \draw (-0.15,0.5) -- (-0.15,1.5);
  \end{tikzpicture}
  \ ,\qquad
    \begin{tikzpicture}[>=stealth,baseline={([yshift=-.5ex]current bounding box.center)}]
    \draw (-0.75,0) rectangle (0.75,-0.5) node[midway] {$f,(n)$};
    \draw (-0.45,-0.5) -- (-0.45,-1.5);
    \draw (0.45,-0.5) -- (0.45,-1.5);
    \draw (-0.15,-0.5) -- (-0.15,-1.5);
    \draw (0.15,-0.5) -- (0.15,-1.5);
  \end{tikzpicture}
  \ =\
  \begin{tikzpicture}[>=stealth,baseline={([yshift=-.5ex]current bounding box.center)}]
    \draw (-0.75,0) rectangle (0.75,-0.5) node[midway] {$f,(n)$};
    \draw (-0.45,-0.5) -- (-0.45,-1.5);
    \draw (0.45,-0.5) -- (0.45,-1.5);
    \draw (-0.15,-0.5) -- (0.15,-1.5);
    \draw (0.15,-0.5) -- (-0.15,-1.5);
  \end{tikzpicture}\ ,
\end{equation}
\begin{equation} \label{eq:idem-diag}
  \begin{tikzpicture}[>=stealth,baseline={([yshift=-.5ex]current bounding box.center)}]
    \draw (0,0) -- (0,2);
    \draw (0.6,0) -- (0.6,2);
    \filldraw[fill=white] (-0.4,1.2) rectangle (0.9,1.7) node[midway] {$f,(n)$};
    \filldraw[fill=white] (-0.5,0.3) rectangle (0.5,0.8) node[midway] {$f,(k)$};
  \end{tikzpicture}
  \ =\
  \begin{tikzpicture}[>=stealth,baseline={([yshift=-.5ex]current bounding box.center)}]
    \draw (0.25,0) -- (0.25,2);
    \filldraw[fill=white] (-0.4,0.75) rectangle (0.9,1.25) node[midway] {$f,(n)$};
  \end{tikzpicture}
  \ =\
  \begin{tikzpicture}[>=stealth,baseline={([yshift=-.5ex]current bounding box.center)}]
    \draw (0,0) -- (0,2);
    \draw (0.6,0) -- (0.6,2);
    \filldraw[fill=white] (-0.4,0.3) rectangle (0.9,0.8) node[midway] {$f,(n)$};
    \filldraw[fill=white] (-0.5,1.2) rectangle (0.5,1.7) node[midway] {$f,(k)$};
  \end{tikzpicture}\ ,
\end{equation}
where the strands are either all oriented up or all oriented down.  Note also that, in light of \eqref{eq:rel0a} and \eqref{eq:rel0b}, crossings of strands carrying dots may also be absorbed into boxes, leaving the dots on uncrossed strands.  For example, we have
\[
  \begin{tikzpicture}[>=stealth,baseline={([yshift=-.5ex]current bounding box.center)}]
    \draw (-0.75,0) rectangle (0.75,0.5) node[midway] {$f,(n)$};
    \draw (-0.45,0.5) -- (-0.45,1.5);
    \draw (0.45,0.5) -- (0.45,1.5);
    \draw (-0.15,0.5) -- (0.15,1.5);
    \draw (0.15,0.5) -- (-0.15,1.5);
    \bluedot{(0.075,0.75)};
  \end{tikzpicture}
  \ =\
  \begin{tikzpicture}[>=stealth,baseline={([yshift=-.5ex]current bounding box.center)}]
    \draw (-0.75,0) rectangle (0.75,0.5) node[midway] {$f,(n)$};
    \draw (-0.45,0.5) -- (-0.45,1.5);
    \draw (0.45,0.5) -- (0.45,1.5);
    \draw (-0.15,0.5) -- (-0.15,1.5);
    \draw (0.15,0.5) -- (0.15,1.5);
    \bluedot{(-0.15,1)};
  \end{tikzpicture}\ .
\]
This will be used repeatedly in the calculations to follow.

For $b_1,\dotsc,b_k \in B$ and $\bb = b_1 \otimes \dotsb \otimes b_k$, we define
\begin{equation}
  \bb^\vee = b_1^\vee \otimes \dotsb \otimes b_k^\vee.
\end{equation}
Note that $\bb^\vee$ is not necessarily right dual to $\bb$ under the trace $\tr_B^{\otimes n}$; it may differ from the right dual by a sign.  We also define
\begin{equation} \label{eq:multi-dots}
  \begin{tikzpicture}[>=stealth,baseline={([yshift=-.5ex]current bounding box.center)}]
    \draw[->] (0,2) -- (0,0);
    \bluedot{(0,1)} node[anchor=east, color=black] {$\bb$};
  \end{tikzpicture}
  \ = \
  \begin{tikzpicture}[>=stealth,baseline={([yshift=-.5ex]current bounding box.center)}]
    \draw[->] (0,2) -- (0,0);
    \bluedot{(0,1.7)} node[anchor=east, color=black] {$b_1$};
    \draw[->] (1,2) -- (1,0);
    \bluedot{(1,1.4)} node[anchor=east, color=black] {$b_2$};
    \draw (1.7,1) node {$\cdots$};
    \draw[->] (2.4,2) -- (2.4,0);
    \bluedot{(2.4,0.3)} node[anchor=west, color=black] {$b_k$};
  \end{tikzpicture}
  \quad \text{and} \quad
  \begin{tikzpicture}[>=stealth,baseline={([yshift=-.5ex]current bounding box.center)}]
    \draw[<-] (0,2) -- (0,0);
    \bluedot{(0,1)} node[anchor=east, color=black] {$\bb$};
  \end{tikzpicture}
  \ = \
  \begin{tikzpicture}[>=stealth,baseline={([yshift=-.5ex]current bounding box.center)}]
    \draw[<-] (0,2) -- (0,0);
    \bluedot{(0,1.7)} node[anchor=east, color=black] {$b_k$};
    \draw (0.7,1) node {$\cdots$};
    \draw[<-] (1.4,2) -- (1.4,0);
    \bluedot{(1.4,0.6)} node[anchor=west, color=black] {$b_2$};
    \draw[<-] (2.4,2) -- (2.4,0);
    \bluedot{(2.4,0.3)} node[anchor=west, color=black] {$b_1$};
  \end{tikzpicture}.
\end{equation}
We then define strands carrying dots labeled by arbitrary elements of $B^{\otimes k}$, $k \in \N_+$, by linearity.

\begin{lem} \label{lem:multi-strand-up-down-rel}
  For all idempotents $f,g \in B$ and $m, n \in \N_+$, we have
  \begin{equation} \label{eq:multi-strand-up-down-rel}
    \begin{tikzpicture}[>=stealth,baseline={([yshift=-.5ex]current bounding box.center)}]
      \draw (-0.4,1) rectangle (-1.6,1.5) node[midway] {$f,(n)$};
      \draw (0.4,1) rectangle (1.6,1.5) node[midway] {$g,(m)$};
      \draw (-0.4,-1) rectangle (-1.6,-1.5) node[midway] {$f,(n)$};
      \draw (0.4,-1) rectangle (1.6,-1.5) node[midway] {$g,(m)$};
      \draw[->] (-1,1) .. controls (1,0) .. (-1,-1);
      \draw[<-] (1,1) .. controls (-1,0) .. (1,-1);
    \end{tikzpicture}
    \ =\
    \sum_{\ell=0}^{\min \{m,n\}} \sum_{\bb \in \cB^{\otimes \ell}} (-1)^\ell \frac{m! n!}{\ell! (m-\ell)! (n-\ell)!}\
    \begin{tikzpicture}[>=stealth,baseline={([yshift=-.5ex]current bounding box.center)}]
      \draw (-0.4,1) rectangle (-1.6,1.5) node[midway] {$f,(n)$};
      \draw (0.4,1) rectangle (1.6,1.5) node[midway] {$g,(m)$};
      \draw (-0.4,-1) rectangle (-1.6,-1.5) node[midway] {$f,(n)$};
      \draw (0.4,-1) rectangle (1.6,-1.5) node[midway] {$g,(m)$};
      \draw[->] (-1.2,1) -- (-1.2,-1);
      \draw[<-] (1.2,1) -- (1.2,-1);
      \draw[->] (-0.8,1) arc (180:360:0.8);
      \bluedot{(0.72,0.6)} node[color=black,anchor=north] {$\bb^\vee$};
      \draw[<-] (-0.8,-1) arc (180:0:0.8);
      \bluedot{(-0.72,-0.6)} node[color=black,anchor=south] {$\bb$};
    \end{tikzpicture}
  \end{equation}
\end{lem}

\begin{proof}
  We prove the result by induction on $n$.  Consider first the case $n=1$.  We prove this case by induction on $m$.  The case $m=1$ is simply \eqref{eq:rel2a}.  Then, for $m > 1$, we have
  \begin{multline*}
    \begin{tikzpicture}[>=stealth,baseline={([yshift=-.5ex]current bounding box.center)}]
      \draw (-0.4,1) rectangle (-1.6,1.5) node[midway] {$f,(1)$};
      \draw (0.4,1) rectangle (1.6,1.5) node[midway] {$g,(m)$};
      \draw (-0.4,-1) rectangle (-1.6,-1.5) node[midway] {$f,(1)$};
      \draw (0.4,-1) rectangle (1.6,-1.5) node[midway] {$g,(m)$};
      \draw[->] (-1,1) .. controls (1,0) .. (-1,-1);
      \draw[<-] (1,1) .. controls (-1,0) .. (1,-1);
    \end{tikzpicture}
    \ \stackrel{\eqref{eq:rel2a}}{=}\
    \begin{tikzpicture}[>=stealth,baseline={([yshift=-.5ex]current bounding box.center)}]
      \draw (-0.4,1) rectangle (-1.6,1.5) node[midway] {$f,(1)$};
      \draw (0.4,1) rectangle (1.6,1.5) node[midway] {$g,(m)$};
      \draw (-0.4,-1) rectangle (-1.6,-1.5) node[midway] {$f,(1)$};
      \draw (0.4,-1) rectangle (1.6,-1.5) node[midway] {$g,(m)$};
      \draw[->] (-1.2,1) .. controls (1,0) .. (-1.2,-1);
      \draw[<-] (0.8,1) .. controls (-1,0) .. (0.8,-1);
      \draw[<-] (1.2,1) -- (1.2,-1) node[pos=0.4,anchor=west] {$1$};
    \end{tikzpicture}
    \ - \sum_{c \in \cB}
    \begin{tikzpicture}[>=stealth,baseline={([yshift=-.5ex]current bounding box.center)}]
      \draw (-0.4,1) rectangle (-1.6,1.5) node[midway] {$f,(1)$};
      \draw (0.4,1) rectangle (1.6,1.5) node[midway] {$g,(m)$};
      \draw (-0.4,-1) rectangle (-1.6,-1.5) node[midway] {$f,(1)$};
      \draw (0.4,-1) rectangle (1.6,-1.5) node[midway] {$g,(m)$};
      \draw[<-] (1.2,1) -- (1.2,-1);
      \draw[->] (-0.8,1) arc (180:360:0.8);
      \bluedot{(0.72,0.6)} node[color=black,anchor=north] {$c^\vee$};
      \draw[<-] (-0.8,-1) arc (180:0:0.8);
      \bluedot{(-0.72,-0.6)} node[color=black,anchor=south] {$c$};
    \end{tikzpicture} \\
    =\
    \begin{tikzpicture}[>=stealth,baseline={([yshift=-.5ex]current bounding box.center)}]
      \draw (-0.4,1) rectangle (-1.6,1.5) node[midway] {$f,(1)$};
      \draw (0.4,1) rectangle (1.6,1.5) node[midway] {$g,(m)$};
      \draw (-0.4,-1) rectangle (-1.6,-1.5) node[midway] {$f,(1)$};
      \draw (0.4,-1) rectangle (1.6,-1.5) node[midway] {$g,(m)$};
      \draw[->] (-1,1) -- (-1,-1);
      \draw[<-] (1,1) -- (1,-1);
    \end{tikzpicture}
    \ - m \sum_{c \in \cB}\
    \begin{tikzpicture}[>=stealth,baseline={([yshift=-.5ex]current bounding box.center)}]
      \draw (-0.4,1) rectangle (-1.6,1.5) node[midway] {$f,(1)$};
      \draw (0.4,1) rectangle (1.6,1.5) node[midway] {$g,(m)$};
      \draw (-0.4,-1) rectangle (-1.6,-1.5) node[midway] {$f,(1)$};
      \draw (0.4,-1) rectangle (1.6,-1.5) node[midway] {$g,(m)$};
      \draw[->] (-1.2,1) -- (-1.2,-1);
      \draw[<-] (1.2,1) -- (1.2,-1);
      \draw[->] (-0.8,1) arc (180:360:0.8);
      \bluedot{(0.72,0.6)} node[color=black,anchor=north] {$c^\vee$};
      \draw[<-] (-0.8,-1) arc (180:0:0.8);
      \bluedot{(-0.72,-0.6)} node[color=black,anchor=south] {$c$};
    \end{tikzpicture}\ ,
  \end{multline*}
  where the second equality follows from the inductive hypothesis.  Therefore, \eqref{eq:multi-strand-up-down-rel} holds when $n=1$.

  Now suppose $n>1$.  Then
  \begin{align*}
    \begin{tikzpicture}[>=stealth,baseline={([yshift=-.5ex]current bounding box.center)}]
      \draw (-0.4,1) rectangle (-1.6,1.5) node[midway] {$f,(n)$};
      \draw (0.4,1) rectangle (1.6,1.5) node[midway] {$g,(m)$};
      \draw (-0.4,-1) rectangle (-1.6,-1.5) node[midway] {$f,(n)$};
      \draw (0.4,-1) rectangle (1.6,-1.5) node[midway] {$g,(m)$};
      \draw[->] (-1,1) .. controls (1,0) .. (-1,-1);
      \draw[<-] (1,1) .. controls (-1,0) .. (1,-1);
    \end{tikzpicture}
    \ &=\
    \begin{tikzpicture}[>=stealth,baseline={([yshift=-.5ex]current bounding box.center)}]
      \draw (-0.4,1) rectangle (-1.6,1.5) node[midway] {$f,(n)$};
      \draw (0.4,1) rectangle (1.6,1.5) node[midway] {$g,(m)$};
      \draw (-0.4,-1) rectangle (-1.6,-1.5) node[midway] {$f,(n)$};
      \draw (0.4,-1) rectangle (1.6,-1.5) node[midway] {$g,(m)$};
      \draw[->] (-0.8,1) .. controls (1,0) .. (-0.8,-1);
      \draw[<-] (1,1) .. controls (-1,0) .. (1,-1);
      \draw[->] (-1.2,1) -- (-1.2,-1) node[pos=0.4, anchor=east] {$1$};
    \end{tikzpicture}
    \ - m \sum_{c \in \cB}\
    \begin{tikzpicture}[>=stealth,baseline={([yshift=-.5ex]current bounding box.center)}]
      \draw (-0.4,1) rectangle (-1.6,1.5) node[midway] {$f,(n)$};
      \draw (0.4,1) rectangle (1.6,1.5) node[midway] {$g,(m)$};
      \draw (-0.4,-1) rectangle (-1.6,-1.5) node[midway] {$f,(n)$};
      \draw (0.4,-1) rectangle (1.6,-1.5) node[midway] {$g,(m)$};
      \draw[->] (-1.2,1) .. controls (-1.2,0) and (0.5,0.5) .. (0.5,0) .. controls (0.5,-0.5) and (-1.2,0) .. (-1.2,-1);
      \draw[<-] (1.2,1) .. controls (1.2,0) and (-0.5,0.5) .. (-0.5,0) .. controls (-0.5,-0.5) and (1.2,0) .. (1.2,-1);
      \draw[->] (-0.8,1) .. controls (-0.8,0.2) and (0.8,0.2) .. (0.8,1);
      \bluedot{(0.76,0.75)} node[color=black, anchor=east] {$c^\vee$};
      \draw[<-] (-0.8,-1) .. controls (-0.8,-0.2) and (0.8,-0.2) .. (0.8,-1);
      \bluedot{(-0.76,-0.75)} node[color=black, anchor=west] {$c$};
    \end{tikzpicture} \\
    &= \sum_{\ell=0}^{\min \{m,n-1\}} \sum_{\bb \in \cB^{\otimes \ell}} (-1)^\ell \frac{m! (n-1)!}{\ell! (m-\ell)! (n-1-\ell)!}\
    \begin{tikzpicture}[>=stealth,baseline={([yshift=-.5ex]current bounding box.center)}]
      \draw (-0.4,1) rectangle (-1.6,1.5) node[midway] {$f,(n)$};
      \draw (0.4,1) rectangle (1.6,1.5) node[midway] {$g,(m)$};
      \draw (-0.4,-1) rectangle (-1.6,-1.5) node[midway] {$f,(n)$};
      \draw (0.4,-1) rectangle (1.6,-1.5) node[midway] {$g,(m)$};
      \draw[->] (-1.2,1) -- (-1.2,-1);
      \draw[<-] (1.2,1) -- (1.2,-1);
      \draw[->] (-0.8,1) arc (180:360:0.8);
      \bluedot{(0.72,0.6)} node[color=black,anchor=north] {$\bb^\vee$};
      \draw[<-] (-0.8,-1) arc (180:0:0.8);
      \bluedot{(-0.72,-0.6)} node[color=black,anchor=south] {$\bb$};
    \end{tikzpicture} \\
    &\qquad - m \sum_{\ell=0}^{\min \{m-1,n-1\}} \sum_{\substack{c \in \cB \\ \bb \in \cB^{\otimes \ell}}} (-1)^\ell\frac{(m-1)! (n-1)!}{\ell! (m-1-\ell)! (n-1-\ell)!}\
    \begin{tikzpicture}[>=stealth,baseline={([yshift=-.5ex]current bounding box.center)}]
      \draw (-0.4,1) rectangle (-1.6,1.5) node[midway] {$f,(n)$};
      \draw (0.4,1) rectangle (1.6,1.5) node[midway] {$g,(m)$};
      \draw (-0.4,-1) rectangle (-1.6,-1.5) node[midway] {$f,(n)$};
      \draw (0.4,-1) rectangle (1.6,-1.5) node[midway] {$g,(m)$};
      \draw[->] (-1.2,1) -- (-1.2,-1);
      \draw[<-] (1.2,1) -- (1.2,-1);
      \draw[->] (-0.8,1) arc (180:360:0.8);
      \bluedot{(0.72,0.6)} node[color=black,anchor=east] {$\bb^\vee\!  \otimes c^\vee \!$};
      \draw[<-] (-0.8,-1) arc (180:0:0.8);
      \bluedot{(-0.72,-0.6)} node[color=black,anchor=west] {$\bb \otimes c$};
    \end{tikzpicture} \\
    &= \sum_{\ell=0}^{\min \{m,n-1\}} \sum_{\bb \in \cB^{\otimes \ell}} (-1)^\ell \frac{m! (n-1)!}{\ell! (m-\ell)! (n-1-\ell)!}\
    \begin{tikzpicture}[>=stealth,baseline={([yshift=-.5ex]current bounding box.center)}]
      \draw (-0.4,1) rectangle (-1.6,1.5) node[midway] {$f,(n)$};
      \draw (0.4,1) rectangle (1.6,1.5) node[midway] {$g,(m)$};
      \draw (-0.4,-1) rectangle (-1.6,-1.5) node[midway] {$f,(n)$};
      \draw (0.4,-1) rectangle (1.6,-1.5) node[midway] {$g,(m)$};
      \draw[->] (-1.2,1) -- (-1.2,-1);
      \draw[<-] (1.2,1) -- (1.2,-1);
      \draw[->] (-0.8,1) arc (180:360:0.8);
      \bluedot{(0.72,0.6)} node[color=black,anchor=north] {$\bb^\vee$};
      \draw[<-] (-0.8,-1) arc (180:0:0.8);
      \bluedot{(-0.72,-0.6)} node[color=black,anchor=south] {$\bb$};
    \end{tikzpicture} \\
    &\qquad + \sum_{\ell=1}^{\min \{m,n\}} \sum_{\bb \in \cB^{\otimes \ell}} (-1)^\ell\frac{m! (n-1)!}{(\ell-1)! (m-\ell)! (n-\ell)!}\
    \begin{tikzpicture}[>=stealth,baseline={([yshift=-.5ex]current bounding box.center)}]
      \draw (-0.4,1) rectangle (-1.6,1.5) node[midway] {$f,(n)$};
      \draw (0.4,1) rectangle (1.6,1.5) node[midway] {$g,(m)$};
      \draw (-0.4,-1) rectangle (-1.6,-1.5) node[midway] {$f,(n)$};
      \draw (0.4,-1) rectangle (1.6,-1.5) node[midway] {$g,(m)$};
      \draw[->] (-1.2,1) -- (-1.2,-1);
      \draw[<-] (1.2,1) -- (1.2,-1);
      \draw[->] (-0.8,1) arc (180:360:0.8);
      \bluedot{(0.72,0.6)} node[color=black,anchor=east] {$\bb^\vee$};
      \draw[<-] (-0.8,-1) arc (180:0:0.8);
      \bluedot{(-0.72,-0.6)} node[color=black,anchor=west] {$\bb$};
    \end{tikzpicture}\ ,
  \end{align*}
  where we used the $n=1$ case in the first equality and the inductive hypothesis in the second equality.  Now, for $1 \le \ell \le \min\{m,n-1\}$, the sum of the coefficients in front of the last two diagrams above is equal to
  \begin{equation} \label{eq:coefficient-sum}
    (-1)^\ell \frac{m!(n-1)!}{(\ell-1)!(m-\ell)!(n-1-\ell)!} \left( \frac{1}{\ell} + \frac{1}{n-\ell} \right) = (-1)^\ell \frac{m!n!}{\ell!(m-\ell)!(n-\ell)!}.
  \end{equation}
  In addition, when $\ell=0$, the coefficient of the second-to-last diagram above also agrees with the right-hand side of  \eqref{eq:coefficient-sum}.  Finally, when $\ell=n$, the coefficient of the last diagram above also agrees with the right-hand side of \eqref{eq:coefficient-sum}.  This completes the proof of the inductive step.
\end{proof}

There exists a covariant autoequivalence $\Omega' \colon \cH_B' \to \cH_B'$ which is the identity on objects and, on morphisms, maps a diagram $D$ to $(-1)^{w(D)}$, where $w(D)$ is the number of crossings in $D$.  The functor $\Omega'$ induces a covariant autoequivalence $\Omega$ of $\cH_B$ squaring to the identity.  It is clear that for any $n \in \N_+$ and idempotent $f \in B$, we have
\[
  \Omega' \left( e_{f,(n)} \right) = e_{f,(1^n)},\quad \Omega' \left( e_{f,(1^n)} \right) = e_{f,(n)}.
\]
Thus,
\begin{equation} \label{eq:Psi-reversals}
  \Omega \left( \sP_f^{(n)} \right) = \sP_f^{(1^n)},\quad
  \Omega \left( \sP_f^{(1^n)} \right) = \sP_f^{(n)},\quad
  \Omega \left( \sQ_f^{(n)} \right) = \sQ_f^{(1^n)},\quad
  \Omega \left( \sQ_f^{(1^n)} \right) = \sQ_f^{(n)}.
\end{equation}

We will often need to consider the sum of several copies of an object of $\cH_B$ with different shifts. To simplify notation, if $V$ is a $\Z\times \Z_2$-graded vector space over $\F$ with a homogeneous basis $\{v\}_{v\in \mathfrak{b}}$ and $(M,h)$ is an object of $\cH_B$, then we define the notation
\begin{equation} \label{eq:mult-space}
  V\otimes (M,h):= \bigoplus_{v \in \mathfrak{b} }\{|v|,\bar{v}\}(M,h).
\end{equation}
It is clear that the above is independent of the chosen basis of $V$.

The following theorem is one of the key indegredients in the proof of our main theorem (Theorem~\ref{theo:main-iso}).  It is a categorical analogue of the relations of Proposition~\ref{prop:hB-presentation} and Remark~\ref{rem:even-type-Q}.

\begin{theo} \label{theo:functor-isos}
  Suppose $f,g \in B$ are idempotents, $n,m \in \N_+$, and $i \in \{r+1,\dotsc,\ell\}$.  Then, in the category $\cH_B$, we have the following isomorphisms:
  \begin{gather}
    \sQ^{(n)}_f \otimes \sQ^{(m)}_g \cong \sQ^{(m)}_g \otimes \sQ^{(n)}_f, \quad
    \sQ^{(1^n)}_f \otimes \sQ^{(1^m)}_g \cong \sQ^{(1^m)}_g \otimes \sQ^{(1^n)}_f, \label{eq:iso-Sn} \\
    \sP^{(n)}_f \otimes \sP^{(m)}_g \cong \sP^{(m)}_g \otimes \sP^{(n)}_f, \quad
    \sP^{(1^n)}_f \otimes \sP^{(1^m)}_g \cong \sP^{(1^m)}_g \otimes \sP^{(1^n)}_f,\label{eq:iso-Ln} \\
    \bigoplus_{k=0}^n \sQ_{e_i}^{(2k)} \otimes \sQ_{e_i}^{(2n-2k)} \cong \bigoplus_{k=0}^{n-1} \sQ_{e_i}^{(2k+1)} \otimes \sQ_{e_i}^{(2n-2k-1)}, \label{eq:Q-even-isom} \\
    \bigoplus_{k=0}^n \sP_{e_i}^{(2k)} \otimes \sP_{e_i}^{(2n-2k)} \cong \bigoplus_{k=0}^{n-1} \sP_{e_i}^{(2k+1)} \otimes \sP_{e_i}^{(2n-2k-1)}, \label{eq:P-even-isom} \\
    \sQ^{(n)}_f \otimes \sP^{(m)}_g \cong \bigoplus_{k\geq 0} S^k(fBg) \otimes \left( \sP^{(m-k)}_g\otimes \sQ^{(n-k)}_f \right), \label{eq:iso-LSn} \\
    \sQ^{(1^n)}_f \otimes \sP^{(1^m)}_g \cong \bigoplus_{k\geq 0} S^k(fBg) \otimes \left( \sP^{(1^{m-k})}_g\otimes \sQ^{(1^{n-k})}_f \right), \label{eq:iso-1n1m} \\
    \sQ^{(n)}_f \otimes \sP^{(1^m)}_g \cong \bigoplus_{k\geq 0} \Lambda^k(fBg) \otimes \left( \sP^{(1^{m-k})}_g\otimes \sQ^{(n-k)}_f \right), \label{eq:iso-n1m} \\
    \sQ^{(1^n)}_f \otimes \sP^{(m)}_g \cong \bigoplus_{k\geq 0} \Lambda^k(fBg) \otimes \left( \sP^{(m-k)}_g\otimes \sQ^{(1^{n-k})}_f \right), \label{eq:iso-1nm}
  \end{gather}
  where we have used the notation from \eqref{eq:mult-space}.  Note that the space $fBg$ appearing above is naturally isomorphic (as a vector space) to $\HOM_B(Bf,Bg)$.
\end{theo}

\begin{proof}
  To show \eqref{eq:iso-Ln}, let $\lambda=(n)$, $\mu=(m)$ or $\lambda=(1^n)$, $\mu=(1^m)$, and consider the morphisms
  \begin{equation} \label{eq:PQ-commute-crossing}
    \begin{tikzpicture}[>=stealth,baseline={([yshift=-.5ex]current bounding box.center)}]
      \draw (0,0) rectangle (1,0.5) node[midway] {$f,\lambda$};
      \draw (2,0) rectangle (3,0.5) node[midway] {$g,\mu$};
      \draw (0,1.5) rectangle (1,2) node[midway] {$g,\mu$};
      \draw (2,1.5) rectangle (3,2) node[midway] {$f,\lambda$};
      \draw[<-] (0.5,1.5) -- (2.5,0.5);
      \draw[<-] (2.5,1.5) -- (0.5,0.5);
    \end{tikzpicture}
    \qquad \text{and} \qquad
    \begin{tikzpicture}[>=stealth,baseline={([yshift=-.5ex]current bounding box.center)}]
      \draw (0,0) rectangle (1,0.5) node[midway] {$g,\mu$};
      \draw (2,0) rectangle (3,0.5) node[midway] {$f,\lambda$};
      \draw (0,1.5) rectangle (1,2) node[midway] {$f,\lambda$};
      \draw (2,1.5) rectangle (3,2) node[midway] {$g,\mu$};
      \draw[<-] (0.5,1.5) -- (2.5,0.5);
      \draw[<-] (2.5,1.5) -- (0.5,0.5);
    \end{tikzpicture}\ .
  \end{equation}
  The first morphism of \eqref{eq:PQ-commute-crossing}, followed by the second, is equal to
  \[
    \begin{tikzpicture}[>=stealth,baseline={([yshift=-.5ex]current bounding box.center)}]
      \draw (0,0) rectangle (1,0.5) node[midway] {$f,\lambda$};
      \draw (2,0) rectangle (3,0.5) node[midway] {$g,\mu$};
      \draw (0,1.5) rectangle (1,2) node[midway] {$g,\mu$};
      \draw (2,1.5) rectangle (3,2) node[midway] {$f,\lambda$};
      \draw (0,3) rectangle (1,3.5) node[midway] {$f,\lambda$};
      \draw (2,3) rectangle (3,3.5) node[midway] {$g,\mu$};
      \draw[<-] (0.5,1.5) -- (2.5,0.5);
      \draw[<-] (2.5,1.5) -- (0.5,0.5);
      \draw[<-] (0.5,3) -- (2.5,2);
      \draw[<-] (2.5,3) -- (0.5,2);
    \end{tikzpicture}
    \ \stackrel{\substack{\eqref{eq:rel0a}, \eqref{eq:rel0b} \\ \eqref{eq:triple-point}, \eqref{eq:idem-diag}}}{=}\
    \begin{tikzpicture}[>=stealth,baseline={([yshift=-.5ex]current bounding box.center)}]
      \draw (0,0) rectangle (1,0.5) node[midway] {$f,\lambda$};
      \draw (2,0) rectangle (3,0.5) node[midway] {$g,\mu$};
      \draw (0,2) rectangle (1,2.5) node[midway] {$f,\lambda$};
      \draw (2,2) rectangle (3,2.5) node[midway] {$g,\mu$};
      \draw[<-] (0.5,2) .. controls (3,1.25) .. (0.5,0.5);
      \draw[<-] (2.5,2) .. controls (0,1.25) .. (2.5,0.5);
    \end{tikzpicture}
    \ \stackrel{\eqref{eq:rel1b}}{=}\
    \begin{tikzpicture}[>=stealth,baseline={([yshift=-.5ex]current bounding box.center)}]
      \draw (0,0) rectangle (1,0.5) node[midway] {$f,\lambda$};
      \draw (2,0) rectangle (3,0.5) node[midway] {$g,\mu$};
      \draw (0,2) rectangle (1,2.5) node[midway] {$f,\lambda$};
      \draw (2,2) rectangle (3,2.5) node[midway] {$g,\mu$};
      \draw[<-] (0.5,2) -- (0.5,0.5);
      \draw[<-] (2.5,2) -- (2.5,0.5);
    \end{tikzpicture}\ .
  \]
  Similarly, the second morphism of \eqref{eq:PQ-commute-crossing}, followed by the first, is also equal to the identity.  Thus, the morphisms of \eqref{eq:PQ-commute-crossing} are inverses of each other.  This gives the isomorphism \eqref{eq:iso-Ln}.  The proof of \eqref{eq:iso-Sn} is analogous.

  It follows from the fact that the representation theory of $\bS_n$ is governed by the algebra of Schur $Q$-functions (see \cite[\S 4C]{Joz89} or the more recent \cite[Th.~3.5]{WW12}) and the identity \cite[III, (8.2)]{Mac95} that we have an isomorphism of $\bS_{2n}$-modules
  \[
    \bigoplus_{k=0}^n \bS_{2n} \left( e_{(2k)} \otimes e_{(2n-2k)} \right) \cong \bigoplus_{k=0}^{n-1} \bS_{2n} \left( e_{(2k+1)} \otimes e_{(2n-2k-1)} \right).
  \]
  Note that no parity shifts are needed in the above isomorphism because $\bS_{m} e_{(m)}$ is a type $Q$ module for all $m \in \N_+$, and hence all summands above are invariant under parity shift.
  \details{
    The map $\pr{\left( \sum_{k=1}^m 1^{\otimes (k-1)} \otimes c \otimes 1^{\otimes (m-k)} \right)}$ yields an isomorphism $\bS_{m}e_{(m) } \cong \Pi \bS_{m} e_{(m)}$.
  }
  Tensoring on left with $P_i^{\otimes 2n} \cong B^{\otimes 2n} e_i^{\otimes 2n}$ over $\bS_{2n}$ yields an isomorphism of $B^{\otimes 2n} \rtimes S_{2n}$-modules
  \[
    \bigoplus_{k=0}^n B^{\otimes 2n} \left( e_{(2k) \langle i \rangle} \otimes e_{(2n-2k) \langle i \rangle} \right) \cong \bigoplus_{k=0}^{n-1} B^{\otimes 2n} \left( e_{(2k+1) \langle i \rangle} \otimes e_{(2n-2k-1) \langle i \rangle} \right).
  \]
  Thus, there exists an invertible element in
  \begin{multline*}
    \Hom_{B^{\otimes 2n} \rtimes S_{2n}} \left( \bigoplus_{k=0}^n B^{\otimes 2n} \left( e_{(2k) \langle i \rangle} \otimes e_{(2n-2k) \langle i \rangle} \right), \bigoplus_{k=0}^{n-1} B^{\otimes 2n} \left( e_{(2k+1) \langle i \rangle} \otimes e_{(2n-2k-1) \langle i \rangle} \right) \right) \\
    \cong \bigoplus_{k=0}^n \bigoplus_{k'=0}^{m-1} \left( e_{(2k) \langle i \rangle} \otimes e_{(2n-2k) \langle i \rangle} \right) B^{\otimes 2n} \left( e_{(2k'+1) \langle i \rangle} \otimes e_{(2n-2k'-1) \langle i \rangle} \right).
  \end{multline*}
  Under the map \eqref{eq:map-Bn-Qn}, this yields the isomorphism \eqref{eq:Q-even-isom}.  The proof of \eqref{eq:P-even-isom} is analogous.

  The isomorphism \eqref{eq:iso-LSn} requires a little more work.  Fix a homogeneous basis $\cB$ of $B$ with right dual basis $\cB^\vee$.  For $0 \leq k\leq \min\{m,n\}$ and $\bb = b_1 \otimes \dotsb \otimes b_k \in B^{\otimes k}$, we define
  \[
    \alpha_\bb =\
    \begin{tikzpicture}[>=stealth,baseline={([yshift=-.5ex]current bounding box.center)}]
      \draw (-0.3,1) rectangle (-2.2,1.5) node[midway] {$g,(m-k)$};
      \draw (0.3,1) rectangle (2.2,1.5) node[midway] {$f,(n-k)$};
      \draw (-0.5,-1) rectangle (-2,-1.5) node[midway] {$f,(n)$};
      \draw (0.5,-1) rectangle (2,-1.5) node[midway] {$g,(m)$};
      \draw[<-] (-0.8,-1) arc (180:0:0.8);
      \bluedot{(-0.72,-0.6)} node[color=black,anchor=west] {$\bb$};
      \draw[<-] (-1.25,1) .. controls (-1.25,0.3) and (1.7,-0.3) .. (1.7,-1);
      \draw[->] (1.25,1) .. controls (1.25,0.3) and (-1.7,-0.3) .. (-1.7,-1);
    \end{tikzpicture}
    \ \colon  \sQ^{(n)}_f \otimes \sP^{(m)}_g \to \{|\bb^\vee|,\bar{\bb}^\vee\} (\sP^{(m-k)}_g \otimes \sQ^{(n-k)}_f).
  \]
  The degree of this diagram is the sum of the degrees of the caps and the dots, and so is equal to $k(-\delta,\sigma) + (|\bb|,\bar \bb) = (-|\bb^\vee|,\bar \bb^\vee)$.  Note also that $\alpha_\bb=0$ if $\psi^{-1}(g)b_\ell f=0$ for some $\ell \in \{1,\dotsc,k\}$.

  If we let $\cB_{\psi^{-1}(g),f}$ be a homogeneous basis of $\psi^{-1}(g)B f\subseteq B$ and analogously for other idempotents, we can, without loss of generality, assume that
  \[
    \cB = \cB_{\psi^{-1}(g),f} \sqcup \cB_{(1-\psi^{-1}(g)),f} \sqcup \cB_{\psi^{-1}(g),(1-f)} \sqcup \cB_{(1-\psi^{-1}(g)),(1-f)}.
  \]
  Then, for all $b \in \cB$, we have
  \begin{equation}
    \psi^{-1}(g) b f =
    \begin{cases}
      b & \text{if } b \in \cB_{\psi^{-1}(g),f}, \\
      0 & \text{otherwise}.
    \end{cases}
  \end{equation}
  Notice that, if $b \in \cB_{\psi^{-1}(g),f}$ and $c \in \cB$, then
  \[
    \delta_{b,c} = \tr_B(b c^\vee) = \tr_B(\psi^{-1}(g) b f c^\vee) = \tr_B(b f c^\vee g),
  \]
  and so
  \[
    \cB^\vee = \cB^\vee_{f,g} \sqcup \cB^\vee_{f,(1-g)} \sqcup \cB^\vee_{(1-f),g} \sqcup \cB^\vee_{(1-f),(1-g)}.
  \]
  where $\cB^\vee_{f,g} \subseteq \cB^\vee$ is the right dual basis to $\cB_{\psi^{-1}(g),f}$ and similarly for the other idempotents.  In particular,
  \[
    \psi^{-1}(g) b f \ne 0 \iff f b^\vee g \ne 0 \iff b^\vee \in \cB^\vee_{f,g}.
  \]
  Hence $\alpha_\bb = 0$ unless $b_\ell^\vee \in \cB^\vee_{f,g}$ for all $\ell = 1,\dotsc,k$.

  For a given $\bb^\vee= b_1^\vee \otimes \dotsb \otimes b_k^\vee \in (\cB^\vee_{f,g})^{\otimes k}$, notice that there are two possibilities for $e_{f,(k)}{\bb}^\vee e_{g,(k)}$.  The first possibility is that there exists  $w \in S_k$ such that $w \cdot \bb^\vee = -\bb^\vee$, which happens exactly if there exist $1\leq r,s\leq k$, $r \ne s$, such that $b_r^\vee=b_s^\vee$ and $\bar{b}_r^\vee=1$. If this is the case, then $e_{f,(k)}{\bb}^\vee e_{g,(k)}=0$. The other possibility is that $e_{f,(k)}{\bb}^\vee e_{g,(k)}\neq 0$, in which case we let
  \begin{equation} \label{eq:Sbbvee-def}
    S_{\bb^\vee}=\{w\in S_k \mid w \cdot \bb^\vee = \bb^\vee\}.
  \end{equation}

  Now, let $\mathcal{D}^k_{f,g} \subseteq (\cB_{\psi^{-1}(g),f})^{\otimes k}$ be such that $\{e_{f,(k)}{\bb}^\vee e_{g,(k)} \mid \bb \in \mathcal{D}^k_{f,g}\}$ is a basis of $e_{f,(k)}A_ke_{g,(k)}$.  Then, taking the sum over all $\bb \in \mathcal{D}^k_{f,g}$, we get a 2-morphism
  \begin{equation} \label{eq:alpha-b-map}
    \alpha_k = \bigoplus_{{\bb} \in \mathcal{D}^k_{f,g}} \alpha_\bb \colon \sQ^{(n)}_f \otimes \sP^{(m)}_g  \to \bigoplus\limits_{\bb\in\mathcal{D}^k_{f,g}}\{|\bb^\vee|,\bar{\bb}^\vee\} (\sP^{(m-k)}_g \otimes \sQ^{(n-k)}_f) = (e_{f,(k)}A_ke_{g,(k)})\otimes (\sP^{(m-k)}_g \otimes \sQ^{(n-k)}_f)
  \end{equation}
  where we have used the notation of \eqref{eq:mult-space}.  Since, as in \eqref{eq:hom-symmetrizer}, there is an isomorphism of $\Z\times\Z_2$-graded vector spaces
  \[
    e_{f,(k)} A_k e_{g,(k)} \cong \HOM_{A_k}(A_ke_{f,(k)},A_ke_{g,(k)}) \cong S^k(\HOM_B(Bf,Bg)),
  \]
  we can also rewrite \eqref{eq:alpha-b-map} as
  \begin{equation} \label{eq:alpha-b-map2}
    \alpha_k \colon  \sQ^{(n)}_f \otimes \sP^{(m)}_g \to S^k(\HOM_B(Bf,Bg)) \otimes (\sP^{(m-k)}_g \otimes \sQ^{(n-k)}_f).
  \end{equation}

  For $0 \leq k \leq \min\{m,n\}$ and $\bb \in B^{\otimes k}$, we define
  \begin{equation}
    \beta_\bb =\
    \begin{tikzpicture}[>=stealth,baseline={([yshift=-.5ex]current bounding box.center)}]
      \draw (-0.3,-1) rectangle (-2.2,-1.5) node[midway] {$g,(m-k)$};
      \draw (0.3,-1) rectangle (2.2,-1.5) node[midway] {$f,(n-k)$};
      \draw (-0.5,1) rectangle (-2,1.5) node[midway] {$f,(n)$};
      \draw (0.5,1) rectangle (2,1.5) node[midway] {$g,(m)$};
      \draw[->] (-0.8,1) arc (180:360:0.8);
      \bluedot{(+0.72,0.6)} node[color=black,anchor=east] {$\bb$};
      \draw[->] (-1.25,-1) .. controls (-1.25,-0.3) and (1.7,0.3) .. (1.7,1);
      \draw[<-] (1.25,-1) .. controls (1.25,-0.3) and (-1.7,0.3) .. (-1.7,1);
    \end{tikzpicture}
    \ \colon \{|\bb|,\bar{\bb}\} \left( \sP^{(m-k)}_g \otimes \sQ^{(n-k)}_f \right) \to  \sQ^{(n)}_f \otimes \sP^{(m)}_g.
  \end{equation}
  We then define
  \begin{equation}
    \beta_k = \sum_{\bb \in \mathcal{D}_{f,g}^k} \frac{m!n!}{|S_{{\bb}^\vee}|(m-k)!(n-k)!} \beta_{\bb^\vee} \colon \HOM_{A_k}(A_ke_{f,(k)},A_ke_{g,(k)})\otimes \left( \sP^{(m-k)}_g \otimes \sQ^{(n-k)}_f \right) \to \sQ^{(n)}_f \otimes \sP^{(m)}_g.
  \end{equation}
  We will prove that
  \begin{equation}
    \alpha := \bigoplus_{k=0}^{\min \{m,n\}} \alpha_k \quad \text{and} \quad
    \beta := \sum_{k=0}^{\min \{m,n\}} \beta_k
  \end{equation}
  are mutually inverse morphisms.  Note that we use the notation $\bigoplus$ in the definition of $\alpha$ to denote a map into a direct sum with components $\alpha_k$ (and similarly in \eqref{eq:alpha-b-map}) and the notation $\sum$ in the definition of $\beta$ to denote a map from a direct sum.  In matrix notation, $\alpha$ would be a column matrix with entries $\alpha_k$ and $\beta$ would be a row matrix with entries $\beta_k$.

  We first prove that
  \begin{equation} \label{eq:simplify-identity}
    \beta \circ \alpha = \sum_{k=0}^{\min\{n,m\}} \sum_{\bb \in \mathcal{D}_{f,g}^k} \beta_{\bb^\vee} \circ \alpha_\bb = \id_{\sQ^{(n)}_f \otimes \sP^{(m)}_g}.
  \end{equation}
  Now, for $\bb \in B^{\otimes k}$, we have
  \begin{multline} \label{eq:beta-alpha-simplify}
    \beta_{{\bb}^\vee}\circ \alpha_\bb
    \ =\
    \begin{tikzpicture}[>=stealth,baseline={([yshift=-.5ex]current bounding box.center)}]
      \draw (-0.3,1) rectangle (-2.2,1.5) node[midway] {$g,(m-k)$};
      \draw (0.3,1) rectangle (2.2,1.5) node[midway] {$f,(n-k)$};
      \draw (-0.5,-1) rectangle (-2,-1.5) node[midway] {$f,(n)$};
      \draw (0.5,-1) rectangle (2,-1.5) node[midway] {$g,(m)$};
      \draw[<-] (-0.8,-1) arc (180:0:0.8);
      \bluedot{(-0.72,-0.6)} node[color=black,anchor=west] {$\bb$};
      \draw[<-] (-1.25,1) .. controls (-1.25,0.3) and (1.7,-0.3) .. (1.7,-1);
      \draw[->] (1.25,1) .. controls (1.25,0.3) and (-1.7,-0.3) .. (-1.7,-1);
      \draw (-0.5,3.5) rectangle (-2,4) node[midway] {$f,(n)$};
      \draw (0.5,3.5) rectangle (2,4) node[midway] {$g,(m)$};
      \draw[->] (-0.8,3.5) arc (180:360:0.8);
      \bluedot{(+0.72,3.1)} node[color=black,anchor=east] {$\bb^\vee$};
      \draw[->] (-1.25,1.5) .. controls (-1.25,2.2) and (1.7,2.8) .. (1.7,3.5);
      \draw[<-] (1.25,1.5) .. controls (1.25,2.2) and (-1.7,2.8) .. (-1.7,3.5);
    \end{tikzpicture}
    \stackrel{\substack{\eqref{eq:rel0a}, \eqref{eq:rel0b} \\ \eqref{eq:triple-point}, \eqref{eq:idem-diag}}}{=}
    \begin{tikzpicture}[>=stealth,baseline={([yshift=-.5ex]current bounding box.center)}]
      \draw (-0.5,-1) rectangle (-2,-1.5) node[midway] {$f,(n)$};
      \draw (0.5,-1) rectangle (2,-1.5) node[midway] {$g,(m)$};
      \draw[<-] (-0.8,-1) arc (180:0:0.8);
      \bluedot{(-0.72,-0.6)} node[color=black,anchor=west] {$\bb$};
      \draw[<-] (1.7,3).. controls (1.7,2.3) and (-1.25,1.7).. (-1.25,1) .. controls (-1.25,0.3) and (1.7,-0.3) .. (1.7,-1);
      \draw[->] (-1.7,3).. controls (-1.7,2.3) and (1.25,1.7).. (1.25,1) .. controls (1.25,0.3) and (-1.7,-0.3) .. (-1.7,-1);
      \draw (-0.5,3) rectangle (-2,3.5) node[midway] {$f,(n)$};
      \draw (0.5,3) rectangle (2,3.5) node[midway] {$g,(m)$};
      \draw[->] (-0.8,3) arc (180:360:0.8);
      \bluedot{(+0.72,2.6)} node[color=black,anchor=east] {$\bb^\vee$};
    \end{tikzpicture} \\
    \stackrel{\eqref{eq:multi-strand-up-down-rel}}{=} \sum_{\substack{k\leq \ell\leq\min\{m,n\} \\ {\bc}\in(\cB_{\psi^{-1}(g),f})^{\otimes (\ell-k)}}} (-1)^{\ell-k}(\ell-k)!{{m-k} \choose {\ell-k}} {{n-k} \choose {\ell-k}}\
    \begin{tikzpicture}[>=stealth,baseline={([yshift=-.5ex]current bounding box.center)}]
      \draw (-0.4,1) rectangle (-1.6,1.5) node[midway] {$f,(n)$};
      \draw (0.4,1) rectangle (1.6,1.5) node[midway] {$g,(m)$};
      \draw (-0.4,-1) rectangle (-1.6,-1.5) node[midway] {$f,(n)$};
      \draw (0.4,-1) rectangle (1.6,-1.5) node[midway] {$g,(m)$};
      \draw[->] (-1.2,1) -- (-1.2,-1);
      \draw[<-] (1.2,1) -- (1.2,-1);
      \draw[->] (-0.9,1) arc (180:360:0.85);
      \bluedot{(0.72,0.6)} node[color=black,anchor=east] {$\bc^\vee\! \otimes\! \bb^\vee$};
      \draw[<-] (-0.8,-1) arc (180:0:0.8);
      \bluedot{(-0.72,-0.6)} node[color=black,anchor=west] {$\bc \otimes \bb$};
    \end{tikzpicture}\ .
  \end{multline}
  For $1 \le k \le \ell \le \min \{m,n\}$, and $\bb \in B^{\otimes k}$, define
  \[
    \gamma^{f,n,g,m}_{\bb,\ell} = \sum_{{\bc}\in(\cB_{\psi^{-1}(g),f})^{\otimes {(\ell-k)}}}\
    \begin{tikzpicture}[>=stealth,baseline={([yshift=-.5ex]current bounding box.center)}]
      \draw (-0.4,1) rectangle (-1.6,1.5) node[midway] {$f,(n)$};
      \draw (0.4,1) rectangle (1.6,1.5) node[midway] {$g,(m)$};
      \draw (-0.4,-1) rectangle (-1.6,-1.5) node[midway] {$f,(n)$};
      \draw (0.4,-1) rectangle (1.6,-1.5) node[midway] {$g,(m)$};
      \draw[->] (-1.2,1) -- (-1.2,-1);
      \draw[<-] (1.2,1) -- (1.2,-1);
      \draw[->] (-0.9,1) arc (180:360:0.85);
      \bluedot{(0.72,0.6)} node[color=black,anchor=east] {$\bc^\vee\! \otimes\! \bb^\vee$};
      \draw[<-] (-0.8,-1) arc (180:0:0.8);
      \bluedot{(-0.72,-0.6)} node[color=black,anchor=west] {$\bc \otimes \bb$};
    \end{tikzpicture}
    \quad \text{and} \quad
    \gamma^{f,n,g,m}_\ell = \sum_{\bc \in (\cB_{\psi^{-1}(g),f})^{\otimes {\ell}}}\
    \begin{tikzpicture}[>=stealth,baseline={([yshift=-.5ex]current bounding box.center)}]
      \draw (-0.4,1) rectangle (-1.6,1.5) node[midway] {$f,(n)$};
      \draw (0.4,1) rectangle (1.6,1.5) node[midway] {$g,(m)$};
      \draw (-0.4,-1) rectangle (-1.6,-1.5) node[midway] {$f,(n)$};
      \draw (0.4,-1) rectangle (1.6,-1.5) node[midway] {$g,(m)$};
      \draw[->] (-1.2,1) -- (-1.2,-1);
      \draw[<-] (1.2,1) -- (1.2,-1);
      \draw[->] (-0.8,1) arc (180:360:0.8);
      \bluedot{(0.72,0.6)} node[color=black,anchor=north] {$\bc^\vee$};
      \draw[<-] (-0.8,-1) arc (180:0:0.8);
      \bluedot{(-0.72,-0.6)} node[color=black,anchor=south] {$\bc$};
    \end{tikzpicture}\ .
  \]
  Note that
  \begin{equation}\label{eq:rec-rec}
    \gamma^{f,n,g,m}_\ell
    = \sum_{\bb \in (\cB_{\psi^{-1}(g),f})^{\otimes k}} \gamma^{f,n,g,m}_{\bb,\ell}
    = \sum_{\bb \in \mathcal{D}_{f,g}^k}\frac{k!}{|S_{{\bb}^\vee}|} \gamma^{f,n,g,m}_{\bb,\ell}.
  \end{equation}
  Thus,
  \begin{align*}
    \sum_{k=0}^{\min\{n,m\}} \sum_{{\bb}^\vee\in\mathcal{D}_{f,g}^{ k}} &\frac{m!n!}{|S_{{\bb}^\vee}|(m-k)!(n-k)!} \beta_{{\bb}^\vee}\circ  \alpha_\bb \\
    &\!\!\! \stackrel{\eqref{eq:beta-alpha-simplify}}{=} \sum_{k=0}^{\min\{n,m\}} \sum_{\ell=k}^{\min\{n,m\}}\sum_{{\bb}^\vee\in\mathcal{D}_{f,g}^{ k}} (-1)^{\ell-k} \frac{m!n!(\ell-k)!}{|S_{{\bb}^\vee}|(m-k)!(n-k)!}{{m-k} \choose {\ell-k}} {{n-k} \choose {\ell-k}} \gamma^{f,n,g,m}_{\bb,\ell} \\
    &= \sum_{k=0}^{\min\{n,m\}} \sum_{\ell=k}^{\min\{n,m\}} (-1)^{\ell-k} \frac{m!n!}{k!(\ell-k)!(m-\ell)!(n-\ell)!}
    \left( \sum_{{\bb}^\vee\in\mathcal{D}_{f,g}^{ k}}\frac{k!}{|S_{{\bb}^\vee}|} \gamma^{f,n,g,m}_{\bb,\ell} \right)\\
    &\!\!\! \stackrel{\eqref{eq:rec-rec}}{=} \sum_{\ell=0}^{\min\{n,m\}} \sum_{k=0}^{\ell} (-1)^{\ell-k} \frac{m!n!}{k!(\ell-k)!(m-\ell)!(n-\ell)!} \gamma^{f,n,g,m}_\ell \\
    &= \sum_{\ell=0}^{\min\{n,m\}}\frac{m!n!}{\ell!(m-\ell)!(n-\ell)!}\left( \sum_{k=0}^{\ell} (-1)^{\ell-k}\frac{\ell!}{k!(\ell-k)!}\right) \gamma^{f,n,g,m}_\ell \\
    &= \gamma^{f,n,g,m}_0 = \id_{\sQ^{(n)}_f \otimes \sP^{(m)}_g},
  \end{align*}
  where, in second-to-last equality, we have used the fact that the sum in parentheses in the penultimate line above is the binomial expansion for $(1-1)^\ell$ when $\ell \ge 1$.  Thus, \eqref{eq:simplify-identity} is satisfied.

  Now, for $\bb \in \cB^{\otimes k'}$ and $\bc \in \cB^{\otimes k}$, consider the composition
  \[
    \alpha_\bc \circ \beta_{\bb^\vee} =\
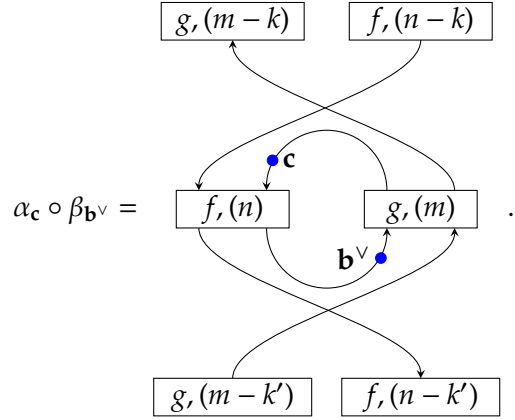
    \begin{tikzpicture}[>=stealth,baseline={([yshift=-.5ex]current bounding box.center)}]
      \draw (-0.2,-1) rectangle (-2.3,-1.5) node[midway] {$g,(m-k')$};
      \draw (0.2,-1) rectangle (2.3,-1.5) node[midway] {$f,(n-k')$};
      \draw (-0.5,1) rectangle (-2,1.5) node[midway] {$f,(n)$};
      \draw (0.5,1) rectangle (2,1.5) node[midway] {$g,(m)$};
      \draw[->] (-0.8,1) arc (180:360:0.8);
      \bluedot{(+0.72,0.6)} node[color=black,anchor=east] {$\bb^\vee$};
      \draw[->] (-1.25,-1) .. controls (-1.25,-0.3) and (1.7,0.3) .. (1.7,1);
      \draw[<-] (1.25,-1) .. controls (1.25,-0.3) and (-1.7,0.3) .. (-1.7,1);
      \draw (-0.3,3.5) rectangle (-2.2,4) node[midway] {$g,(m-k)$};
      \draw (0.3,3.5) rectangle (2.2,4) node[midway] {$f,(n-k)$};
      \draw[<-] (-0.8,1.5) arc (180:0:0.8);
      \bluedot{(-0.72,1.9)} node[color=black,anchor=west] {$\bc$};
      \draw[<-] (-1.25,3.5) .. controls (-1.25,2.8) and (1.7,2.2) .. (1.7,1.5);
      \draw[->] (1.25,3.5) .. controls (1.25,2.8) and (-1.7,2.2) .. (-1.7,1.5);
    \end{tikzpicture}\ .
  \]
  Notice that this composition is zero if $k \neq k'$ because, in that case, when we expand the idempotents of the symmetric group as permutations, it will be impossible to match all caps and cups to get circles and all diagrams will contain a left curl, which is zero by \eqref{eq:rel3b}.  When $k=k'$, again we get left curls for any permutation in the box
  \begin{tikzpicture}[>=stealth,baseline={([yshift=-.5ex]current bounding box.center)}]
    \draw (0,0) rectangle (1.1,0.5) node[midway] {$f,(n)$};
  \end{tikzpicture}
  that is not in $S_{n-k}\times S_k$ and any permutation in the box \begin{tikzpicture}[>=stealth,baseline={([yshift=-.5ex]current bounding box.center)}]
    \draw (0,0) rectangle (1.2,0.5) node[midway] {$g,(m)$};
  \end{tikzpicture}
  that is not in $S_k\times S_{m-k}$. Also, for any permutation in $S_{n-k}\times S_k$ from \begin{tikzpicture}[>=stealth,baseline={([yshift=-.5ex]current bounding box.center)}]
    \draw (0,0) rectangle (1.1,0.5) node[midway] {$f,(n)$};
  \end{tikzpicture}
  we have exactly $(m-k)!$ permutations in
  \begin{tikzpicture}[>=stealth,baseline={([yshift=-.5ex]current bounding box.center)}]
    \draw (0,0) rectangle (1.2,0.5) node[midway] {$g,(m)$};
  \end{tikzpicture}
  that make it so there are no left curls.  When those conditions are satisfied, the circles in the middle will give coefficients of $\tr_B(c f b^\vee g)$ which equals $\delta_{b,c}$ when $b^\vee\in \cB^\vee_{f,g}$ and equals zero otherwise.  Recall that $S_{{\bb}^\vee}=\{w\in S_k \mid w{\bb}^\vee w^{-1}={\bb}^\vee\}$.  Then, among the $k!$ elements of $S_k$, only $|S_{{\bb}^\vee}|$ of them will give that all the circles are nonzero.  Therefore, after resolving the loops and circles and using \eqref{eq:rel2b}, we have
  \[
    \alpha_\bc \circ \beta_{\bb^\vee} =
    \begin{cases}
      \frac{ |S_{{\bb}^\vee}|(n-k)!(m-k)!}{n!m!} \id_{\sP^{(m-k)}_g \sQ^{(n-k)}_f} & \text{if } \bb = \bc \text{ (up to permutation)},\ b_p^\vee \in \cB_{f,g}^\vee \text{ for all } 0 \le p \le k,\\
      0 & \text{otherwise}.
    \end{cases}
  \]
  Now, if $\bc, \bb \in \mathcal{D}_{f,g}^k$, then the condition $b_p^\vee \in \cB_{f,g}^\vee$ is automatically satisfied for all $0 \le p \le k$ and $\bb$ is a permutation of $\bc$ if and only if $\bb = \bc$ (otherwise the elements of $\mathcal{D}_{f,g}^k$ would not be linearly independent).  Thus $\alpha \circ \beta$ is also the identity morphism, completing the proof of \eqref{eq:iso-LSn}.

  Considering the autoequivalence $\Omega$, \eqref{eq:iso-1n1m} follows from \eqref{eq:iso-LSn} by \eqref{eq:Psi-reversals}.  The proof of the isomorphism \eqref{eq:iso-n1m} is analogous to the proof of \eqref{eq:iso-LSn} except that we replace everywhere idempotents $e_{g,(k)}$, $k \in \N_+$, by $e_{g,(1^k)}$, and change the definition of $S_{\bb^\vee}$ in \eqref{eq:Sbbvee-def} to be $S_{\bb^\vee} = \{w \in S_k \mid w \cdot \bb^\vee = (-1)^{\ell(w)} \bb^\vee\}$.  Finally, \eqref{eq:iso-1nm} follows from \eqref{eq:iso-n1m} by another application of the functor $\Omega$.
\end{proof}

\begin{rem} \label{rem:missing-isoms}
  Note that the isomorphisms \eqref{eq:Q-even-isom} and \eqref{eq:P-even-isom} are categorical analogues of the relations  \eqref{eq:Q-even-reduction} and \eqref{eq:P-even-reduction}.  Since those relations are missing in \cite[Prop.~5.1]{CS14} and \cite[Prop.~1]{HS15} (see Remark~\ref{rem:missing-relations}), the authors of those papers do not prove the corresponding isomorphisms in the categories they consider.\footnote{As noted in the footnote to Remark~\ref{rem:missing-relations}, the authors of \cite{HS15} fixed this issue in the published version of their paper.}  Fortunately, these isomorphisms do in fact hold, as can be seen by specializing \eqref{eq:Q-even-isom} and \eqref{eq:P-even-isom} for the appropriate choices of $B$ (see Remark~\ref{rem:special-cases}).  Thus, the assertions in \cite{CS14,HS15} that there exists a map from the Heisenberg algebra to the Grothendieck group of the categories defined in those papers remain true.
\end{rem}

\section{Main result} \label{sec:main-result}

We now prove our main theorem, which identifies the Grothendieck group of the category $\cH_B$ with the Heisenberg algebra $\fh_B$.

By Proposition~\ref{prop:hB-presentation} and Theorems~\ref{theo:action-functor} and~\ref{theo:functor-isos}, we have maps
\begin{equation} \label{eq:main-iso}
  \fh_B \xrightarrow{p} K_0(\cH_B) \xrightarrow{K_0(\bF)} \End\left(\bigoplus_{n\in\N}K_0(A_n\pmd)\right).
\end{equation}
It follows from the definitions of Section \ref{sec:actions} and \cite[Th.~5.7]{RS15a} that the composition $K_0(\bF)\circ p$ is the Fock space representation of $\fh_B$ (the representation induced from the trivial representation of $H^-$).  Viewing $\Q$ as a $\kk$-module via $q \mapsto 1$, $\pi \mapsto 1$, we can extend scalars from $\kk$ to $\Q$.  Recall that $\fh_B \otimes_\kk \Q$ is the Heisenberg double of $H^+ \otimes_\kk \Q$ (see Remark~\ref{rem:Heis-double}). Since the Fock space representation of a Heisenberg double is faithful by \cite[Th.~2.11]{SY15}, we have that $p$ is injective.

\begin{lem}
  Every object of $\cH_B'$ is isomorphic to a direct sum of shifts of objects of the form $\sP^n \sQ^m$.
\end{lem}

\begin{proof}
  By definition, objects of $\cH_B'$ are direct sums of shifts of words in $\sP$ and $\sQ$.  Taking $f=g=1\in B$ and $m=n=1$ in \eqref{eq:iso-LSn}, we have an isomorphism in $\cH_B$ where all the idempotents involved are $1$.  Therefore, it is actually an isomorphism in $\cH'_B$. We can then use that isomorphism to move any $\sQ$'s past the $\sP$'s all the way to the right, and the result follows.
\end{proof}

\begin{lem} \label{lem:end-spaces}
  If $\delta\neq 0$, (where $(-\delta,\sigma)$ is the degree of the trace map $\tr_B$ of $B$), then $\END_{\cH_B'}(\sP^n\sQ^m)_{i,\epsilon}=0$ when $i<0$ and $\END_{\cH_B'}(\sP^n\sQ^m)_{0,0} \cong (A_n^{\op})_{0,0}\otimes_\F (A_m)_{0,0}$.
\end{lem}

\begin{proof}
  Any negative degree endomorphism must contain a left cap, since that is the only generating morphism of negative degree. Since all the $\sP$'s are to the left of all the $\sQ$'s, after using the local relations we have that all left caps can be straightened out unless they are part either of a left curl or of a closed diagram. By \eqref{eq:rel3b}, in the case of a left curl the whole diagram is zero. In the second case, by Proposition \ref{prop:identity-end-surjection}, all closed diagrams can be expressed as combinations of the bubbles $c_{b,d}$ which have positive degrees. It follows that $\END_{\cH_B'}(\sP^n\sQ^m)_{i,\epsilon}=0$ when $i<0$.

  Now consider degree zero morphisms. Any instance of a right curl or of a clockwise circle, since they have strictly positive degrees, has to be balanced by at least one instance of a left cap. In the previous discussion for negative degree morphisms, we have already established that such a left cap has to appear in a closed diagram, which consists of a combination of bubbles $c_{b,d}$ and is therefore of strictly positive degree. It follows that no right curls nor clockwise circles can appear. Notice also that, again because of the order in which the $\sP$'s and $\sQ$'s appear, any right cup has to be part of a right curl or clockwise circle, and hence cannot occur. It then follows that any morphism in $\END_{\cH_B'}(\sP^n\sQ^m)_{0,0}$ can be written as a braid-like diagram (containing neither caps nor cups) possibly carrying some dots. Under the maps \eqref{eq:map-Bn-Pn} and \eqref{eq:map-Bn-Qn}, we can then identify those as elements of $(A_n^{\op})_{0,0}\otimes_\F (A_m)_{0,0}$.
\end{proof}

From Lemma~\ref{lem:end-spaces} and the definition of the category $\cH_B$ as the idempotent completion of $\cH_B'$ we obtain immediately the following.

\begin{cor} \label{cor:1morph}
  If $\delta \ne 0$, then every object of $\cH_B$ is a direct sum of shifts of objects of the form $(\sP^n,e)\otimes (\sQ^m,e')$ for some idempotents $e\in A_n$, $e'\in A_m$.
\end{cor}

\begin{lem} \label{lem:surj}
  If $\delta \ne 0$, then the map $p \colon \fh_B\to K_0(\cH_B)$ of \eqref{eq:main-iso} is surjective.
\end{lem}

\begin{proof}
  By Corollary \ref{cor:1morph}, it is enough to show that, for all $n,m \in \N$, and for all idempotents $e\in A_n$, $e'\in A_m$, we have that the class $[(\sP^n,e)\otimes (\sQ^m,e')]\in K_0(\cH_B)$ is in the image of $p$. But if we take $R=A_ne\in A_n\pmd$ and $S=A_me'\in A_m\pmd$, we have $[R]\#[S]\in\fh_B$ and
  \[
    p([R]\#[S])=[(\sP^n,e)\otimes (\sQ^m,e')]. \qedhere
  \]
\end{proof}

We can now state our main theorem, which follows from the results of this section.

\begin{theo} \label{theo:main-iso}
  Suppose $B$ is an $\N$-graded Frobenius superalgebra not concentrated in $\N$-degree zero, and let $\fh_B$ be the Heisenberg algebra associated to the tower of graded superalgebras $A=\bigoplus_{n\in\N}(B^{\otimes n}\rtimes S_n)$ in Definition~\ref{def:hB} (with presentations as in Proposition~\ref{prop:hB-presentation}).  Then there is an isomorphism of algebras
  \[
    \fh_B \cong K_0(\cH_B).
  \]
\end{theo}

\begin{proof}
  If $B$ is not concentrated in $\N$-degree zero, then its trace map must have strictly negative degree.  That is, $\tr_B$ has degree $(-\delta,\sigma)$ for $\delta > 0$.  The map $p$ of \eqref{eq:main-iso} was observed to be an injective algebra map and by Lemma~\ref{lem:surj} it is surjective, hence it is an isomorphism.
\end{proof}

It follows immediately from \eqref{eq:Psi-reversals} that the functor $\Omega$ induces the automorphism $\omega$ of Corollary~\ref{cor:heis-alg-involution} on the Grothendieck group $K_0(\cH_B) \cong \fh_B$.


\bibliographystyle{alpha}
\bibliography{RossoSavage-biblist}

\def\cprime{$'$} \newcommand{\arxiv}[1]{\href{http://arxiv.org/abs/#1}{\tt
  arXiv:\nolinkurl{#1}}}\newcommand{\doi}[1]{\href{http://dx.doi.org/#1}{\tt
  \nolinkurl{http://dx.doi.org/#1}}}
\begin{thebibliography}{FJW02}

\bibitem[CL]{CL11}
S.~Cautis and A.~Licata.
\newblock Vertex operators and 2-representations of quantum affine algebras.
\newblock \arxiv{1112.6189v2}.

\bibitem[CL12]{CL12}
S.~Cautis and A.~Licata.
\newblock Heisenberg categorification and {H}ilbert schemes.
\newblock {\em Duke Math. J.}, 161(13):2469--2547, 2012.
\newblock \doi{10.1215/00127094-1812726}.

\bibitem[CLLS]{CLLS15}
S.~Cautis, A.~Lauda, A.~Licata, and J.~Sussan.
\newblock {$W$}-algebras from {H}eisenberg categories.
\newblock \arxiv{1501.00589v1}.

\bibitem[CLS14]{CLS14}
S.~Cautis, A.~Licata, and J.~Sussan.
\newblock Braid group actions via categorified {H}eisenberg complexes.
\newblock {\em Compos. Math.}, 150(1):105--142, 2014.
\newblock \doi{10.1112/S0010437X13007367}.

\bibitem[CS15]{CS14}
S.~Cautis and J.~Sussan.
\newblock On a categorical {B}oson-{F}ermion correspondence.
\newblock {\em Comm. Math. Phys.}, 336(2):649--669, 2015.
\newblock \doi{10.1007/s00220-015-2310-3}.

\bibitem[FJW00]{FJW00}
I.~B. Frenkel, N.~Jing, and W.~Wang.
\newblock Vertex representations via finite groups and the {M}c{K}ay
  correspondence.
\newblock {\em Internat. Math. Res. Notices}, (4):195--222, 2000.
\newblock \doi{10.1155/S107379280000012X}.

\bibitem[FJW02]{FJW02}
I.~B. Frenkel, N.~Jing, and W.~Wang.
\newblock Twisted vertex representations via spin groups and the {M}c{K}ay
  correspondence.
\newblock {\em Duke Math. J.}, 111(1):51--96, 2002.
\newblock \doi{10.1215/S0012-7094-02-11112-0}.

\bibitem[Gei77]{Gei77}
L.~Geissinger.
\newblock Hopf algebras of symmetric functions and class functions.
\newblock In {\em Combinatoire et repr\'esentation du groupe sym\'etrique
  ({A}ctes {T}able {R}onde {C}.{N}.{R}.{S}., {U}niv. {L}ouis-{P}asteur
  {S}trasbourg, {S}trasbourg, 1976)}, pages 168--181. Lecture Notes in Math.,
  Vol. 579. Springer, Berlin, 1977.

\bibitem[HK01]{HK01}
R.~S. Huerfano and M.~Khovanov.
\newblock A category for the adjoint representation.
\newblock {\em J. Algebra}, 246(2):514--542, 2001.
\newblock \doi{10.1006/jabr.2001.8962}.

\bibitem[HS]{HS15}
D.~Hill and J.~Sussan.
\newblock A categorification of twisted {H}eisenberg algebras.
\newblock \arxiv{1501.00283v1}.

\bibitem[J{\'o}z89]{Joz89}
T.~J{\'o}zefiak.
\newblock Characters of projective representations of symmetric groups.
\newblock {\em Exposition. Math.}, 7(3):193--247, 1989.

\bibitem[JW02]{JW02}
N.~Jing and W.~Wang.
\newblock Twisted vertex representations and spin characters.
\newblock {\em Math. Z.}, 239(4):715--746, 2002.
\newblock \doi{10.1007/s002090100340}.

\bibitem[Kho10]{Kho10}
M.~Khovanov.
\newblock Categorifications from planar diagrammatics.
\newblock {\em Jpn. J. Math.}, 5(2):153--181, 2010.
\newblock \doi{10.1007/s11537-010-0925-x}.

\bibitem[Kho14]{Kho14}
M.~Khovanov.
\newblock Heisenberg algebra and a graphical calculus.
\newblock {\em Fund. Math.}, 225(1):169--210, 2014.
\newblock \doi{10.4064/fm225-1-8}.

\bibitem[Kle05]{Kle05}
A.~Kleshchev.
\newblock {\em Linear and projective representations of symmetric groups},
  volume 163 of {\em Cambridge Tracts in Mathematics}.
\newblock Cambridge University Press, Cambridge, 2005.
\newblock \doi{10.1017/CBO9780511542800}.

\bibitem[Kru]{Kru15}
A.~Krug.
\newblock Symmetric quotient stacks and {H}eisenberg actions.
\newblock \arxiv{1501.07253v1}.

\bibitem[LS12]{LS12}
A.~Licata and A.~Savage.
\newblock A survey of {H}eisenberg categorification via graphical calculus.
\newblock {\em Bull. Inst. Math. Acad. Sin. (N.S.)}, 7(2):291--321, 2012.
\newblock
  \href{http://web.math.sinica.edu.tw/bulletin_ns/20122/2012203.pdf}{\tt
  \nolinkurl{http://web.math.sinica.edu.tw/bulletin_ns/20122/2012203.pdf}}.

\bibitem[LS13]{LS13}
A.~Licata and A.~Savage.
\newblock Hecke algebras, finite general linear groups, and {H}eisenberg
  categorification.
\newblock {\em Quantum Topol.}, 4(2):125--185, 2013.
\newblock \doi{10.4171/QT/37}.

\bibitem[Mac80]{Mac80}
I.~G. Macdonald.
\newblock Polynomial functors and wreath products.
\newblock {\em J. Pure Appl. Algebra}, 18(2):173--204, 1980.
\newblock \doi{10.1016/0022-4049(80)90128-0}.

\bibitem[Mac95]{Mac95}
I.~G. Macdonald.
\newblock {\em Symmetric functions and {H}all polynomials}.
\newblock Oxford Mathematical Monographs. The Clarendon Press, Oxford
  University Press, New York, second edition, 1995.
\newblock With contributions by A. Zelevinsky, Oxford Science Publications.

\bibitem[PS16]{PS15}
J.~Pike and A.~Savage.
\newblock Twisted {F}robenius extensions of graded superrings.
\newblock {\em Algebr. Represent. Theory}, 19(1):113--133, 2016.
\newblock \doi{10.1007/s10468-015-9565-4}.

\bibitem[RR03]{RR03}
A.~Ram and J.~Ramagge.
\newblock Affine {H}ecke algebras, cyclotomic {H}ecke algebras and {C}lifford
  theory.
\newblock In {\em A tribute to {C}. {S}. {S}eshadri ({C}hennai, 2002)}, Trends
  Math., pages 428--466. Birkh\"auser, Basel, 2003.

\bibitem[RS15a]{RS15a}
D.~Rosso and A.~Savage.
\newblock Towers of graded superalgebras categorify the twisted {H}eisenberg
  double.
\newblock {\em J. Pure Appl. Algebra}, 219(11):5040--5067, 2015.
\newblock \doi{10.1016/j.jpaa.2015.03.016}.

\bibitem[RS15b]{RS15b}
D.~Rosso and A.~Savage.
\newblock Twisted {H}eisenberg {D}oubles.
\newblock {\em Comm. Math. Phys.}, 337(3):1053--1076, 2015.
\newblock \doi{10.1007/s00220-015-2330-z}.

\bibitem[Ser99]{Ser99}
A.~Sergeev.
\newblock The {H}owe duality and the projective representations of symmetric
  groups.
\newblock {\em Represent. Theory}, 3:416--434 (electronic), 1999.
\newblock \doi{10.1090/S1088-4165-99-00085-0}.

\bibitem[SY15]{SY15}
A.~Savage and O.~Yacobi.
\newblock Categorification and {H}eisenberg doubles arising from towers of
  algebras.
\newblock {\em J. Combin. Theory Ser. A}, 129:19--56, 2015.
\newblock \doi{10.1016/j.jcta.2014.09.002}.

\bibitem[Web]{Web13}
B.~Webster.
\newblock Knot invariants and higher representation theory.
\newblock \arxiv{1309.3796v2}.

\bibitem[WW12]{WW12}
J.~Wan and W.~Wang.
\newblock Lectures on spin representation theory of symmetric groups.
\newblock {\em Bull. Inst. Math. Acad. Sin. (N.S.)}, 7(1):91--164, 2012.
\newblock
  \href{http://web.math.sinica.edu.tw/bulletin_ns/20121/2012104.pdf}{\tt
  \nolinkurl{http://web.math.sinica.edu.tw/bulletin_ns/20121/2012104.pdf}}.

\end{thebibliography}

\end{document}